\documentclass[12pt,oneside,english,reqno]{amsart}
\usepackage[T1]{fontenc}
\usepackage[utf8]{inputenc}
\usepackage{geometry}
\usepackage{mathrsfs}
\usepackage{amstext}
\usepackage{amsthm}
\usepackage{amssymb}
\usepackage{color}
\usepackage[dvipsnames]{xcolor}
\usepackage[normalem]{ulem}
\usepackage[numbers]{natbib}

\usepackage{babel}

\usepackage[colorlinks,linkcolor=blue, urlcolor  = blue,citecolor=blue,hypertexnames=false, breaklinks]{hyperref} 
\usepackage{dsfont}
\usepackage{enumerate}
\usepackage{verbatim} 
\usepackage{stmaryrd}
\usepackage{enumitem}
\setlist[enumerate,1]{wide, label=\textnormal{\textbf{(\roman*)}}, labelindent=0pt, itemsep=.1in}
\usepackage{graphicx}
\usepackage{todonotes}
\usepackage{xcolor}

\usepackage{shuffle}

\usepackage[font={scriptsize}]{caption}
\usepackage{subcaption}

\makeatletter
\numberwithin{equation}{section}
\newtheorem{theorem}{Theorem}[section]

\newtheorem{corollary}[theorem]{Corollary}

\newtheorem{definition}[theorem]{Definition}

\newtheorem{hypothesis}[theorem]{Hypothesis}
\newtheorem{lemma}[theorem]{Lemma}
\newtheorem{notation}[theorem]{Notation}

\newtheorem{proposition}[theorem]{Proposition}

\theoremstyle{remark}
\newtheorem{remark}[theorem]{Remark}

\theoremstyle{remark}

\makeatother

\newcommand{\der}{\delta}
\newcommand{\doo}{\delta_{1}}
\newcommand{\dt}{\delta_{2}}

\newcommand{\hone}{\hat{1}}
\newcommand{\htwo}{\hat{2}}
\newcommand{\id}{\mathsf{id}} % feel free to change back if you don't like it; jd

\newcommand{\ott}{[0,T]}
\newcommand{\ou}{[0,1]}

\newcommand{\cb}{\mathcal{B}}
\newcommand{\cac}{\mathcal{C}}
\newcommand{\cd}{\mathcal{D}}

\newcommand{\ci}{\mathcal{I}}
\newcommand{\cj}{\mathcal{J}}

\newcommand{\cp}{\mathcal{P}}

\newcommand{\cs}{\mathcal{S}}

\newcommand{\cz}{\mathcal{Z}}

\newcommand{\EE}{\mathbb{E}}

\newcommand{\NN}{\mathbb{N}}

\newcommand{\RR}{\mathbb{R}}

\newcommand{\XX}{\mathbb{X}}

\newcommand{\bC}{\mathbf{C}}
\newcommand{\bD}{\mathbf{D}}
\newcommand{\bE}{\mathbf{E}}

\newcommand{\bR}{\mathbf{R}}
\newcommand{\br}{\mathbf{r}}

\newcommand{\bx}{\mathbf{x}}

\newcommand{\bz}{\mathbf{z}}

\newcommand{\al}{\alpha}

\newcommand{\ga}{\gamma}

\newcommand{\ka}{\kappa}

\newcommand{\laa}{\Lambda}

\newcommand{\si}{\sigma}

\newcommand{\vp}{\varphi}

\newcommand{\lp}{\left(}
\newcommand{\rp}{\right)}
\newcommand{\lc}{\left[}
\newcommand{\rc}{\right]}
\newcommand{\lcl}{\left\{}
\newcommand{\rcl}{\right\}}

\definecolor{dg}{rgb}{0, 0.5, 0}
\definecolor{dp}{rgb}{0.50, 0, 0.40}

\newcommand{\eps}{\varepsilon}

\begin{document}

%%%%%%%%%%%%%%%%%%%%%%%%%%
%%%%%%%%%%%%%%%%%%%%%%%%%%

\title[Rough sheets]{On It\^o-Stratonovich formula for rough sheets}

%%%%%%%%%%%%%%%%%%%%%%%%%%
%%%%%%%%%%%%%%%%%%%%%%%%%%

\author[Y. Hakiki \and S. Tindel]{Youssef Hakiki \and Samy Tindel}

\address{Youssef Hakiki: Department of Mathematics, 
Purdue University, USA}
\email{yhakiki@purdue.edu}

\address{Samy Tindel: Department of Mathematics, 
Purdue University, USA}
\email{stindel@purdue.edu}
\urladdr{https://www.math.purdue.edu/~stindel/}

\date{\today}

%%%%%%%%%%%%%%%%%%%%%%%%%%
%%%%%%%%%%%%%%%%%%%%%%%%%%

\begin{abstract}
In this paper, we explore a new strategy towards an It\^o-Stratonovich type formula for rough sheets. Historically, planar integration for irregular paths has been notoriously cumbersome. The emergence of ``mixed differential'' terms in 2D leads to overlapping iterated integrals, which previously required the construction of exhaustive combinatorial structures. As an example of this kind of structure, let us mention the massive 36-element planar signature introduced by K. Chouk and M. Gubinelli (Rough sheets. https://arxiv.org/pdf/1406.7748, $2014$). In this work we propose a simplified setting for rough calculus in the plane, which relies on elementary Taylor expansions in a more fundamental way. We claim that this simple trick allows us to significantly streamline the complexity of planar algebraic integration. We illustrate this methodology for a specialized \emph{Rough-Young} framework: we extend the classical planar change-of-variable formula to paths possessing an asymmetric H\"older regularity: $\gamma_1 > 1/3$ in the first direction and $\gamma_2 > 1/2$ in the second direction. Relying on a structured controlled path expansions, we rigorously minimize the set of iterated integrals needed in the signature, and express the resulting planar change-of-variable formula as the explicit limit of Riemann sums. 
\end{abstract}



%%%%%%%%%%%%%%%%%%%%%%%%%%%%%%%%%%%%%%
%%%%%%%%%%%%%%%%%%%%%%%%%%%%%%%%%%%%%%

\maketitle

\tableofcontents

%%%%%%%%%%%%%%%%%%%%%%%%%%%%%%%%%%%%%%
%%%%%%%%%%%%%%%%%%%%%%%%%%%%%%%%%%%%%%

\section{Introduction}\label{sec:intro}
For processes indexed by a one-dimensional parameter, Lyons' theory of rough paths (see~\cite{Friz2010,Friz2014}) provides a highly effective framework for understanding complex maps, such as the stochastic integral or the Itô map. This type of method has been instrumental in simplifying proofs and establishing new results across stochastic analysis, enabling pathwise solutions of SDEs, the differentiation of the Itô map, and the extension of calculus to complex noises like fractional Brownian motion. In this 1D setting, Gubinelli introduced the notion of a controlled path to abstract the basic structures necessary for integrating rough signals, and highlighting the existence of a \emph{sewing map} that serves as the core of the integration process.  

Extending these robust 1D structures to processes indexed by the plane (or higher-order objects) is notoriously a cumbersome topic. In order to get an idea of this fact, let us start from the simplest situation of a smooth function $x$ indexed by $[0,1]^{2}$ and a regular function $\varphi\in C^{2}(\RR)$. Then some elementary computations show that 
\begin{equation}\label{eq:simple-change-vb}
[\der \vp(x)]_{s_1s_2;t_1t_2}
=\int_{[s_1,s_2]\times[t_1,t_2]} \vp^{(1)}(x_{u;v}) \, d_{uv}x_{u;v}
+ \int_{[s_1,s_2]\times[t_1,t_2]} \vp^{(2)}(x_{u;v}) \, d_{u}x_{u;v} d_{v}x_{u;v},
\end{equation}
for all $0\le s_1<s_2\le 1$ and $0\le t_1<t_2\le 1$, where we have set $[\der y]_{s_1s_2;t_1t_2}$ for the planar increment of $y$ in the rectangle $[s_1,s_2]\times[t_1,t_2]$, namely
\begin{equation}\label{eq:def-planar-increments-intro}
[\der y]_{s_1s_2;t_1t_2} 
\equiv 
y_{s_2;t_2} - y_{s_1;t_2} - y_{s_2;t_1} + y_{s_1;t_1},
\end{equation}
and where $d_{u}x$ is a shorthand for $(\partial x/\partial u) du$.
This simple formula already exhibits the extra term $\int \vp^{(2)}(x_{u;v}) \, d_{u}x \, d_{v}x$ with respect to integration in $\RR$, and the mixed differential term $d_{u}x \, d_{v}x$ is one of the main source of complications when one tries to extend \eqref{eq:simple-change-vb} to more complex situations. 

Before detailing our approach, we emphasize that the study of two-dimensional fields and their signatures has recently seen 
substantial developments. We distinguish three main lines of investigations:
\begin{enumerate}[label=\textnormal{\textbf{(\alph*)}}]
\item
In relation to possible applications in data science and image processing, a proper notion of 2D signatures for images has been  developed and applied as a low-dimensional feature set for texture classification \cite{ZLT22, DEHT25}. Concurrently, discrete two-parameter extensions of iterated-sums signatures have been tightly linked to the theory of quasisymmetric functions \cite{DS26}. Deep topological and higher-algebraic properties of these multiparameter objects have also been uncovered; for instance, generalized mapping space signatures have been constructed by adapting Chen's cochain constructions to cubical mapping spaces \cite{GLNO22}. Furthermore, the laws of random surfaces have been elegantly framed in~\cite{LO24}, relying on algebraic and topological arguments. Those last references mostly use the algebra of Jacobians, as opposed to the  rectangular increments we wish to handle.

\item
From a pure geometric measure theory perspective, robust notions of integrating nonsmooth rough differential 2-forms over planes have been formalized, bridging Young and Z\"ust integration with It\^o and Stratonovich-type sums evaluated over two-dimensional simplexes \cite{ST21, AST24}. 

\item
Within the realm of stochastic analysis, two-parameter rough differential equations have recently been developed to study signature and Schwinger-Dyson kernels \cite{cass2026}, as well as to establish H\"ormander's theorem for semilinear SPDEs \cite{HG2018}. However, we stress that these particular frameworks primarily deal with a very specific, factorized type of dynamic of the form $d_{12}x_{u,v}= d_{\hone\htwo}x_{u,v}=dX_u dY_v$, where the planar driving noise is built from the tensor product of two independent one-parameter rough paths $X$ and $Y$.
We should also note the interesting Malliavin calculus approaches of~\cite{tudor2003ito,tudor2006ito,reveillac2012hermite}. However, this type of method does not give raise to pathwise integrals and relies on Skorohod integration, which is sometimes seen as non physical (see~\cite{CT} for Skorohod-Stratonovich corrections).
\end{enumerate}

Despite the progress reported above, the analytic challenge of rigorously solving \eqref{eq:simple-change-vb} for highly irregular stochastic fields remains daunting. 
Among the works in this direction, the foundational paper of Chouk and Gubinelli~\cite{CG} is our primary source of inspiration. Before describing it in more detail, let us set up some notation that will be used throughout the paper. Following the convention of~\cite{CG,CT}, we carefully separate the two coordinate directions, labeling them direction~1 and direction~2. The evaluation of a function $f:[0,1]^{2k}\to\RR$ is written $f_{s_{1}\cdots s_{k};\,t_{1}\cdots t_{k}}$. We also distinguish two types of differential elements in~\eqref{eq:simple-change-vb}: the bidirectional differential $d_{12}x \equiv d_{uv}x$, and the product differential $d_{1}x\,d_{2}x$, which we further shorten to $d_{\hone\htwo}x$. Setting $y=\vp(x)$ and $y^{(j)}=\vp^{(j)}(x)$ for $j\ge 1$, the formula~\eqref{eq:simple-change-vb} takes the compact form
\begin{equation}\label{a1}
\der y = \int_{1}\int_{2} y^{(1)}\,d_{12}x + \int_{1}\int_{2} y^{(2)}\,d_{\hone\htwo}x,
\end{equation}
where $\delta=\delta_{1}\delta_{2}$ is the rectangular increment operator (with $\delta_{1}$ and $\delta_{2}$ defined precisely in Section~\ref{sec:planar-increments}). We further denote by $\cp_{k,l}(V)$ the space of functions with $k$ variables in direction~1 and $l$ variables in direction~2 satisfying vanishing conditions on diagonals; the most important instance is $\cp_{2,2}(V)$, on which we introduce H\"older norms. Namely, given $\ga_{1},\ga_{2}>0$ and $f\in\cp_{2,2}(V)$, we set
\begin{equation*}
\|f\|_{\ga_{1};\,\ga_{2}} = \sup\lcl \frac{|f_{s_{1}s_{2};\,t_{1}t_{2}}|}{|s_{2}-s_{1}|^{\ga_{1}}|t_{2}-t_{1}|^{\ga_{2}}} \,;\, s_{1},s_{2},t_{1},t_{2}\in\ott \rcl,
\end{equation*}
and we denote by $\cp_{2,2}^{\ga_{1},\ga_{2}}(V)$ the subspace of $\cp_{2,2}(V)$ on which this norm is finite.

With this set of preliminary notation in hand, the approach of~\cite{CG} rests on two complementary principles which can be roughly summarized as follows:
\begin{enumerate}[label=\textnormal{\textbf{(\arabic*)}}]
\item The 2D sewing map (see Proposition~\ref{proposition:planar Lambda} below) is understood as a tensorization of the classical 1D sewing map (Proposition~\ref{prop:Lambda}), applied iteratively in each coordinate direction.
\item In order to invoke the 2D sewing map, rectangular increments such as $[\der\vp(x)]_{s_{1}s_{2};\,t_{1}t_{2}}$ in~\eqref{eq:simple-change-vb} are first systematically dissected through~\eqref{eq:simple-change-vb}, until the resulting remainder falls within the domain of the sewing lemma.
\end{enumerate}
While this global strategy is conceptually easy to summarize, its execution meets several nontrivial obstacles. The article~\cite{CG} treats fields $x$ with $\der x\in\cp_{2,2}^{\ga_{1},\ga_{2}}$, $\ga_{1},\ga_{2}>1/3$, and encounters the following three sources of difficulty:
\begin{enumerate}[label=\textnormal{\textbf{(\alph*)}}]
\item Giving separate rigorous meaning to the two integral terms in~\eqref{a1}, involving $d_{12}x$ and $d_{\hone\htwo}x$ respectively,  requires distinct analytical treatments.
\item The Taylor expansion of rectangular increments generates boundary terms, that is, integrals over subdomains $[0, s_{1}]\times[t_{1},t_{2}]$ and $[s_{1},s_{2}]\times [0,t_{1}]$ of the rectangle, which must be tracked carefully at each step.
\item\label{it:overlap}
When expanding, overlapping iterated integrals in directions~1 and~2 inevitably appear. Otherwise stated, one often has to start a new iterated integral in direction 2 while still integrating in direction 1 (or the other way around).
\end{enumerate}

The last obstacle \ref{it:overlap} above is in fact the most serious. To describe the objects involved, we introduce a compact index notation for iterated integrals of $x$. Each element $\bx^{e}$ encodes its integration pattern via a superscript: indices $1$ and $\hone$ record a step in direction~1 against $d_{1}x$ or $d_{\hone}x$ respectively; indices $2$ and $\htwo$ similarly record a step in direction~2. A dot~$\cdot$ marks the start of a new independent integration factor. And $0$ signals no integration in the given direction at that step. As a first illustration,
\begin{equation}\label{a2}
\bx^{1\hone;0\cdot\htwo} = \int_{1} d_{1}x \int_{2} d_{\hone\htwo}x,
\quad\text{that is}\quad
\bx^{1\hone;0\cdot\htwo}_{s_{1}s_{2};t_{1}t_{2}} =
\int_{s_{1}}^{s_{2}}\int_{t_{1}}^{t_{2}}
\lp \int_{s_{1}}^{\si_{2}} d_{1}x_{\si_{1};t_{1}} \rp d_{1}x_{\si_{2};\tau_{1}} d_{2}x_{\si_{2};\tau_{1}},
\end{equation}
which belongs to $\cp_{2,2}$. As a second example, $\bx^{1\cdot\hone;0\htwo}\in\cp_{3,2}$ is defined by
\begin{equation}\label{a3}
\bx^{1\cdot\hone;0\htwo} = \int_{1} d_{1}x \int_{1}\int_{2} d_{\hone\htwo}x,
\quad\text{that is}\quad
\bx^{1\cdot\hone;0\htwo}_{s_{1}s_{2}s_{3};t_{1}t_{2}} =
\int_{s_{1}}^{s_{2}} d_{1}x_{\si_{1};t_{1}}
\int_{s_{2}}^{s_{3}}\int_{t_{1}}^{t_{2}}
 d_{1}x_{\si_{2};\tau_{1}} d_{2}x_{\si_{2};\tau_{1}}.
\end{equation}
Then notice in~\eqref{a2}-\eqref{a3} that the break in integration can occur at different steps in directions~1 and~2: these are precisely the overlapping iterated integrals generated during expansion. 
Those integrals render the rough analysis substantially more complex.
In~\cite{CG}, they are resolved by a technique called \emph{splitting}, which decomposes them into pieces amenable to the 2D sewing map, but at the cost of much longer expansions. In the case $\ga_{1},\ga_{2}>1/3$, this forces 36 distinct elements (according to our count) in the signature $\XX$ of $x$ so that one can define a rough integral. This large number has to be compared  to just 2 elements in the corresponding 1D setting for $\ga>1/3$. Furthermore, it appears difficult to extend the approach in~\cite{CG} to the regime $\ga_{1},\ga_{2}\le 1/3$.

Let us now turn to our own strategy. It differs from~\cite{CG} in one key respect: we invest considerably more effort in Taylor-expanding the rectangular increments \emph{before} invoking the sewing maps. We believe that this simple trick allows a more efficient dissection of the planar integrals. 
For sake of conciseness, we will deal with Young regularity $\ga_{2}>1/2$ in direction~2. In this simplified context we hope the reader can see what Taylor expansions can bring to the picture, and how they avoid splitting operations (see Lemma~\ref{L:2D-Taylor-expansion1} and~\eqref{E:delta-y-expansion} below for a first concrete instance).
Then the overall outcome of our approach is fourfold:
\begin{enumerate}[label=\textnormal{\textbf{(\roman*)}}]
\item We establish an exact change-of-variable formula for $y=\vp(x)$, expressing $\der y$ as the limit of modified Riemann sums (see \eqref{Eq:riemann-delta(y)} in Proposition~\ref{prop:ch-var} below). The geometric rough sheet $\XX$ required for this formula consists solely of the 14 elements listed in Table~\ref{table:rp} and Table~\ref{table:rs2}, a drastic reduction from the 36 elements of~\cite{CG}.
\item We introduce a notion of controlled process (Definition~\ref{def:controlled-process}) adapted to the Rough-Young setting, which appears to be a natural and robust concept for planar integration. In the context of a rough-Young regularity setting, controlled processes form  a reasonably compact structure.
\item Although we restrict the analysis to the Rough-Young regime $\ga_{1}>1/3$, $\ga_{2}>1/2$ for conciseness, the same strategy should extend to the fully rough case $\ga_{1},\ga_{2}>1/3$. The reason is that the mechanism at the heart of our method (Taylor-expanding the rectangular increments before invoking the sewing maps) is symmetric in the two directions and does not rely on $\ga_{2}>1/2$. In the present setting, direction~2 enters only through a first-order Young expansion (see Lemma~\ref{L:2D-Taylor-expansion1} and~\eqref{E:delta-y-expansion}), while direction~1 is handled by the full rough apparatus of second-order controlled processes together with the sewing map $\Lambda_{1}$ acting at order $3\ga_{1}>1$. To reach $\ga_{1},\ga_{2}>1/3$ one would simply carry the direction-2 expansion to the same depth as the direction-1 one, so that both variables are treated as second-order controlled processes. The overlapping iterated integrals discussed around~\eqref{a2}-\eqref{a3} would then be handled, exactly as here, by pushing the Taylor expansion one step further rather than by the splitting technique of~\cite{CG}. The price to pay is heavier combinatorial bookkeeping and a moderately larger signature $\XX$, but we don't expect new conceptual obstructions along the road. A complete treatment of the fully rough regime, including the precise description of the enlarged signature, is deferred to subsequent publications.
\item\label{it:appli-fbs}
As a primary stochastic application, we show that the fractional Brownian sheet with Hurst parameters $H_{1}>1/3$ and $H_{2}>1/2$ can be lifted to a geometric rough sheet satisfying Hypotheses~\ref{hyp:rs1}--\ref{hyp:rs3} below, via multiparameter Garsia-Rodemich-Rumsey inequalities (see Proposition~\ref{prop:iterated-xN} for the precise statement). This places our It\^o-Stratonovich formula in a concrete probabilistic setting.
\end{enumerate}

To summarize, this paper develops a self-contained rough calculus in the plane under the Rough-Young regularity assumption $\ga_{1}>1/3$, $\ga_{2}>1/2$. Thanks to deeper Taylor expansions of rectangular increments, we are allowed to keep the signature small and the proofs transparent, at the cost of more careful bookkeeping in the expansion steps. The fBs application in Section~\ref{sec:fbs} shows that the framework is not merely theoretical: the rough sheet $\XX$ can be concretely constructed for a natural Gaussian field.

The paper is organized as follows. Section~\ref{sec:one-dim} recalls the algebraic integration framework in dimension~1: increments, products of increments, iterated integrals, and the 1D sewing map (Proposition~\ref{prop:Lambda}). Section~\ref{sec:alg-intg-plane} develops the 2D counterpart of this framework: planar increments and the 2D sewing map (Proposition~\ref{proposition:planar Lambda}), together with the signature elements listed in Table~\ref{table:rp}. Section~\ref{sec:dissection} is devoted to the dissection of planar integrals: we introduce the controlled path notion (Definition~\ref{def:controlled-process}), construct the rough-Young integrals $z^{1}$ and $z^{2}$ (Theorem~\ref{th:main-yr-integral}), and establish the change-of-variable formula (Proposition~\ref{prop:ch-var}). Section~\ref{sec:fbs} is devoted to the application to the fractional Brownian sheet: after establishing a generalized multiparameter GRR inequality, we prove that the fBs admits a lift to a geometric rough sheet satisfying Hypotheses~\ref{hyp:rs1}--\ref{hyp:rs3} (Proposition~\ref{prop:iterated-xN}).

\section{Algebraic integration in dimension 1}
\label{sec:one-dim}

We recall here the minimal amount of notation concerning algebraic integration theory in $\RR$, in order to prepare the ground for further developments in the plane. We refer to ~\cite{Gu,GT} for a more detailed introduction.

\subsection{Increments}\label{incr}

The integration theory developed below rests on the algebraic framework of
increments, originally introduced in~\cite{Gu,GT}, which we briefly recall.
For a vector space $V$ and an integer $k\ge 1$, define $\cac_{k}(V)$ as the
set of functions $g:[0,1]^{k}\to V$ that vanish whenever two consecutive
arguments coincide:
\begin{equation*}
  g_{t_{1}\cdots t_{k}}=0 \quad\text{whenever } t_{i}=t_{i+1}
  \text{ for some } i\le k-1.
\end{equation*}
Elements of $\cac_{k}(V)$ are called \emph{$(k-1)$-increments}, and we set
$\cac_{*}(V)=\bigcup_{k\ge 1}\cac_{k}(V)$. The algebraic structure on
$\cac_{*}(V)$ is governed by the coboundary operator
\begin{equation}
  \label{eq:coboundary}
  \delta:\cac_{k}(V)\to\cac_{k+1}(V),\qquad
  (\delta g)_{t_{1}\cdots t_{k+1}}
  =\sum_{i=1}^{k+1}(-1)^{k-i}\,g_{t_{1}\cdots\hat{t}_{i}\cdots t_{k+1}},
\end{equation}
where $\hat{t}_{i}$ indicates that the $i$-th argument is omitted. A direct
computation confirms that $\delta^{2}=0$, i.e.\ the composition
$\delta:\cac_{k}(V)\to\cac_{k+1}(V)$ followed by
$\delta:\cac_{k+1}(V)\to\cac_{k+2}(V)$ is the zero map. We accordingly set
\begin{equation*}
  \cz\cac_{k}(V):=\ker\delta\cap\cac_{k}(V),\qquad
  \cb\cac_{k}(V):=\mathrm{Im}\,\delta\cap\cac_{k}(V).
\end{equation*}

\vspace{0.3cm}

The two instances of~\eqref{eq:coboundary} that appear most frequently in
this paper are $g\in\cac_{1}(V)$ and $h\in\cac_{2}(V)$. In this case, for
$s,u,t\in\ott$ we get:
\begin{equation}
  \label{eq:simple_application}
  \delta g_{st}=g_{t}-g_{s},\qquad
  \delta h_{sut}=h_{st}-h_{su}-h_{ut}.
\end{equation}
Since $\delta^{2}=0$, every exact increment belongs to the kernel of
$\delta$. In fact the converse also holds: $\cz\cac_{k+1}(V)=\cb\cac_{k}(V)$
for all $k\ge 1$. The special case $k=1$ is worth stating explicitly:
\begin{lemma}\label{exd}
  Let $h\in\cz\cac_{2}(V)$. Then there exists a (non-unique)
  $f\in\cac_{1}(V)$ such that $h=\delta f$.
\end{lemma}
\noindent
In other words, $\delta|_{\cac_{2}(V)}$ measures the obstruction for a
$1$-increment to be a genuine finite difference: $h_{st}=f_{t}-f_{s}$ for
some $f$ if and only if $\delta h=0$.

\vspace{0.3cm}

Our analysis will mainly involve $k$-increments with $k\le 2$, equipped with
H\"older-type norms. That is, for $f\in\cac_{2}(V)$, set
\begin{equation}\label{eq:def-norm-C2}
  \|f\|_{\mu}
  =\sup_{s,t\in\ott}\frac{|f_{st}|}{|t-s|^{\mu}},
  \qquad
  \cac_{2}^{\mu}(V)
  :=\bigl\{f\in\cac_{2}(V)\,:\,\|f\|_{\mu}<\infty\bigr\}.
\end{equation}
The H\"older norm on $\cac_{1}(V)$ is defined by transport through $\delta$:
for a continuous $g\in\cac_{1}(V)$,
\begin{equation}\label{def:hnorm-c1}
  \|g\|_{\mu}:=\|\delta g\|_{\mu},
\end{equation}
and $g\in\cac_{1}^{\mu}(V)$ iff this quantity is finite. Note that
$\|\cdot\|_{\mu}$ is only a semi-norm on $\cac_{1}(V)$. For
$h\in\cac_{3}(V)$, the analogous definition reads
\begin{equation}\label{eq:normOCC2}
  \|h\|_{\mu}
  =\sup_{s,u,t\in\ott}\frac{|h_{sut}|}{|t-s|^{\mu}},
  \qquad
  \cac_{3}^{\mu}(V)
  :=\bigl\{h\in\cac_{3}(V)\,:\,\|h\|_{\mu}<\infty\bigr\}.
\end{equation}
We set $\cac_{3}^{1+}(V):=\bigcup_{\mu>1}\cac_{3}^{\mu}(V)$, along with the
corresponding closed-increment subspace $\cz\cac_{3}^{1+}(V)$.
With those notions in hand, the following inversion result for $\delta$ is the cornerstone of the
pathwise integration theory we will use; a concise proof can be found in~\cite{GT}.
\begin{proposition}[The $\Lambda$-map]\label{prop:Lambda}
  There exists a unique linear map $\Lambda:\cz\cac_{3}^{1+}(V)
  \to\cac_{2}^{1+}(V)$ such that
  \begin{equation*}
    \delta\Lambda=\id_{\cz\cac_{3}^{1+}(V)}
    \qquad\text{and}\qquad
    \Lambda\delta=\id_{\cac_{2}^{1+}(V)}.
  \end{equation*}
  Equivalently, every $h\in\cz\cac_{3}^{1+}(V)$ has a unique preimage
  $\Lambda h\in\cac_{2}^{1+}(V)$ under $\delta$. Furthermore, for any
  $\mu>1$, the map $\Lambda$ is continuous from $\cz\cac_{3}^{\mu}(V)$ to
  $\cac_{2}^{\mu}(V)$, with the explicit bound
  \begin{equation}\label{ineqla}
    \|\Lambda h\|_{\mu}\le\frac{1}{2^{\mu}-2}\|h\|_{\mu},
    \qquad h\in\cz\cac_{3}^{\mu}(V).
  \end{equation}
\end{proposition}

\noindent
The following corollary spells out how Proposition~\ref{prop:Lambda} connects
the algebraic framework to Riemann-type integration.
\begin{corollary}\label{cor:integration}
  For any $1$-increment $g\in\cac_{2}(V)$ such that $\delta g\in\cac_{3}^{1+}$,
  set $\delta f:=(\id-\Lambda\delta)g$. Then the increments of $f$ can be written as limits of Riemann type sums:
  \begin{equation*}
    \delta f_{st}
    =\lim_{|\Pi_{st}|\to 0}\sum_{i=0}^{n-1}g_{t_{i}\,t_{i+1}},
  \end{equation*}
  where the limit is over any partition $\Pi_{st}=\{t_{0}=s,\dots,t_{n}=t\}$
  of $[s,t]$ as its mesh tends to zero. Thus $\delta f$ is the indefinite
  integral of the $1$-increment $g$.
\end{corollary}

\subsection{Products of increments}\label{sec:comp-cac}

Throughout this section we specialize to the scalar case $V=\RR$ and write
$\cac_{k}^{\ga}$ in place of $\cac_{k}^{\ga}(\RR)$. The natural product
of two increments in $(\cac_{*},\delta)$ is obtained by identifying the
last variable of the first factor with the first variable of the second
(see e.g.~\cite{Gu}):
\begin{definition}\label{def:pdt-increments-1d}
  For $g\in\cac_{n}$ and $h\in\cac_{m}$, denote by $gh$ the element of
  $\cac_{n+m-1}$ given by
  \begin{equation}\label{eq:convention-pdt}
    (gh)_{t_{1},\dots,t_{m+n-1}}
    =g_{t_{1},\dots,t_{n}}\,h_{t_{n},\dots,t_{m+n-1}},
    \quad t_{1},\dots,t_{m+n-1}\in\ott.
  \end{equation}
\end{definition}

\noindent
The operator $\delta$ interacts with this product through the following
Leibniz-type rules.
\begin{proposition}\label{prop:difrul}
  Consider the product introduced in Definition~\ref{def:pdt-increments-1d}.
  \begin{enumerate}
  \item If $g,h\in\cac_{1}$, then $gh\in\cac_{1}$ and
    \begin{equation}\label{eq:difrulu-C1-C1}
      \delta(gh)=\delta g\,h+g\,\delta h.
    \end{equation}
  \item If $g\in\cac_{1}$ and $h\in\cac_{2}$, then $gh\in\cac_{2}$ and
    \begin{equation}\label{eq:difrul-C1-C2}
      \delta(gh)=-\delta g\,h+g\,\delta h.
    \end{equation}
  \item If $g\in\cac_{2}$ and $h\in\cac_{1}$, then $gh\in\cac_{2}$ and
    \begin{equation}\label{eq:difrul-C2-C1}
      \delta(gh)=\delta g\,h+g\,\delta h.
    \end{equation}
  \end{enumerate}
\end{proposition}

\begin{proof}
  We verify~\eqref{eq:difrulu-C1-C1}; the remaining cases are equally
  straightforward. For $g,h\in\cac_{1}$ and $s,t\in\ott$,
  \begin{equation*}
    [\delta(gh)]_{st}
    =g_{t}h_{t}-g_{s}h_{s}
    =g_{s}(h_{t}-h_{s})+(g_{t}-g_{s})h_{t}
    =g_{s}\,\delta h_{st}+\delta g_{st}\,h_{t},
  \end{equation*}
  which is~\eqref{eq:difrulu-C1-C1}.
\end{proof}

While Definition~\ref{def:pdt-increments-1d} and the
notation~\eqref{eq:convention-pdt} cover most situations arising in this
paper, some later computations require identifying more than one variable.
We record this generalization for future reference.
\begin{definition}\label{def:k-pdt-increments-1d}
  Let $g\in\cac_{n}$, $h\in\cac_{m}$, and $k\le\min(n,m)$. The increment
  $g*_{k}h\in\cac_{m+n-k}$ is defined by
  \begin{equation}\label{b1}
    (g*_{k}h)_{t_{1},\ldots,t_{m+n-k}}
    =g_{t_{1},\ldots,t_{n}}\,h_{t_{n-k},\ldots,t_{n+m-k}}.
  \end{equation}
\end{definition}

\subsection{Iterated integrals as increments}\label{sec:iterated-inc}

Iterated integrals of smooth functions on $\ott$ are obviously
particular cases of elements of $\cac_2$, which will be of interest
for us. A typical example of this kind of object is given as follows: consider $f^j\in\cac_1^\infty$ for $j=1,\ldots,n$ and $0\le s_1<s_2\le 1$. For $n\ge 1$, we denote by $\cs_n(s_1,s_2)$ the simplex
\begin{equation}\label{eq:def-simplex}
\cs_n(s_1,s_2)=
\lcl
(\si_1,\ldots,\si_n)\in\ott^n ; \, s_1<\si_1<\cdots<\si_n<s_2
\rcl
\end{equation}
and  we set
\begin{equation}\label{eq:def-iterated-intg}
h_{s_1s_2}^{1,\ldots,n} \equiv
\int_{\cs_n(s_1,s_2)} df_{\si_{1}}^{1} \cdots df_{\si_{n}}^{n}
= \int_{s_1}^{s_2} \int_{s_{1}}^{\si_{n-1}} \cdots \int_{s_{1}}^{\si_{2}} df_{\si_{1}}^{1}   \cdots  df_{\si_{n}}^{n}.
\end{equation}

\smallskip

With these conventions in mind, the following relations between multiple integrals and the operator $\der$ will also be useful. The reader is sent to \cite{GT} for its elementary proof.
\begin{proposition}\label{prop:dif-intg-1}
Let $f\in\cac_1^\infty$ and $g\in\cac_1^\infty$.
Then it holds that
\begin{equation*}
\der g = \int dg, \qquad \der\lp \int f\,dg\rp = 0, \qquad \der\lp
\int df\,dg \rp = \der f \, \der g.
\end{equation*}
% and
% \begin{equation}
% \der \lp [ d, \ldots, d ] (f^1,\ldots , f^n)\rp  =
% \sum_{i=1}^{n-1}
%  [ d, \ldots, d ] (f^1,\ldots , f^{j}) \,  [ d, \ldots, d ] (f^{j+1},\ldots , f^n).
% \end{equation}
\end{proposition}
We close this section by recalling the definition of Young's integral thanks to algebraic integration techniques. This will be used in order to obtain our 2$d$ change of variable formula.

\begin{proposition}\label{prop:young}
Let $f\in \cac^{\alpha}_1$ and $g\in \cac_1^{\beta}$ with $\alpha+\beta>1$. Then the Young integral $\int fdg$ is defined as 
\begin{equation}\label{E:young-integral}
\int fdg = f\delta g -\Lambda\bigl(\delta f\delta g\bigl)=\bigl[\id -\Lambda\delta \bigl](f\delta g),
\end{equation}
where $\Lambda$ is the sewing map from Proposition \ref{prop:Lambda}.
\end{proposition}

\section{Algebraic integration in the plane}
\label{sec:alg-intg-plane}

This section is devoted to recall the elements of algebraic integration necessary to define an integral of the form $\int_{[0,1]^{2}} \varphi(x) \, dx$ for a H\"older function $x$ in the plane with H\"older exponent greater than $1/3$. This requires a tensorization of the algebraic structures defined in the previous section, plus some extra tools that we proceed to introduce.

\subsection{Some general notation}
\label{sec:general-notation}
The conventions used in this section were set up in Section~\ref{sec:intro}, to which we refer for motivation. We recall them briefly: the two coordinate directions are labeled direction~1 and direction~2; the evaluation of $f:[0,1]^{2k}\to\RR$ is written $f_{s_{1}\cdots s_{k};\,t_{1}\cdots t_{k}}$; the bidirectional differential is $d_{12}x$ and the product differential $d_{1}x\,d_{2}x$ is abbreviated to $d_{\hone\htwo}x$. We set $y=\vp(x)$ and $y^{(j)}=\vp^{(j)}(x)$ for $j\ge 1$. With these shorthands, the change-of-variable formula~\eqref{eq:simple-change-vb} for a smooth $x:[0,1]^{2}\to\RR$ reads as in~\eqref{a1}:
\begin{equation}\label{eq:strato-intro-notation-d12}
\der y = \int_{1}\int_{2} y^{(1)} \, d_{12}x + \int_{1}\int_{2} y^{(2)} \, d_{\hone\htwo}x,
\end{equation}

\subsection{Planar increments}
\label{sec:planar-increments}

We consider here increments of a variable $s$ (also called direction 1) and a variable $t$ (also called direction 2), with $(s,t)\in [0,1]^2$. For a vector space $V$, we set
\begin{equation}\label{c3}
\cp_{k,l}(V)=
\lcl
f\in\cac([0,1]^k\times[0,1]^l ;\, V) ; \, f_{s_1\cdots s_k; \, t_1\cdots t_l}=0 \mbox{ whenever } s_i=s_{i+1}
\mbox{ or } t_j=t_{j+1}
\rcl.
\end{equation}
In the particular case $V=\RR$, we simply set $\cp_{k,l}(\RR)\equiv \cp_{k,l}$.

\smallskip

Some partial difference operators $\doo$ and $\dt$ with respect to the first and second direction can be defined as in the previous section. Namely, for $f\in \cp_{k,l}(V)$ we set
\begin{equation*}
\delta_1 : \cp_{k,l}(V) \to \cp_{k+1,l}(V), \qquad
\doo g_{s_1 \cdots s_{k+1}; \, t_1\cdots t_l} = \sum_{i=1}^{k+1} (-1)^{k-i}
g_{s_1  \cdots \hat s_i \cdots s_{k+1}; \, t_1\cdots t_l} ,
\end{equation*}
and we define $\dt$ similarly. The planar increment $\der$ is then obtained as $\der=\doo\,\dt$. Notice that for $f\in\cp_{1,1}$ we have
\begin{equation}\label{eq:rect-incr}
\der f_{s_1s_2 ; \, t_1 t_2}
=f_{s_2; \, t_2} - f_{s_2; \, t_1} - f_{s_1; \, t_2} + f_{s_1; \, t_1},
\end{equation}
which is the usual rectangular increment of a function $f$ defined on $\ott^2$ and is consistent with formula \eqref{eq:def-planar-increments-intro}. Let us label the following notation for further use:
\begin{notation}\label{not:BP-ZP}
For $j=1,2$, we set $\cz_j\cp_{k,l}= \cp_{k,l}\cap \ker(\der_j)$ and $\cb_j\cp_{k,l}= \cp_{k,l}\cap {\rm Im}(\der_j)$. We also write $\cz\cp_{k,l}$ for $\cp_{k,l}\cap \ker(\der)$ and $\cb\cp_{k,l}$ for $\cp_{k,l}\cap {\rm Im}(\der)$.
\end{notation}

\smallskip

As in the 1-d case, the H\"older regularity of planar increments is an essential feature of our generalized integration theory. On $\cp_{2,2}(V)$ and $\cp_{3,3}(V)$, it is  measured by a tensorization of the H\"older norms defined at \eqref{eq:def-norm-C2} and \eqref{eq:normOCC2}. Namely, if $f\in\cp_{2,2}(V)$, we set
\begin{equation*}
\|f\|_{\ga_1;\, \ga_2} =
\sup\lcl  \frac{|f_{s_1 s_2; \, t_1 t_2}|}{|s_2-s_1|^{\ga_1} |t_2-t_1|^{\ga_2}} \, ; \, s_1,s_2,t_1,t_2\in\ott\rcl.
\end{equation*}
Related to this type of quantity, we introduce a family of norms for functions defined in $[0,T]^2$.
\begin{definition}\label{def:2D-hnorm-P2P3}
Let $x$ be an increment defined on $[0,T]^2$, with values in a vector space $V$. Then 
\begin{enumerate}
\item 
We denote by $\cp_{2,2}^{\ga_1, \ga_2}(V)$ the space of increments in $\cp_{2,2}(V)$ whose $\|\cdot \|_{\ga_1;\, \ga_2}$ norm is finite.
\item
We set $\cp_{2,2}^{*,\ga_2}=\bigcup_{\ga_1\in (0,1)}\cp_{2,2}^{\ga_1,\ga_2}$.
\item
Let $h\in \cp_{3,3}(V)$, whose domain of definition is restricted to the product of simplex $[\cs_{3}(0,T)]^{2}$ defined in~\eqref{eq:def-simplex}. We say that $h\in \cp_{3,3}^{\ga_1, \ga_2}(V)$ if
\begin{equation*}
\sup\lcl  \frac{|h_{s_1 s_2 s_3; \, t_1 t_2 t_3}|}{|s_3-s_1|^{\ga_1}  |t_3-t_1|^{\ga_2}} 
\, ; \, (s_1,s_2,s_3), \, (t_1,t_2,t_3)\in\cs_{3}(0,T)\rcl < \infty.
\end{equation*}
\end{enumerate}
Similar norms, omitted here for sake of conciseness, can be defined on $\cp_{2,3}(V)$ and $\cp_{3,2}(V)$.
\end{definition} 
\smallskip

For H\"older continuous increments with regularity greater than 1, one gets the following inversion properties, which are a direct consequence of the one dimensional Proposition~\ref{prop:Lambda}:
\begin{proposition}\label{proposition:planar Lambda}
Let $\ga_1,\ga_2>1$, and recall that the spaces $\cb_j\cp$ are defined for $j=1,2$ in Notation~\ref{not:BP-ZP}. Then:

\smallskip

\noindent\emph{(1)}
There exist two maps $\laa_1:\cb_1\cp_{3,3}^{\ga_1,\ga_2}\to\cp_{2,3}^{\ga_1,\ga_2}$ and $\laa_2:\cb_2\cp_{3,3}^{\ga_1,\ga_2}\to\cp_{3,2}^{\ga_1,\ga_2}$ such that $\der_j\laa_j=\id$. These maps satisfy the bound $\|\laa_j(h)\|_{\ga_1,\ga_2} \le c_{\ga_j} \|h\|_{\ga_1,\ga_2}$ for $j=1,2$.

\smallskip

\noindent\emph{(2)}
There exists a map $\laa:\cb\cp_{3,3}^{\ga_1,\ga_2}\to\cp_{2,2}^{\ga_1,\ga_2}$ such that $\der\laa=\id$. This map satisfies the bound $\|\laa(h)\|_{\ga_1,\ga_2} \le c_{\ga_1,\ga_2} \|h\|_{\ga_1,\ga_2}$.
\end{proposition}
We do not include the proof of this proposition for sake of conciseness, and we refer to \cite{CG} for more details. Let us just mention that (as the reader might imagine) we have $\laa=\laa_1\laa_2$. It should also be observed that some 2-dimensional Riemann sums are related to the sewing map $\laa$, echoing Corollary~\ref{cor:integration}:

\begin{proposition}\label{proposition:integ-2d}
Let $g\in\cp_{2,2}$ satisfying the following assumptions:
\begin{equation*}
\doo g \in \cp_{3,2}^{\ga_1,*}, \qquad
\dt g \in \cp_{2,3}^{*,\ga_2}, \qquad
\der g \in \cp_{3,3}^{\ga_1,\ga_2},
\end{equation*}
for $\ga_1,\ga_2>1$, where the spaces $\cp$ are introduced in Definition~\ref{def:2D-hnorm-P2P3} and we recall that $*$ denotes any kind of H\"older regularity. Then there exists $f\in\cp^{1,1}$ such that
\begin{equation*}
\der f = \lc  \id - \laa_1 \doo \rc \lc  \id - \laa_2 \dt \rc g,
\quad\mbox{and}\quad
\lim_{|\pi|\to 0}\sum_{\si_i,\tau_j\in\pi} g_{\si_{i} \si_{i+1} ;  \tau_{j} \tau_{j+1}} = \der f_{s_1 s_2; t_1 t_2},
\end{equation*}
where $\pi$ designates a family of rectangular partitions of $[s_1,s_2]\times[t_1,t_2]$ whose mesh goes to 0.
\end{proposition}

We finish this section on basic notation by giving the notion of product of planar increments. Note that here we specialize the setting to real-valued increments, with a state space $V=\RR$.
\begin{definition}\label{def:2d-ppi}
For  $g\in\cp_{n_1,n_2}$ and $h\in\cp_{m_1,m_2}$, we denote by  $gh$ the element lying in the space $\cp_{n_1+m_1-1,n_2+m_2-1}$ defined by
\begin{equation}\label{c4}
(gh)_{s_1,\ldots s_{n_1+m_1-1}; t_1,\dots,t_{n_2+m_2-1}}=
g_{s_1,\ldots s_{n_1}; t_1,\dots,t_{n_2}} h_{s_{n_1},\ldots s_{n_1+m_1-1}; t_{n_2},\dots,t_{n_2+m_2-1}}.
\end{equation}
\end{definition}

\subsection{Some elements of the signature}
In this section we introduce some iterated integrals of $x$ which will feature in our computations. Note that in the sequel we will start by performing our computations for smooth functions. Otherwise stated, we start from a smooth approximation $x^{n}$ to our path $x$ and we introduce a useful notation for the remainder of the computations:

\begin{notation}\label{not:smooth-path-y-fx}
We shall drop the index $n$ of approximations in $x^{n}$, which means that $x$ will stand for a generic smooth path defined on $\ou^{2}$ whenever the notation is non ambiguous. For a smooth function $\vp:\RR\to\RR$, we also write $y$ for the path $\vp(x)$ and for all $j\ge 1$ we set $y^{(j)}=\vp^{(j)}(x)$.
\end{notation}

With these notations in hand, for a smooth sheet $x$ and $f\in\cac_b^{2}$ it is well known that formula \eqref{eq:simple-change-vb} holds true. Recall that we have written this relation under the following form:
\begin{equation}\label{eq:ito-plane-smooth2}
\der y = \int_{1}\int_{2} y^{(1)} \, d_{12}x + \int_{1}\int_{2} y^{(2)} \, d_{\hone\htwo}x.
\end{equation}
We shall see that this formula still holds true in the limit for $x$, except that the integrals involved in the right hand side of \eqref{eq:ito-plane-smooth2} have to be interpreted in a sense which goes beyond the Riemann-Stieltjes case. Our main task will thus be to obtain a definition of $\int_{1}\int_{2} y^{1} \, d_{12}x$ and $ \int_{1}\int_{2} y^{2} \, d_{\hone\htwo}x$ involving iterated integrals of $x$ and increments of $y$ (or $y^{(j)}$ for $j\ge 1$) only.

\begin{comment}
Let us introduce what will be later interpreted as the first order elements of the planar rough path above $x$:
\begin{notation}\label{not:1st-order-increments-x}
Let $x\in\cp_{1,1}^{\ga_1,\ga_2}$ with $\ga_1,\ga_2>1/2$. We set
\begin{equation*}
\bx^{1;2} = \der x, 
\quad\text{and}\quad
\bx^{\hone;\htwo} = [\id-\laa_{1}\der_{1}] [\id-\laa_{2}\der_{2}] \left( \der_{1}x \, \der_{2}x\right).
\end{equation*}
Notice that for smooth functions we also have
\begin{equation*}
\bx^{1;2} = \int_{1}\int_{2} d_{12}x, 
\quad\text{and}\quad
\bx^{\hone;\htwo} = \int_{1}\int_{2} d_{1}x \, d_{2}x.
\end{equation*}
\end{notation}
\end{comment}

%%%%
As mentioned in the introduction ,the computations in \cite{CG}
are based on lengthy expansions
of double integrals, based
on iterations of \eqref{eq:strato-intro-notation-d12}. This
type of iteration gets extremely
involved due to overlapping
integrals in directions 1 and 2.
In the current contribution we
propose to simplify this approach
in two ways: (1) We
rely more heavily on Taylor type expansions. (2) We consider a restricted
setting, with $x\in \mathcal{P}^{\ga_1,\ga_2}$ where $\ga_1>1/3$ and $\ga_2>1/2$. As a result, we will only need the iterated integrals summarized below.
%%%%%%%%%%%

%In the sequel, we will also assume that the following iterated integrals of $x$ exist:
\begin{hypothesis}\label{hyp:rs1}
The function $x$ is such that $\der x\in\cp_{2,2}^{\ga_{1},\ga_{2}}$ with $\ga_{1}>1/3$ and $\ga_{2}>1/2$ (see \eqref{c3} for the definition of $\mathcal{P}_{1,1}^{\ga_1,\ga_2}$).  Moreover, the rough sheet $\XX$ in Table~\ref{table:rp} can be constructed out of $x$.
\begin{table}[htp]
\caption{Rough sheets above $x$ in the rough-Young case}
\begin{center}
\begin{tabular}{|c|c|c||c|c|c|}
\hline
\emph{Increment} & \emph{Interpretation} & \emph{Regularity} & \emph{Increment} & \emph{Interpretation} & \emph{Regularity} \\
\hline
$\bx^{1;0}$ & $\int_{1} d_{1}x$ & $(\ga_{1},0)$ & 
$\bx^{0;2}$ & $\int_{2} d_{2}x$ & $(0,\ga_{2})$ \\
\hline
$\bx^{11;00}$ & $\int_{1} d_{1}xd_{1}x$ & $(2\ga_{1},0)$ & 
$\bx^{00;22}$ & $\int_{2} d_{2}x d_{2}x$ & $(0,2\ga_{2})$ \\
\hline
$\bx^{1;2}$ & $\int_{1}\int_{2} d_{12}x$ & $(\ga_{1},\ga_{2})$ & 
$\bx^{\hone;\htwo}$ & $\int_{1}\int_{2} d_{\hone\htwo}x$ & $(\ga_{1},\ga_{2})$ \\
\hline
\end{tabular}
\end{center}
\label{table:rp}
\end{table}
Note that in Table~\ref{table:rp}, all increments belong to $\cp_{2,2}$, so that a regularity $(\al,\beta)$ means that the increment lyes into $\cp_{2,2}^{\al,\beta}$. Furthermore, the stack $\XX$ is a geometric rough sheet, insofar as there exists a regularization $x^{n}$ of $x$ such that $\lim_{n\to\infty}\|x-x^{n}\|_{\ga_1,\ga_2}=0$ and such that all the integrals in $\XX^{n}$, constructed out of $x^{n}$ in the Lebesgue-Stieljes sense, converge with respect to their natural respective norms in $\cp_{2,2}^{\ga_1,\ga_2}$, $\cp_{2,2}^{2\ga_1,\ga_2}$, $\cp_{2,2}^{\ga_1,2\ga_2}$ or $\cp_{2,2}^{2\ga_1,2\ga_2}$. 
\end{hypothesis}
\begin{remark}\label{Rmk:rs-regularization}
Let us be slightly more explicit about the our notion of geometric rough sheet, taking the example $\bx^{\hone;\htwo}\in \XX$. As mentioned in Hypothesis \ref{hyp:rs1}, we assume that there exists a regularization $x^{n}$ of $x$ such that $d_1x$, $d_2x$ and $d_{12}x$ are continuous and 
$$
\lim_{n\rightarrow \infty}\lVert x-x^n\rVert_{\ga_1,\ga_2}=0.
$$
Recalling our notation \eqref{eq:def-simplex} for simplexes, for a fixed $n\ge 0$ and $(s_1,s_2)\times (t_1,t_2)\in \Bigl(\cs_2\bigl([0,T]]\bigl)\Bigl)^2$, one can define %\samy{Write the example for $\bx^{\hat{1},\hat{2}}$ instead of $\bx^{n;11;02}$.}
$$
\bx^{n;\hone;\htwo}_{s_1s_2;t_1t_2}=\int_{s_1}^{s_2}\int_{t_1}^{t_2}\partial_1 x^n_{\sigma;\tau}\partial_{2}x^n_{\sigma;\tau}d\sigma d\tau,
$$
as a Riemann-Stieljes integral. What we assume in Hypothesis \ref{hyp:rs1} is thus the existence of an increment $\bx^{\hone;\htwo}$ in $\cp_{2,2}^{\ga_1;\ga_2}$ such that 
$$
\lim_{n\rightarrow}\lVert \bx^{n;\hone;\htwo}-\bx^{\hone;\htwo}\rVert_{\ga_1,\ga_2}=0.
$$
Similar statements hold true for the other elements of Table \ref{table:rp}.
\end{remark}

\begin{remark}\label{Rmk:rs-signature}
Like in the rough path case indexed by a 1-dimensional parameter, one cannot construct directly the elements in Table \ref{table:rp} thanks to classical integration tools when $\ga_1\le 1/2$. One exception is the first order increment $\bx^{1;2}=\delta x$, where $\delta$ is defined by \eqref{eq:rect-incr}. If both $\ga_1$ and $\ga_2$ are $>1/2$, the element $\bx^{\hone;\htwo}$ can also be defined through Young integration. In terms of the maps $\Lambda_1$, $\Lambda_2$ introduced in Proposition \ref{prop:Lambda}, in the Young case we have 
$$
\bx^{\hone;\htwo}=[\id-\laa_{1}\der_{1}] [\id-\laa_{2}\der_{2}] \left( \der_{1}x \, \der_{2}x\right).
$$
Similar constructions hold true for the other elements in Table \ref{table:rp} in the Young case.
\end{remark}

\section{Dissection of planar integrals}\label{sec:dissection}

This section is devoted to the abstract formulation of a controlled path decomposition and integration. We will start from the simple case of a path $y=\vp(x)$. Drawing on this example we shall come up with a reasonable notion of controlled path $y$ in the rough sheet context. 
For such a path, we will then focus on the definition of the increments
\begin{equation}\label{d01}
z^{1}:=\int_1\int_2\,y^{(1)}\,d_{12}x ,
\quad\text{and}\quad
z^{2}:=\int_1\int_2\,y^{(2)}\,d_{\hone\htwo}x .
\end{equation}
A simple enough decomposition for those two increments will be provided in Theorem~\ref{th:main-yr-integral}, leading to a change of variables for the rectangular increments of $y=\vp(x)$.

\subsection{Decomposition for the increments of a function of $\mathbf{x}$}\label{sec:dcp-function-x}

In this section, we perform a
Taylor expansion for a planar
increment $\delta \varphi^{(i)}(x)$, for $i=1,2$, and observe the
elements of the signature of $x$
showing up. This will be the key to get an
intuition about what the natural notion
of controlled paths in 2D should be. We begin
with an elementary Taylor expansion, for which we need an additional assumption on the function $\varphi$.

\begin{hypothesis}\label{hyp:bounds-higher-derivatives}
The funcion $\varphi: \RR \rightarrow \RR$ is such that $\varphi \in \mathcal{C}_{b}^{\infty}$ and its successive derivatives $\{\varphi^{(\ell)};\ell \ge 0\}$ fulfill the inequalities
\begin{align}\label{eq:bounds-higher-derivatives}
\lVert \varphi^{(\ell)}\rVert_{\infty}\leq \mathbf{c}_{1}^{\ell},
\quad\text{for a constant}\quad
\mathbf{c}_1>0.
\end{align}
 \end{hypothesis}
\begin{remark}\label{Rmk:deriv-bnd}
Condition \eqref{eq:bounds-higher-derivatives} might seem restrictive. However, it includes the following relevant cases:
\begin{enumerate}[label=\textbf{(\roman*)}]
    \item
    The function $\varphi$ can be chosen as an oscillating function of the form 
    $$
    \varphi(x)=a_1\sin(b_1 x)+ a_2\sin(b_2 x),
    \quad\text{for}\quad a_1,b_1,a_2,b_2\in \RR.
    $$
    \item
    For most paths $x$ of interest (like the fractional Browniam sheet studied in Section \ref{sec:fbs}), one can resort to a localization procedure. Namely one can restrict the analysis to the random square $[0,T_M]^2$ with 
    $$
    T_M:=\sup\bigl\{t\ge 0; |x_{\sigma;\tau}|\le M \text{ for all } (\sigma,\tau)\in [0,t]^2\bigl\}.
    $$
    In this case all the considerations below are understood in $[0,T_M]^2$. In particular, condition~\eqref{eq:bounds-higher-derivatives} becomes
    $$
    \sup\bigl\{ |\varphi^{(\ell)}(z)|; |z|\le M\bigl\}\le \mathbf{c}_1^{\ell}.
    $$
    It is satisfied for most of the usual functions. Typical examples would include 
    $$
    \varphi(x)=a_1 e^{b_1 x}, \quad \varphi(x)=\sum_{j=0}^{n}c_j x^{j},
    $$
    as well as sums and products of those functions. 
If $\lim_{M\rightarrow \infty}T_M=\infty$, the family $\{T_M\,;\,M\ge 1\}$ provides a proper localization sequence for Hypothesis \ref{hyp:bounds-higher-derivatives}.
\end{enumerate}
\end{remark}

Our next step is to give a controlled path decomposition for a function of the form $y=\vp(x)$. Before that we should check that H\"older spaces in the plane are stable by composition with a smooth function. Though this is arguably a classical property, we include a short proof for sake of completeness.

\begin{lemma}\label{L:stable-hnorm}
Let $\ga_1,\ga_2>0$ and recall that the spaces $\cp_{1,1}^{\ga_1,\ga_2}$ are defined by \eqref{c3}. We consider a field $x$ in $\cp_{1,1}^{\ga_1,\ga_2}$ and a function $\varphi\in \mathcal{C}_{b}^2(\RR)$. We set $y=\vp(x)$. Then $y \in \cp_{1,1}^{\ga_1,\ga_2}$ and there exists a constant $c_{\varphi}>0$ such that 
\begin{align}\label{eq:stable-H-norm}
\lVert y\rVert_{\ga_1,\ga_2}\, \le c_{\vp}\Bigl(1+ \lVert x\rVert^2_{\ga_1,\ga_2}\Bigl)
\end{align}
\end{lemma}
\begin{proof}
Let $(s_1,s_2)$ and $(t_1,t_2)$ be elements of $\mathcal{S}_2([0,T])$. We create an interpolation path between $x_{s_1;t_1}$ and $x_{s_2;t_2}$ indexed by $(\lambda,\mu)\in [0,1]^2$:
$$
a(\lambda,\mu) = x_{s_1;t_1} + \lambda\, \delta_1 x_{s_1s_2;t_1}+ \mu\, \delta_2 x_{s_1;t_1t_2}+\lambda\,\mu\,\delta x_{s_1s_2;t_1t_2}.
$$
It is then easily checked that 
\begin{equation*}
    \delta y_{s_1s_2;t_1t_2} =\vp(a(1,1))-\vp(a(1,0))-\vp(a(0,1))-\vp(a(0,0))
    = \delta b_{01;01},
\end{equation*}
where we have set $b(\lambda,\mu):=\vp \circ a(\lambda,\mu)$. Hence applying the classical differentiation rule \eqref{eq:simple-change-vb} in $[0,1]^2$, we obtain 
\begin{multline*}
\delta y_{s_1s_2;t_1t_2} \\
= \int_{[0,1]\times[0,1]}\vp'(a(\lambda,\mu))d_{12}a(\lambda,\mu)d\lambda d\mu 
+ \int_{[0,1]\times[0,1]}\vp''(a(\lambda,\mu))d_{1}a(\lambda,\mu)d_{2}a(\lambda,\mu)d\lambda d\mu. 
\end{multline*}
Taking into account $\lVert \vp'\lVert_{\infty}+\lVert \vp''\lVert_{\infty}<C_{\vp}$ and computing $d_1 a$, $d_2 a$ and $d_{12} a$, we easily get
\begin{multline*}
\bigl|\delta y_{s_1s_2;t_1t_2}\bigl| \\
\le 
C_{\vp}\Bigl(\bigl|\delta x_{s_1s_2;t_1t_2}\bigl|+ \bigl|\delta_1 x_{s_1s_2;t_1}\bigl|  \bigl|\delta_2 y_{s_1;t_1t_2}\bigl| + \bigl|\delta_1 x_{s_1s_2;t_1}\bigl|  \bigl|\delta y_{s_1s_2;t_1t_2}\bigl|+ \bigl|\delta_2 x_{s_1;t_1t_2}\bigl|  \bigl|\delta y_{s_1s_2;t_1t_2}\bigl| \Bigl),
\end{multline*}
from which our claim \eqref{eq:stable-H-norm} is easily deduced.    
\end{proof}

We now turn to a first decomposition of a field $y=\vp(x)$ as a controlled process with respect to the $2d$ increments of $x$.
\begin{lemma}\label{L:2D-Taylor-expansion1}
Consider a regular path $x\in \mathcal{P}^{1,1}_{1,1}$ satisfying Hypothesis \ref{hyp:rs1} (see~\eqref{c3} for the definition of $\mathcal{P}^{1,1}_{1,1}$) and a smooth function $\varphi:\RR\rightarrow \RR$. 
As mentioned in Notation \ref{not:smooth-path-y-fx}, for a given integer $i\ge 0$ we set $y^{(i)}=\varphi^{(i)}(x)$. Then the quantity $\delta y^{(i)}$ admits the following expansion:
\begin{align}\label{E:delta-y-expansion}
    \delta y^{(i)}=  y^{(i+1)}\delta x \, +\,\delta_2 y^{(i+1)} \delta_1 x+\,\underset{:=\bR^{(i)}}{\underbrace{\sum_{\ell=2}^{\infty}\frac{1}{\ell!}\,\Bigl( y^{(i+\ell)}\delta_2(\delta_1 x)^{\ell}\, +\,\delta_2y^{(i+\ell)}(\delta_1 x)^\ell\Bigl)}} \,.
\end{align}
Moreover, if $x\in \mathcal{P}_{1,1}^{\ga_1,\ga_2}$ and $\varphi\in C_b^{\infty}(\mathbb R)$ satisfies Hypothesis \ref{hyp:bounds-higher-derivatives}, then the reminder $\bR^{(i)}$ is such that $\bR^{(i)}\in \mathcal{P}_{2,2}^{2\ga_1;\ga_2}$.
\end{lemma}

\begin{proof}
Recall from \eqref{eq:rect-incr} that the rectangular increments $\delta y^{(i)}=\delta \varphi^{(i)}(x)$ can be written as
\begin{align}\label{e:rec:inc}
    \delta y^{(i)}_{s_1 s_2;t_1 t_2}=[\varphi^{(i)}(x_{s_2;t_2})-\varphi^{(i)}(x_{s_1;t_2})]-[\varphi^{(i)}(x_{s_2;t_1})-\varphi^{(i)}(x_{s_1;t_1})].
\end{align}    
Therefore, a formal Taylor expansion
with respect to the $s$ variables for
the right hand side of relation~\eqref{e:rec:inc} yields
\begin{align*}
\delta y^{(i)}_{s_1s_2;t_1t_2}=\sum_{j\geq 1}\frac{y^{(i+j)}_{s_1;t_2}}{j!}(\delta_1 x_{s_1s_2;t_2})^j-\sum_{j\geq 1}\frac{y^{(i+j)}_{s_1;t_1}}{j!}(\delta_1 x_{s_1s_2;t_1})^j
\end{align*}
We now insert $\pm y^{(i+j)}_{s_1;t_1}\, (\delta_1 x_{s_1s_2;t_2})^{j}$, in the quantities above.  After some elementary algebraic manipulations we obtain
\begin{align}\label{E:expan:delta y}
 \delta y^{(i)}_{s_1s_2;t_1t_2}=
A^{(i)}_{s_1s_2;t_1t_2} + B^{(i)}_{s_1s_2;t_1t_2},
\end{align}
where the increments $A^{(i)},B^{(i)}$ are respectively defined by
\begin{eqnarray}
A^{(i)}_{s_1s_2;t_1t_2}&=&
\,\delta_2 y^{(i+1)}_{s_1;t_1t_2} \, \delta_1 x_{s_1s_2;t_2}+ \sum_{j\geq 2}\frac{1}{j!} \, \delta_2 y^{(i+j)}_{s_1;t_1t_2}(\delta_1 x_{s_1s_2;t_2})^j
\label{d1}
\\
B^{(i)}_{s_1s_2;t_1t_2}&=&\, y^{(i+1)}_{s_1t_1} \, \delta x_{s_1s_2;t_1t_2}\,+\,
\sum_{j\geq 2}\frac{1}{j!} \, y^{(i+j)}_{s_1;t_1}\Bigl[(\delta_1 x_{s_1s_2;t_2})^j-(\delta_1 x_{s_1s_2;t_1})^j\Bigl] .
\label{d2}
\end{eqnarray}
%%%%
Hence gathering \eqref{d1}-\eqref{d2} and using our convention~\eqref{c4} on products we end up with  
\begin{align}
\delta y^{(i)}\,=\,A^{(i)}+B^{(i)}= y^{(i+1)}\delta x+ \delta_2 y^{(i+1)}\delta_1 x+\bR^{(i)}
\end{align}
This yields the desired identity \eqref{E:delta-y-expansion}. 

We now prove that $\bR^{(i)}$ in \eqref{E:delta-y-expansion} is such that $\bR^{(i)}\in \mathcal{P}_{2,2}^{2\ga_1;\ga_2}$. To this aim we recall from Hypothesis \ref{hyp:rs1} that $x \in \mathcal{P}_{1,1}^{\ga_1;\ga_2}$ with $\ga_1>1/3$, $\ga_2>1/2$. Next we resort to the following elementary identity,
which valid for $z_1,z_2\in \RR$ and $\ell\ge 1$:
$$
z_2^{\ell}-z_1^{\ell}=(z_2-z_1)\sum_{k=0}^{\ell-1}z_1^k z_2^{\ell-1-k},
$$
applied to $z_1=\delta_1 x_{s_1s_2;t_1}$ and $z_2=\delta_1 x_{s_1s_2;t_2}$. We obtain
$$
\delta_2(\delta_1 x_{s_1s_2;.})^{\ell}_{t_1t_2}= \delta x_{s_1s_2;t_1t_2}\sum_{k=0}^{\ell-1}(\delta_1 x_{s_1s_2;t_1})^k(\delta_1 x_{s_1s_2;t_2})^{\ell-1-k}.
$$
Reporting this identity in the definition \eqref{E:delta-y-expansion} of $\bR^{(i)}$ we obtain
\begin{align}
&\Bigl\lvert \bR^{(i)}_{s_1s_2;t_1t_2}\Bigl\rvert
 \leq 
 \sum_{\ell\ge 2}\frac{1}{\ell!}\Bigl[|y^{(i+\ell)}_{s_1t_1}| \left|\delta_2(\delta_1 x_{s_1s_2;.})^{\ell}_{t_1t_2}\right|+|\delta_2 y^{(i+\ell)}_{s_1;t_1t_2}|\left|(\delta_1 x_{s_1s_2;t_2})^{\ell}\right|\Bigl]\\
&\le \sum_{\ell\ge 2}\frac{1}{\ell!}\Bigl[\lvert y^{(i+\ell)}_{s_1t_1}\rvert \left|(\delta x_{s_1s_2;t_1t_2})\right| \sum_{k=0}^{\ell-1}|(\delta_1 x_{s_1s_2;t_1})|^k|(\delta_1 x_{s_1s_2;t_2})|^{\ell-1-k}+|\delta_2 y^{(i+\ell)}_{s_1;t_1t_2}|\left|(\delta_1 x_{s_1s_2;t_2})^{\ell}\right|\Bigl] .
\nonumber
\end{align}
In addition, one can bound the terms $|y^{(\ell)}_{s;t}|=|\varphi^{(\ell)}(x_{s;t})|$ by $\lVert \varphi^{(\ell)}\rVert_{\infty}$. The terms $\delta_1 x$, $\der x$ can also be upper bounded thanks to Hypothesis \ref{hyp:rs1} on the H\"older regularity of the increments of $x$. 
Recalling our Definition \ref{def:2D-hnorm-P2P3} of H\"older norms in the plane, this yields
\begin{multline}
\Bigl\lvert \bR^{(i)}_{s_1s_2;t_1t_2}\Bigl\rvert 
\le  
\sum_{\ell\ge 2}\frac{1}{\ell!}\Bigl[\ell\lVert \varphi^{(i+\ell)}\rVert_{\infty} \left\lVert x \right\rVert_{\ga_1,\ga_2} \lVert  x\rVert_{\ga_1;*}^{\ell-1}\,+\,\lVert y^{(i+\ell)}\rVert_{*;\ga_2}\lVert x\rVert^{\ell}_{\ga_1;*}\Bigl]|s_1-s_2|^{\ell\ga_1}|t_1-t_2|^{\ga_2}\\
\le 
\sum_{\ell\ge 2}\frac{1}{\ell!}\Bigl[\ell\lVert \varphi^{(i+\ell)}\rVert_{\infty} \left\lVert x \right\rVert_{\ga_1,\ga_2} \lVert  x\rVert_{\ga_1;*}^{\ell-1}\,+\,\lVert \varphi^{(i+\ell+1)}\rVert_{\infty}\lVert x\rVert_{*;\ga_2}\lVert x\rVert^{\ell}_{\ga_1;*}\Bigl]|s_1-s_2|^{\ell\ga_1}|t_1-t_2|^{\ga_2} \label{E:(2,1)-Holder-remainder}
\end{multline}
One can thus extract regularities of the form $|s_1-s_2|^{2\ga_1}|t_1-t_2|^{\ga_2}$ from inequality \eqref{E:(2,1)-Holder-remainder}. This reads
\begin{multline}
\frac{\bigl| \bR^{(i)}_{s_1s_2;t_1t_2} \bigr|}
     {|s_1-s_2|^{2\ga_1}\,|t_1-t_2|^{\ga_2}}
\le
\sum_{\ell\ge 2}\frac{1}{\ell!}
\Bigl[
\ell\lVert \varphi^{(i+\ell)}\rVert_{\infty}
\lVert x \rVert_{\ga_1,\ga_2}
\lVert x \rVert_{\ga_1;*}^{\ell-1}
\\
\qquad\qquad
+\,\lVert \varphi^{(i+\ell+1)}\rVert_{\infty}
\lVert x \rVert_{*;\ga_2}
\lVert x \rVert^{\ell}_{\ga_1;*}
\Bigr]
|s_1-s_2|^{(\ell-2)\ga_1}.\notag
\end{multline}
Hence estimating the quantities $\lVert \varphi^{(i+\ell)}\rVert_{\infty}$ and $\lVert \varphi^{(i+\ell+1)}\rVert_{\infty}$ thanks to the assumption \eqref{eq:bounds-higher-derivatives}, one can write 
\begin{equation}\label{c5}
\frac{\Bigl\lvert \bR^{(i)}_{s_1s_2;t_1t_2}\Bigl\rvert}{|s_1-s_2|^{2\ga_1}|t_1-t_2|^{\ga_2}} 
\le \sum_{\ell\ge 2}\frac{\mathbf{c}_1^{\ell}}{\ell!}\Bigl[ \mathbf{c}_1^{i}\,\ell \left\lVert x \right\rVert_{\ga_1,\ga_2} \lVert  x\rVert_{\ga_1;*}^{\ell-1}\,+\,\mathbf{c}_1^{i+1}\lVert x\rVert_{*;\ga_2}\lVert x\rVert^{\ell}_{\ga_1;*}\Bigl]|s_1-s_2|^{(\ell-2)\ga_1} .
\end{equation}
Translating~\eqref{c5} into H\"older norms, we end up with
\begin{align}
    \lVert \bR^{(i)}\rVert_{2\ga_1;\ga_2}\le \sum_{\ell\ge 2}\frac{\mathbf{c}_1^{\ell}}{\ell!}\Bigl[\mathbf{c}_1^{i} \ell \left\lVert x \right\rVert_{\ga_1,\ga_2} \lVert  x\rVert_{\ga_1;*}^{\ell-1}\,+\,\mathbf{c}_1^{i+1}\lVert x\rVert_{*;\ga_2}\lVert x\rVert^{\ell}_{\ga_1;*}\Bigl]T^{(\ell-2)\ga_1}<\infty,
\end{align}
where we have used Hypothesis \ref{hyp:bounds-higher-derivatives} for the second inequality. This shows that $\bR^{(i)}\in \mathcal{P}_{2,2}^{2\ga_1;\ga_2}$, and completes our proof.
\end{proof}
%Now, it is well know from one-parameter controlled processes that, when $\varphi:\RR\rightarrow \RR$ is a $\mathcal{C}_b^{i+2}$ function for $i=1,2$, then 
%\begin{align}\label{E:reminder1}
%R^{{(i)};2\ga_1;0}:=\delta_1 y^{(i)}-y^{(i+1)}(\delta_1 x),    
%\end{align}
%is in $\mathcal{P}_{2,1}^{2\ga_1,*}$.  
We have argued that Hypothesis \ref{hyp:bounds-higher-derivatives} covered enough cases of interest for the function $y$ (see in
particular Remark \ref{Rmk:deriv-bnd}). Under this
assumption we have obtained Lemma \ref{L:2D-Taylor-expansion1} in a simple way. However, elaborating slightly on standard Taylor expansion techniques, one can get a much more
standard hypothesis on $\vp$. We summarize those findings in the lemma below.

%The regularity of $\bR^{(i)}$ from Lemma \ref{L:2D-Taylor-expansion1} can be obtained under conditions minimal than \eqref{eq:bounds-higher-derivatives}. Then, we have the following lemma   %(relax condition \eqref{E:deriv-bnd})
\begin{lemma}\label{L:2D-Taylor-expansion2}
Let $x\in \mathcal{P}_{1,1}^{\ga_1,\ga_2}$ with $\ga_1>1/3$, $\ga_2>1/2$ and $\varphi\in C_b^{5}(\mathbb R)$. For $i=1,2$, recall that we have set  $y^{(i)}:=\varphi^{(i)}(x)$ and in~\eqref{E:delta-y-expansion} we have defined $\mathbf{R}^{(i)}$ by:
\begin{equation}\label{d3}
\mathbf{R}^{(i)}:= \delta y^{(i)}-  y^{(i+1)}\delta x \, -\,\delta_2 y^{(i+1)} \delta_1 x,
\end{equation}
for $i=1,2$. Then $\mathbf{R}^{(i)}\in \mathcal P_{2,2}^{2\ga_1;\ga_2}$. Moreover, we have
\begin{multline}\label{E:R^(i)-norm-bnd}
\lVert \bR^{(i)}\rVert_{2\ga_1;\ga_2} 
\le \lVert \vp^{(i+3)}\rVert_{\infty} \lVert \bx^{11;00}\rVert_{2\ga_1;*}\Bigl(\lVert \bx^{0;2}\rVert_{*;\ga_2}+\frac{T^{\ga_1}}{3}\lVert \bx^{1;2}\rVert_{\ga_1;\ga_2}\Bigl) \\
+\lVert \vp^{(i+2)}\rVert_{\infty}\lVert \bx^{1;0}\rVert_{\ga_1;*}\lVert \bx^{1;2}\rVert_{\ga_1;\ga_2}.
\end{multline}
In particular, the conclusion of Lemma \ref{L:2D-Taylor-expansion1} still hold when one replaces Hypothesis \ref{hyp:bounds-higher-derivatives} by $\vp \in \cac_b^{5}(\RR)$.
\end{lemma}

\begin{proof} 
Since $\varphi:\RR\rightarrow \RR$ is a $\mathcal{C}_b^{i+2}$ function for $i=1,2$, then Taylor's expansion of order $2$ with integral type reminder yields
\begin{align}\label{E:reminder1}
\br^{{(i)}}_{s_1s_2;t_1}:=&\delta_1 y^{(i)}_{s_1s_2;t_1}-y_{s_1t_1}^{(i+1)}\delta_1 x
_{s_1s_2;t_1}\notag\\
 =&\int_0^1 (1-\lambda)\varphi^{(i+2)}\bigl(x_{s_1;t_1}+\lambda\, \delta_1 x_{s_1s_2;t_1}\bigl)(\delta_1 x_{s_1s_2;t_1})^2\,d\lambda.    
\end{align}
This easily implies $\br^{(i)}\in \mathcal{P}_{2,1}^{2\ga_1,*}$. Moreover, applying $\delta_2$ on the left hand side of \eqref{E:reminder1} as well as~\eqref{eq:difrul-C1-C2}, we obtain
\begin{align}\label{E:delta2-r^(i)}
\delta_2 \br^{{(i)}}_{s_1s_2;t_1t_2}&=\delta y^{(i)}_{s_1s_2;t_1t_2}-  y^{(i+1)}_{s_1t_1}\delta x_{s_1s_2;t_1t_2} \, -\,\delta_2 y^{(i+1)}_{s_1;t_1t_2} \delta_1 x_{s_1s_2;t_2}
= \mathbf{R}^{{(i)}}_{s_1s_2;t_1t_2},
\end{align}
where $\mathbf{R}^{{(i)}}$ is our remainder in~\eqref{d3}.
Applying Lemma \ref{L:stable-hnorm} to $y^{(i)}=\vp^{(i)}(x)$,  we easily get that $\delta y^{(i)}\in \cp_{2,2}^{\ga_1,\ga_2}$. Moreover, it is easily seen that $\,y^{(i+1)}\delta x$\, and $\,\delta_2 y^{(i+1)} \delta_1 x$\, also lye in $\mathcal{P}_{2,2}^{\ga_1,\ga_2}$. Plugging this information into \eqref{E:delta2-r^(i)} we have that $\bR^{{(i)}}\in \mathcal{P}_{2,2}^{\ga_1,\ga_2}$. 

We will show now that $\bR^{(i)}$ is in fact an element of $\mathcal{P}_{2,2}^{2\ga_1,\ga_2}$. To this aim , 
invoke~\eqref{E:delta2-r^(i)} to write $\mathbf{R}^{{(i)}}=\delta_2 \br^{{(i)}}$ and apply $\delta_{2}$ to the right hand side of~\eqref{E:reminder1}. This yields
\begin{equation}\label{E:R^{i}-decomposition}
\bR^{(i)}_{s_1s_2;t_1t_2}=\delta_2 (\br^{(i)}_{s_1s_2;\cdot})_{t_1t_2}
=\bC^{(i)}_{s_1s_2;t_1t_2}+\bD^{(i)}_{s_1s_2;t_1t_2} ,
\end{equation}
%%%%%%%%%%%%%%%
\begin{comment}
\begin{align}&\bR^{(i)}_{s_1s_2;t_1t_2}=\delta_2 (\br^{(i)}_{s_1s_2;\cdot})_{t_1t_2}\notag\\
&=\,\delta_2\Big(\int_0^1 (1-\lambda)\varphi^{(i+2)}\bigl(x_{s_1;\cdot}+\lambda \,\delta_1 x_{s_1s_2;\cdot}\bigl)(\delta_1 x_{s_1s_2;\cdot})^2 d\lambda\Big)_{t_1t_2}\notag\\
&=\,\int_0^1 (1-\lambda)\Bigl(\varphi^{(i+2)}\bigl(x_{s_1;t_2}+\lambda\, \delta_1 x_{s_1s_2;t_2}\bigl)-\varphi^{(i+2)}\bigl(x_{s_1;t_1}+\lambda \,\delta_1 x_{s_1s_2;t_1}\bigl)\Bigl)(\delta_1 x_{s_1s_2;t_2})^2 \,d\lambda\nonumber\\
&\quad+\int_0^1 (1-\lambda)\varphi^{(i+2)}\bigl(x_{s_1;t_1}+\lambda\, \delta_1 x_{s_1s_2;t_1}\bigl)\delta_2\bigl(\delta_1 x_{s_1s_2;\cdot}\bigl)^2_{t_1t_2}\, d\lambda\notag\\
&=\bC^{(i)}_{s_1s_2;t_1t_2}+\bD^{(i)}_{s_1s_2;t_1t_2},
\end{align}
\end{comment}
%%%%%%%%%%%%%%%
where we have set 
\begin{multline}\label{E:C^(i)}
\bC^{(i)}_{s_1s_2;t_1t_2} \\
:=\int_0^1 (1-\lambda)
\Bigl[\varphi^{(i+2)}\bigl(x_{s_1;t_2}+\lambda\, \delta_1 x_{s_1s_2;t_2}\bigl)
-\varphi^{(i+2)}\bigl(x_{s_1;t_1}+\lambda \,\delta_1 x_{s_1s_2;t_1}\bigl)\Bigl]
(\delta_1 x_{s_1s_2;t_2})^2 \,d\lambda
\end{multline}
and
\begin{equation}\label{E:D^(i)}
\bD^{(i)}_{s_1s_2;t_1t_2}:=\displaystyle\int_0^1 (1-\lambda)\varphi^{(i+2)}\bigl(x_{s_1;t_1}+\lambda\, \delta_1 x_{s_1s_2;t_1}\bigl)\delta_2\bigl(\delta_1 x_{s_1s_2;\cdot}\bigl)^2_{t_1t_2}\, d\lambda.
\end{equation}
Now in order to obtain a proper expression for $\bC^{(i)}$, let us write $x_{s_1;t_2}+\lambda \doo x_{s_1s_2;t_2}=u+h$, with $u=x_{s_1;t_1}+\lambda \doo x_{s_1s_2;t_1}$ and $h=\dt x_{s_1;t_1t_2}+\lambda \der x_{s_1s_2;t_1t_2}$.
Then applying a first order Taylor formula with integral remainder to $\vp^{(i+2)}$, we get
\begin{eqnarray*}
\varphi^{(i+2)}\bigl(x_{s_1;t_2}+\lambda\, \delta_1 x_{s_1s_2;t_2}\bigl)
-\varphi^{(i+2)}\bigl(x_{s_1;t_1}+\lambda \,\delta_1 x_{s_1s_2;t_1}\bigl)
&=&
\vp^{(i+2)}(x+h)-\vp^{(i+2)}(x) \\
&=&
\int_0^1\vp^{(i+3)}(x+\mu h)h \,d\mu.
\end{eqnarray*}
Plugging this relation into $\eqref{E:C^(i)}$ we obtain
\begin{align}\label{E:C^(i)1+2}
\bC^{(i)}_{s_1s_2;t_1t_2}=\bC^{(i),1}_{s_1s_2;t_1t_2}+\bC^{(i),2}_{s_1s_2;t_1t_2},
\end{align}
with 
\begin{multline}\label{E:C^(i,1)}
\bC^{(i),1}_{s_1s_2;t_1t_2}=
\int_{[0,1]^2} (1-\lambda)\varphi^{(i+3)}\bigl(x_{s_1;t_1}+\lambda\, \delta_1 x_{s_1s_2;t_1}+\mu\, \delta_2 x_{s_1;t_1t_2}+\lambda\, \mu\, \delta x_{s_1s_2;t_1t_2}\bigl) 
\\
\times \delta_2 x_{s_1;t_1t_2}(\delta_1 x_{s_1s_2;t_2})^2\,d\lambda\, d\mu \, ,
\end{multline}
and
\begin{multline}\label{E:C^(i,2)}
\bC^{(i),2}_{s_1s_2;t_1t_2}=
\int_{[0,1]^2} \lambda(1-\lambda)\varphi^{(i+3)}\bigl(x_{s_1;t_1}+\lambda\, \delta_1 x_{s_1s_2;t_1}+\mu \delta_2 x_{s_1;t_1t_2}+\lambda \mu \delta x_{s_1s_2;t_1t_2}\bigl) \\
\times \delta x_{s_1s_2;t_1t_2}(\delta_1 x_{s_1s_2;t_2})^2 d\lambda d\mu .
\end{multline}
In addition, the terms $\doo x$, $\dt x$, $\der x$ in \eqref{E:C^(i)1+2} can be identified as elements of the signature in Table \ref{table:rp}. Namely we have $\doo x= \bx^{1;0}$, $\dt x= \bx^{0;2}$, $\delta x=\bx^{1;2}$, $(\dt x)^2= 2\bx^{11;00}$. We end up with   
\begin{multline}\label{E:C^(i,1)'}
\bC^{(i),1}_{s_1s_2;t_1t_2}=\int_{[0,1]^2} 2(1-\lambda)\varphi^{(i+3)}\bigl(x_{s_1;t_1}+\lambda \bx_{s_1s_2;t_1}^{1;0}+\mu \bx^{0;2}_{s_1;t_1t_2}+\lambda \mu \bx^{1;2}_{s_1s_2;t_1t_2}\bigl) \\
\times \bx_{s_1;t_1t_2}^{0;2}\bx_{s_1s_2;t_2}^{11;00}\,d\lambda\, d\mu,
\end{multline}
and 
\begin{multline}\label{E:C^(i,2)'}
\bC^{(i),2}_{s_1s_2;t_1t_2}
=\int_{[0,1]^2} 2(1-\lambda)\varphi^{(i+3)}\bigl(x_{s_1;t_1}+\lambda \bx_{s_1s_2;t_1}^{1;0}+\mu \bx^{0;2}_{s_1;t_1t_2}+\lambda \mu \bx^{1;2}_{s_1s_2;t_1t_2}\bigl)\\
\times (\lambda\bx_{s_1s_2;t_1t_2}^{1;2})\bx_{s_1s_2;t_2}^{11;00}d\lambda d\mu.
\end{multline}
Summing those two terms we get
\begin{align}\label{E:C^(i)'}
\bC^{(i)}_{s_1s_2;t_1t_2}=\int_{[0,1]^2}2(1-\lambda)\varphi^{(i+3)}\bigl(x_{s_1;t_1}+\lambda \bx_{s_1s_2;t_1}^{1;0}&+\mu \bx^{0;2}_{s_1;t_1t_2}+\lambda \mu \bx^{1;2}_{s_1s_2;t_1t_2}\bigl)\\
&\times\bigl(\bx^{0;2}_{s_1;t_1t_2}+\lambda\bx_{s_1s_2;t_1t_2}^{1;2}\bigl)\bx_{s_1s_2;t_2}^{11;00}d\lambda d\mu.\notag
\end{align}
We now handle the term $\bD^{(i)}$ defined by \eqref{E:D^(i)}.  To this aim, we consider the variables $s_1,s_2$ in $\doo x_{s_1s_2;t_1}$ as fixed, that is we set $f_{t}=\doo x_{s_1s_2;t}$. Next we apply \eqref{eq:difrulu-C1-C1} to $g=f$ and $h=f$. This yields $\dt f^{2}_{t_1t_2}=f_{t_1}\, \dt f_{t_1t_2}+ \dt f_{t_1t_2}\,  f_{t_2}$. Taking into account the definition of $f$, some elementary algebraic considerations show that 
\begin{align}\label{d2(d1x)^2}
\dt (\doo x)^2=\doo x *_2^s \der x+ \der x *_2^s \doo x.
\end{align}
where we write $*_2^s$ for the gluing of the variables like in~\eqref{b1}, but where the gluing is only considered in the $s$-variable. In addition, one can express the increments in \eqref{d2(d1x)^2} as elements of the signature in Table \ref{table:rp}, similarly to \eqref{E:C^(i,1)'} and \eqref{E:C^(i,2)'}. We get 
\begin{align*}
\dt(\doo x)^2=\bx^{1;0}*_{2}^s\bx^{1;2}+ \bx^{1;2}*_{2}^s\bx^{1;0}.
\end{align*}
Going back to \eqref{E:D^(i)}, we get
\begin{align}\label{E:D^(i)1}
&\bD^{(i)}_{s_1s_2;t_1t_2}=\int_0^1 (1-\lambda)\varphi^{(i+2)}\bigl(x_{s_1;t_1}+\lambda\,\bx^{1;0}_{s_1s_2;t_1}\bigl)\bigl(\bx^{1;0}_{s_1s_2;t_1}\bx_{s_1s_2;t_1t_2}^{1;2}\,+\, \bx_{s_1s_2;t_1t_2}^{1;2}\bx_{s_1s_2;t_2}^{1;0}\bigl)d\lambda.  
\end{align}
Thus gathering \eqref{E:C^(i)'} and \eqref{E:D^(i)1} into \eqref{E:R^{i}-decomposition}, we have obtained 
%Using the iterated integrals notation $\bx^{1;0}$, $\bx^{0;2}$, $\bx^{1;2}$ and $\bx^{11;00}$ for the increments $\delta_1 x$, $\der_2 x$, $\delta x$ and $(\delta_1 x)^2/2$, respectively, we can rewrite $\bR^{(i)}$ as follows
\begin{align}\label{E:Reminder-signature}
&\bR^{(i)}_{s_1s_2;t_1t_2}= \int_{[0,1]^2} 2(1-\lambda)\varphi^{(i+3)}\bigl(x_{s_1;t_1}+\lambda\, \bx^{1;0}_{s_1s_2;t_1}+\mu\, \bx^{0;2}_{s_1;t_1t_2}+\lambda\, \mu\, \bx^{1;2}_{s_1s_2;t_1t_2}\bigl)\notag\\
&\hspace{7cm}\times\big( \bx^{0;2}_{s_1;t_1t_2}\bx_{s_1s_2;t_2}^{11;00}+ \lambda\,\bx_{s_1s_2;t_1t_2}^{1;2}\bx_{s_1s_2;t_2}^{11;00}\big)\,d\lambda\, d\mu\notag\\
&\qquad\quad\quad+\int_0^1 (1-\lambda)\varphi^{(i+2)}\bigl(x_{s_1;t_1}+\lambda\,\bx^{1;0}_{s_1s_2;t_1}\bigl)\bigl(\bx^{1;0}_{s_1s_2;t_1}\bx^{1;2}_{s_1s_2;t_1t_2}\,+\, \bx^{1;2}_{s_1s_2;t_1t_2}\bx^{1;0}_{s_1s_2;t_1}\bigl)d\lambda.
\end{align}
With \eqref{E:Reminder-signature} in hand, we can now get the bound \eqref{E:R^(i)-norm-bnd} easily. That is for $s_1\le s_2$ and $t_1\le t_2$ we have 
\begin{align}\label{E:bounds-increments-R^(i)}
\lvert \bR^{(i)}_{s_1s_2;t_1t_2}\rvert&\le \lVert \vp^{(i+3)}\rVert_{\infty} \lVert \bx^{11;00}\rVert_{2\ga_1;*}\Bigl(\lVert \bx^{0;2}\rVert_{*;\ga_2}+\frac{1}{3}\lVert \bx^{1;2}\rVert_{\ga_1;\ga_2}|s_2-s_1|^{\ga_1}\Bigl)|s_2-s_1|^{2\ga_1}|t_2-t_1|^{\ga_2}\notag\\
&\quad+ \lVert \vp^{(i+2)}\rVert_{\infty}\lVert \bx^{1;0}\rVert_{\ga_1;*}\lVert \bx^{1;2}\rVert_{\ga_1;\ga_2}|s_2-s_1|^{2\ga_1}|t_2-t_1|^{\ga_2}.
\end{align}
Translating \eqref{E:bounds-increments-R^(i)} into H\"older norms, we end up with
\begin{multline}
\lVert \bR^{(i)}\rVert_{2\ga_1;\ga_2}\\
\le \lVert \vp^{(i+3)}\rVert_{\infty} \lVert \bx^{11;00}\rVert_{2\ga_1;*}\Bigl(\lVert \bx^{0;2}\rVert_{*;\ga_2}+\frac{T^{\ga_1}}{3}\lVert \bx^{1;2}\rVert_{\ga_1;\ga_2}\Bigl) +\lVert \vp^{(i+2)}\rVert_{\infty}\lVert \bx^{1;0}\rVert_{\ga_1;*}\lVert \bx^{1;2}\rVert_{\ga_1;\ga_2}.
\end{multline}
This shows that $\bR^{(i)}\in \mathcal{P}_{2,2}^{2\ga_1;\ga_2}$, and completes the proof of our claim \eqref{E:R^(i)-norm-bnd}.
\end{proof}

\begin{remark}\label{Rmk:control1}
If one puts together \eqref{d3} and \eqref{E:delta2-r^(i)}, we get the following decomposition for the field $y^{(i)}$: 
\begin{align*}
\delta_1 y^{(i)}&=y^{(i+1)}\doo x+\br^{(i)},\\
\der y^{(i)}&= y^{(i+1)}\der x + \dt y^{(i+1)}\doo x +\bR^{(i)},
\end{align*}
where we recall that $y^{(i+1)}\in \cp_{1,1}^{\ga_1,\ga_2}$, $\dt y^{(i+1)}\in \cp_{1,2}^{\ga_1,\ga_2}$, and where we have proved the relations $\br^{(i)}\in \cp_{2,1}^{2\ga_1,*}$, $\dt \br^{(i)}=\bR^{(i)}$, and $\bR^{(i)}\in \cp_{2,2}^{2\ga_1,\ga_2}$.  
As we will see, this decomposition will be a footprint for a more general effective definition of controlled path in the plane.
\end{remark}

\subsection{Rough-Young Integration}
In this section we will start from our conclusion in Section~\ref{sec:dcp-function-x}. That is 
thanks to Remark \ref{Rmk:control1}, we now have a better grasp on the notion of controlled path for our 2d-integration purposes. 
We label a definition in this sense below.
\begin{definition}\label{def:controlled-process} Let $x$ be a field indexed by $[0,T]^2$ such that Hypothesis \ref{hyp:rs1} is satisfied. Let also $\zeta$ be an element of $\mathcal{P}_{1,1}^{\ga_1,\ga_2}$ with $\ga_1>1/3$, $\ga_2>1/2$. We say that $\zeta$ is controlled by $x$ if  
\begin{align}
\delta_1 \zeta &=\zeta^1\, \bx^{1;0}+ \br,\label{E:controlled-(2,0)reminder}\\
\delta \zeta &=\zeta^{1}\,\bx^{1;2}+\delta_2 \zeta^{1} \bx^{1;0}  +\delta_2\mathbf{r},\label{E:controlled-(2,1)reminder}
\end{align}
with $\zeta^{1}\in\mathcal{P}_{1,1}^{\gamma_1,\gamma_2}$ and $\br\in\mathcal{P}_{2,1}^{2\gamma_1,*}$ such that $\der_2 \br\in \mathcal{P}_{2,2}^{2\gamma_1,\ga_2}$, and using our convention \eqref{c4} for products of increments.
\end{definition}

%XXXXX Add Table 2 here XXXXXX
In addition to Definition~\ref{def:controlled-process} and on top of the elements given in Table~\ref{table:rp}, our integration procedure will necessitate some second order increments based on $x$. We label them in the table below

\begin{hypothesis}\label{hyp:rs2}
Under the same assumptions as in Hypothesis \ref{hyp:rs1}, we suppose that the geometric rough sheet $\XX$ includes the elements in Table~\ref{table:rs2}.
\begin{table}[htp]
\caption{Other elements of the rough sheets $\XX$ in the rough-Young case}
\begin{center}
\begin{tabular}{|c|c|c||c|c|c|}
\hline
\emph{Increment} & \emph{Interpretation} & \emph{Regularity} & \emph{Increment} & \emph{Interpretation} & \emph{Regularity} \\
\hline
$\bx^{11;02}$ & $\int_{1}\int_{2} d_{1}x d_{12}x$ & $(2\ga_{1},\ga_{2})$ & 
$\bx^{1\hone;0\htwo}$ & $\int_{1}\int_{2} d_{1}xd_{\hone\htwo}x$ & $(2\ga_{1},\ga_{2})$ \\
\hline
$\bx^{11;0\cdot2}$ & $\int_{1}d_{1}x\int_{2} d_{12}x$ & $(2\ga_{1},\ga_{2})$ & 
$\bx^{1\hone;0\cdot\htwo}$ & $\int_{1}d_{1}x\int_{2} d_{\hone\htwo}x$ & $(2\ga_{1},\ga_{2})$ \\
\hline
$\bx^{11;22}$ & $\int_{1}\int_{2} d_{12}x d_{12}x$ & $(2\ga_{1},2\ga_{2})$ & 
$\bx^{1\hone;2\htwo}$ & $\int_{1}\int_{2} d_{12}x d_{\hone\htwo}x$ & $(2\ga_{1},2\ga_{2})$ \\
\hline
$\bx^{11;2\cdot2}$ & $\int_{1}\int_{2} d_{12}x \int_2 d_{12}x$ & $(2\ga_{1},2\ga_{2})$ & 
$\bx^{1\hone;2\cdot\htwo}$ & $\int_{1}\int_{2} d_{12}x \int_2 d_{\hone\htwo}x$ & $(2\ga_{1},2\ga_{2})$\\
\hline
\end{tabular}
\end{center}
\label{table:rs2}
\end{table}
%Notice that the element $\bx^{1\hone;2\cdot\htwo}$ in Table~\ref{table:rs2} has overlapping integration in direction 1 and 2, 
\end{hypothesis}

\begin{remark}
As mentioned in~\eqref{a3}, one of the main technical challenges of 2-$d$ rough integration is the fact that overlapping integrals appears in direction $1$ and $2$ when expanding rough integrals. In Table~\ref{table:rs2} we have included all the overlapping increments needed for the definition of our rough-young integrals (those involving a $\cdot$ in their exponents).
As in Remark \ref{Rmk:rs-regularization}, let us make the definition of one of those elements more explicit. Namely the simplest of those elements, that is $\bx^{11;0\cdot2}$, is such that for a regularization $x^n$ of $x$ we have $\bx^{11;0\cdot2}=\lim_n \bx^{n;11;0\cdot2}$, in $\cp_{2,2}^{2\ga_1;\ga_2}$, where $\bx^{n;11;0\cdot2}$ is defined by 
\begin{equation}\label{E:def-bx-1102}
\bx^{n;11;0\cdot2}_{s_1s_2;t_1t_2}=\int_{s_1<\sigma_1<\sigma_2<s_2}d_1 x^n_{\sigma_1;t_1}\int_{t_1}^{t_2}d_{12}x^n_{\sigma_2;\tau_1},
\end{equation}
as a Riemann-Stieljes integral. Similar statements hold true for the other elements of Table~\ref{table:rs2}.
Note again that the notation $\cdot$ in $0\!\cdot\!2$ means that we are starting a new integral in direction~$2$ while we continue integrating in direction $1$.
\end{remark}

\begin{remark}\label{rmk:geometric}
Our geometric assumption on $\XX$ will manifest when we have to compute operators of the form $\doo$, $\dt$, $\der$ to elements in Table \ref{table:rs2}. Let us spell out an example here: if one considers the element $\bx^{11;2\cdot 2}$, then applying rules like Proposition \ref{prop:dif-intg-1} in direction $1$ only we get 
\begin{equation}\label{E:d1x^(11;2.2)}
\doo \bx^{11;2\cdot 2}=\bx^{1\cdot1;2\cdot2}=\bx^{1;2}\bx^{1;2}.
\end{equation}
Alternatively, we let the reader check that \eqref{E:d1x^(11;2.2)} holds true thanks to tedious elementary computations. In the sequel we will make extensive use of relations like \eqref{E:d1x^(11;2.2)} without further mention.
\end{remark}

The main aim of this section is to prove that a controlled field, in the sense given by Definition~\ref{def:controlled-process}, can be properly integrated with respect to a rough sheet. Our findings are summarized below.
\begin{theorem}\label{th:main-yr-integral}
Let $x\in \mathcal{P}_{1,1}^{\ga_1,\ga_2}$ with $1/2<\ga_1<1/3$, $\ga_2>1/2$ and  $\varphi\in C^4(\mathbb R)$.
%Recall that the increment $z^{1}$ is given by~\eqref{d1}. 
    With our Hypotheses \ref{hyp:rs1} and \ref{hyp:rs2} in mind,  let $\zeta$ be a field controlled by $x$ as in Definition \ref{def:controlled-process}. Then
\begin{enumerate}
\item Referring to Tables \ref{table:rp} and \ref{table:rs2} for the increments built out of $x$, and recalling that the operators $\Lambda_1$, $\Lambda_2$ are given by Proposition \ref{proposition:planar Lambda}, the increments 
\begin{eqnarray}
\mathbf{z}
&=& 
\lc  \id - \laa_1 \doo \rc \lc  \id - \laa_2 \dt \rc(\zeta \, \bx^{1;2} + \zeta^{1}\,\bx^{11;0\cdot2}),\label{E:z1-rough-young} \\
\hat{\mathbf{z}} 
&=& 
\lc  \id - \laa_1 \doo \rc \lc  \id - \laa_2 \dt \rc(\zeta \, \bx^{\hone;\htwo} + \zeta^{1}\,\bx^{1\hone;0\cdot\htwo}),\label{E:z2-rough-young}
\end{eqnarray}
are well defined in $\cp_{2,2}^{\ga_1,\ga_2}$.
\item Whenever $x$ is smooth, Hypothesis \ref{hyp:rs1} is satisfied. In this case, the increments $\bz$, $\hat{\bz}$ in~\eqref{E:z1-rough-young}-\eqref{E:z2-rough-young} coincide with the Riemann-Stieltjes integrals  \begin{align}\label{z-z-hate}
{\mathbf{z}}=\int_1\int_2\,\zeta\,d_{12}x\quad \text{and}\quad\hat{\mathbf{z}}=\int_1\int_2\,\zeta\,d_{\hone\htwo}x .
\end{align}
\item  Let $\{\pi_n^1,\pi_n^2;n \ge 1\}$ be two sequences of partitions of $[s_1, s_2]$, $[t_1, t_2]$ whose mesh go to $0$ as $n\rightarrow \infty$. Under the general Hypothesis \ref{hyp:rs1}, we have the convergence of the following modified Riemann sums:
\begin{eqnarray}
\label{Eq:riemann z^1}
\lim_{n \to \infty} 
\sum_{\pi_n^1, \pi_n^2}
\big[\zeta_{\sigma_i, \tau_j} \,
\bx^{1,2}_{\sigma_i \sigma_{i+1}; \tau_j \tau_{j+1}}+\zeta^{1}_{\sigma_i, \tau_j} \,
\bx^{11;0\cdot2}_{\sigma_i \sigma_{i+1}; \tau_j \tau_{j+1}}\big]
&=&
 \mathbf{z}_{s_1 s_2; t_1 t_2} \\
\label{Eq:riemann z^2}
\lim_{n \to \infty} 
\sum_{\pi_n^1, \pi_n^2}
\big[\zeta_{\sigma_i, \tau_j} \,
\bx^{\hone,\htwo}_{\sigma_i \sigma_{i+1}; \tau_j \tau_{j+1}}+\zeta^{1}_{\sigma_i, \tau_j} \,
\bx^{1\hone;0\cdot\htwo}_{\sigma_i \sigma_{i+1}; \tau_j \tau_{j+1}}\big]
&=& 
\hat{\mathbf{z}}_{s_1 s_2; t_1 t_2}.
\end{eqnarray}
\end{enumerate}
\end{theorem}

\begin{remark}\label{Rmk:control2}
Owing to Remark \ref{Rmk:control1} and Definition \ref{def:controlled-process}, when $y=y^{(i)}\equiv \vp^{(i)}(x)$ with $i\ge 0$ we have $\zeta=y^{(i)}$ and $\zeta^1=y^{(i+1)}$ in relations \eqref{E:z1-rough-young}-\eqref{E:z2-rough-young}. 
\end{remark}

\begin{proof}[Proof of Theorem \ref{th:main-yr-integral}]
We will prove (i), (ii), and (iii)   for a smooth path $x$. The general case $x\in \cp_{1,1}^{\ga_1,\ga_2}$ is then obtained by a regularization procedure which is now standard (see e.g. \cite[Theorem 5.4]{GT} for additional details). 
Let us thus assume that $x$ is a smooth field defined on $[0,1]^2$. Then the field $\bz$ given by \eqref{z-z-hate} is well defined in the Riemann sense. According to an elementary identity which can be found e.g in \cite[p. 16]{CT}, the following formula holds for the increments of $z$:
\begin{multline}\label{CT}
\bz_{s_1s_2; t_1t_2} \,=\, \int_{s_1}^{s_2}\int_{t_1}^{t_2} \zeta_{\sigma;\tau}\,d_{12}x_{\sigma;\tau}= \zeta_{s_1;t_1} \delta x_{s_1s_2; t_1t_2} + \int_{s_1}^{s_2} \delta_1 \zeta_{s_1\sigma_1; t_1}\int_{t_1}^{t_2} d_{12} x_{\sigma_1; \tau_1} \\
+ \int_{t_1}^{t_2} \delta_2 \zeta_{s_1; t_1\tau_1} \int_{s_1}^{s_2}d_{12} x_{\sigma_1; \tau_1} + \int_{s_1}^{s_2} \int_{t_1}^{t_2} \delta \zeta_{s_1\sigma_1; t_1\tau_1} d_{12} x_{\sigma_1,\tau_1}
\end{multline}
For notational sake, we now drop the indices from the above relation. We get 
\begin{align}\label{E:rep:2D-integral}
\mathbf{z} &\, =\, \zeta \, \der x + \int_{1} d_1 \zeta \int_{2}  d_{12} x + \int_{2} d_2 \zeta \int_{1}  d_{12} x + \int_{1}\int_{2} d_{12}\zeta \, d_{12} x \nonumber  \\ 
 &:=\, \zeta \, \mathbf{x}^{1;2}+ a^{11;02} + b^{01;22} + \rho^{11;22}.
\end{align}
Our aim is now to express the right hand side of~\eqref{E:rep:2D-integral} in terms of increments of the field $x$. We will treat each term separately.

\smallskip
\noindent
\emph{Step 1: Term $a^{11;02}$}. For $a^{11;02}$ as given in \eqref{E:rep:2D-integral}, we first use the controlled path decomposition~\eqref{E:controlled-(2,0)reminder} for $\zeta$. We get 
\begin{align*}
a^{11;0\cdot2}\ =\ \int_1d_1\zeta\int_2d_{12}x
\ =\ \zeta^1\int_1 d_1x\int_2 d_{12}x+\int_1\br\int_2d_{12}x.
\end{align*}
According to our Table \ref{table:rp}, we can recast this relation as 
\begin{align}\label{E:a^(11;00)-rep}
a^{11;0\cdot2}=\zeta^{1}\bx^{11;0\cdot2}+ a^{111;0\cdot2}, \quad\text{where \ $a^{111;0\cdot2}:=\int_1 \br\int_2 d_{12}x$.}
\end{align} 
Note that $a^{11;0\cdot2}$ is a second order remainder in the variable $s$. Therefore it is natural to further decompose it by computing $\doo a^{111;02}$. 
To this aim, recall the decomposition \eqref{E:controlled-(2,0)reminder} for $\br$, so that 
\begin{equation}\label{E:a^(111;02)}
a^{111;0\cdot2}=\int_1(\doo \zeta-\zeta^1 \bx^{1;0})\int_2d_{12}x=\int_1 \doo \zeta\int_2d_{12}x-\zeta^1\bx^{11;0\cdot2},
\end{equation}
where we have used Table \ref{table:rp} to write $\int_1 \bx^{1;0}d_{12}x=\bx^{11;0\cdot2}$. We now apply $\delta_1$ on both sides of \eqref{E:a^(111;02)}. That is we apply Proposition \ref{prop:dif-intg-1} to handle terms of the form $\int_1d_{1}f \, d_1g$, as well as Proposition \ref{prop:difrul} for products of increments. In adddition we invoke the fact that $\XX$ is geometric to write 
\begin{equation}\label{E:d1x^(11;02)}
\doo \bx^{11;0\cdot2}=\bx^{1;0}\bx^{1;2}.
\end{equation}
We end up with 
\begin{equation}
\doo a^{111;02}=\doo\zeta \, \bx^{1;2}+\doo \zeta^1\bx^{11;0\cdot2}-\zeta^1\bx^{1;0}\bx^{1;2}.
\end{equation}
It is now easily seen from \eqref{E:controlled-(2,0)reminder} that the above formula simplifies into 
\begin{equation}\label{E:do-a^(111;02)}
\delta_1 a^{111;02}=\mathbf{r}\,\bx^{1;2}+\delta_1 \zeta^{1}\bx^{11;0\cdot2}.
\end{equation}
Since we have supposed that $\br\in \cp^{2\ga_1,*}$ in Definition \ref{def:controlled-process}, and owing to the regularities assumed in Hypothesis \ref{hyp:rs1}, relation \eqref{E:do-a^(111;02)} proves in particular that $\delta_1 a^{111;02}$ lies in $\cp_{3,2}^{3\ga_1,\ga_2}$. As $3\ga_1>1$, one can apply the sewing map $\Lambda_1$ given in Proposition \ref{proposition:planar Lambda}. We obtain 
$$
a^{111;02}=\Lambda_1\big(\mathbf{r}\,\bx^{1;2}+\delta_1 \zeta^{1}\bx^{11;0\cdot2}\big).
$$
Plugging this relation into \eqref{E:a^(11;00)-rep}, we get 
\begin{align}\label{E:a^(11;02)-Lambda1}
a^{11;02}=\zeta^1\bx^{11;0\cdot2}+\Lambda_1\Bigl(\br \, \bx^{1;2}+ \delta_1 \zeta^{1}\bx^{11;0\cdot2}\Bigl).
\end{align}
On the other hand, using \eqref{eq:difrul-C1-C2} and the fact that $\doo \bx^{11;0\cdot2}=\bx^{1;0}\bx^{1;2}$, we obtain 
\begin{align*}
\delta_1(\zeta\, \mathbf{x}^{1;2}+ \zeta^{1}\, \bx^{11;0\cdot2})
&=-\delta_1 \zeta \, \bx^{1;2} -\delta_1 \zeta^{1} \bx^{11;0\cdot2}+ \zeta^{1} \bx^{1;0} \, \bx^{1;2}.
\end{align*}
Therefore our assumption \eqref{E:controlled-(2,0)reminder} for the increment of a controlled process $\zeta$ yields
\begin{align*}
\delta_1(\zeta\, \mathbf{x}^{1;2}+ \zeta^{1}\, \bx^{11;0\cdot2})=-\mathbf{r} \, \bx^{1;2}-\delta_1 \zeta^{1}\bx^{11;0\cdot2}.
\end{align*}
Comparing this expression with the right hand side of \eqref{E:a^(11;02)-Lambda1}, we get 
\begin{align}\label{E:a^(11;02)-Lambda1-delta1}
a^{11;02}=\zeta^1\bx^{11;0\cdot2}-\Lambda_1\delta_1\Bigl(\zeta\, \mathbf{x}^{1;2}+ \zeta^{1}\, \bx^{11;0\cdot2}\Bigl).
\end{align}

\smallskip
\noindent
\emph{Step 2: Term $b^{01;22}$}. Recall that $b^{01;22}$ is defined in \eqref{E:rep:2D-integral} by $b^{01;22}= \int_{2} d_2 \zeta \int_{1}  d_{12} x$. We apply $\delta_2 $ on both sides of this identity. For the right hand side, we observe that the integral in direction 2 is of the form $\int_2 d_2f d_2g$. Hence applying Proposition \ref{prop:dif-intg-1} we get 
\begin{align}\label{E:dt-b^(01;22)}
\dt b^{01;22} =  \dt \zeta \,\bx^{1;2}
\end{align}
It is easily seen that the right hand side of \eqref{E:dt-b^(01;22)} has the regularity $\cp_{2,3}^{\ga_1,2\ga_2}$. Since we have assumed $\ga_2>1/2$, one can apply the operator $\Lambda_2$ in order to get 
\begin{equation}
    b^{01;22}=\Lambda_2\bigl(\dt \zeta \,\bx^{1;2}\bigl).
\end{equation}

\smallskip
\noindent
\emph{Step 3: Term $\rho^{11;22}$}. The term $\rho^{11;22}$ in \eqref{E:rep:2D-integral} needs an additional decomposition. Namely start by recalling from \eqref{E:rep:2D-integral} that 
\begin{equation*}
\rho^{11;22}= \int_{1}\int_{2} d_{12}\zeta \, d_{12} x = \int_{1}\int_{2}\delta \zeta \, d_{12} x.
\end{equation*}
Next we apply our assumption \eqref{E:controlled-(2,1)reminder} on $\der \zeta$, which gives
\begin{align}\label{E:rho^(11;22)}
\rho^{11;22}\ &=\ \zeta^{1}\, \bx^{11;22}+ \int_2 d_2\zeta^1\int_1 d_1 x\, d_{12} x+\int_{1}\int_{2} \der_2 \br \,d_{12} x \ =\ b^{11;22}_1 + b^{11;22}_2+\rho^{111;22}.
\end{align}
Observe that the right hand side of~\eqref{E:rho^(11;22)} should have regularity $2\ga_2$ in the variable 2. Since $\ga_2 >1/2$, it is thus natural to apply $\dt$ to both sides of this equation. We will do so by treating the terms in the right hand side of \eqref{E:rho^(11;22)} separately. To begin with, one resorts to the geometricity of $\XX$ and the definition in Table \ref{table:rs2} to write 
\begin{align}
\dt \bx^{11;22}\ =\ \dt \Bigl(\int_1\int_2 d_{12}xd_{12}x\Bigl)\ =\ \int_1\int_2d_{12}x \int_2 d_{12}x\ =\ \bx^{11;2\cdot2}.
\end{align}
Thus we invoke Proposition \ref{prop:dif-intg-1}, which yields 
\begin{equation}\label{E:b1^(11;22)}
\dt \bigl(b_1^{11;22}\bigl)=-\dt \zeta^1\bx^{11;22}+ \zeta^1\bx^{11;2\cdot2}.
\end{equation}

The term $b_2^{11;22}$ requires a special attention. Indeed, by definition of $b_2^{11;22}$ in~\eqref{E:rho^(11;22)} we have
\begin{align}\label{E:b2-def}
(b_2^{11;22})_{s_1s_2;t_1t_2}:=\int_{t_1}^{t_2}\delta_2 \zeta^1_{s_1;t_1\tau}\int_{s_1}^{s_2}\delta_1 x_{s_1\sigma;\tau}d_{12}x_{\sigma;\tau}.
\end{align}
Next since the right hand side of~\eqref{E:b2-def} has some glued $s$ variables, let us ignore for the moment the $s$ variables therein (as well as the integrals with these variables). We get an expression of the form 
\begin{equation}
(\widetilde{b}_2^{11;22})_{s_1s_2;t_1t_2}=\int_{t_1}^{t_2}\dt f_{t_1\tau} \, g_{\tau} \, d_2h_{\tau},
\end{equation}
where we have set $f_{\tau}=\zeta^1_{s_1;\tau}$, $g_{\tau}=\doo x_{s_1\sigma_1;\tau}$, and $h_{\tau}=d_{\sigma;\tau}$. In addition, observe that $\widetilde{b}^{11;22}$ can also be written $\widetilde{b}_2^{11;22}=\int_2\dt fd_2\ell$, with $\ell:=\int_2gd_2h$. Therefore applying Proposition \ref{prop:dif-intg-1} we get $\dt \widetilde{b}_2^{11;22}=\dt f \, \dt \ell$. Putting the $s$ variables and their integrals back, we obtain 
\begin{align*}
(\delta_2b_2^{11;22})_{s_1s_2;t_1t_2t_3}
=\delta_2 \zeta^1_{s_1;t_1t_2}\int_{t_2}^{t_3}\int_{s_1}^{s_2}\delta_1 x_{s_1\sigma;\tau} \, d_{12}x_{\sigma;\tau}=\dt\zeta^1_{s_1;t_1t_2}\bx^{11;02}_{s_1s_2;t_1t_2},
\end{align*}
where we recall that the iterated integral $\bx^{11;02}$ is defined in Table~\ref{table:rs2}. 
Moreover,  $\bx^{11;02}$ can be decomposed as follows
\begin{align}\label{E:x^(11;02)-x^(11;0.2)}
\bx^{11;02}_{s_1s_2;t_2t_3}=\int_{s_1}^{s_2}\int_{t_2}^{t_3}\delta_1 x_{s_1\sigma;\tau} \, d_{12}x_{\sigma;\tau} 
&=\int_{s_1}^{s_2}\int_{t_2}^{t_3}\delta x_{s_1\sigma;t_2\tau} \, d_{12}x_{\sigma;\tau}+\int_{s_1}^{s_2}\delta_1 x_{s_1\sigma;t_2}\int_{t_2}^{t_3}d_{12}x_{\sigma;\tau}\nonumber\\
&= \bx^{11;22}_{s_1s_2;t_2t_3}+\bx^{11;0\cdot2}_{s_1s_2;t_2t_3} .
\end{align}
Therefore
\begin{equation}\label{E:b2^(11;22)}
\delta_2 b^{11;22}_2=\delta_2 \zeta^{1}\bx^{11;22}+\delta_2 \zeta^{1}\bx^{11;0\cdot2}.
\end{equation}
Now, combine \eqref{E:dt-b^(01;22)}, \eqref{E:b1^(11;22)} and \eqref{E:b2^(11;22)}, we get
\begin{align}\label{E:b-b1-b2}
\dt(b^{01;22} +  b_1^{11;22}+b_2^{11;22})= \dt \zeta \,\bx^{1;2}+\delta_2 \zeta^{1}\bx^{11;0\cdot2}+ \zeta^1\bx^{11;2\cdot2}.
\end{align}
On the other hand, start from the expression of $\bx^{11;0\cdot2}$ in \eqref{E:x^(11;02)-x^(11;0.2)}, which reads $\bx^{11;02}=\bx^{11;22}+\bx^{11;0\cdot2}$, and observe that 
\begin{equation}\label{e2}
\bx^{11;02}_{s_1s_2;t_1t_2}=\Bigl(\int_1\int_2d_1xd_{12}x\Bigl)_{s_1s_2;t_1t_2}=\int_{s_1}^{s_2}\int_{t_1}^{t_2}\doo x_{s_1\sigma;\tau}d_{12}x_{\sigma;\tau}=\dt\Bigl(\int_{s_1}^{s_2}\int_{0}^{\cdot}\doo x_{s_1\sigma;\tau}d_{12}x_{\sigma;\tau}\Bigl)_{t_1t_2}.
\end{equation}
It is clear from the right hand side of \eqref{e2} that $\delta_2\bx^{11;02}=0$. Hence taking into account the above relation $\bx^{11;02}=\bx^{11;22}+\bx^{11;0\cdot2}$ we obtain
\begin{equation}\label{E:d2-x^(11;0.2)}
\dt \bx^{11;0\cdot2}= - \dt \bx^{11;22} =- \bx^{11;2\cdot2} \, ,
\end{equation}
where $\bx^{11;2\cdot2}$ is defined in Table \ref{table:rs2}. 
Combining with \eqref{eq:difrul-C1-C2}, one can compute
\begin{equation}\label{E:d2-zet-x}
\delta_2(\zeta\, \mathbf{x}^{1;2}+ \zeta^{1}\, \bx^{11;0\cdot2})=-\delta_2 \zeta\,\bx^{1;2}-\delta_2 \zeta^{1}\bx^{11;0\cdot2}-\zeta^{1}\bx^{11;2\cdot2}.
\end{equation}
We observe that the right hand side of \eqref{E:b-b1-b2} and \eqref{E:d2-zet-x} are the same (with flipped sign). Moreover, it is easily seen that the rhs of \eqref{E:d2-zet-x} has regularity $2\ga_2$ in direction 2. Since $\ga_2>1/2$, we conclude that  
%Hence, combining \eqref{E:b-b1-b2} and \eqref{E:d2-zet-x} we obtain 
\begin{equation}\label{E:sum-bs}
b^{01;22} +  b_1^{11;22}+b_2^{11;22}=-\Lambda_2\delta_2(\zeta\, \mathbf{x}^{1;2}+ \zeta^{1}\, \bx^{11;0\cdot2}).
\end{equation}

\smallskip
\noindent
\emph{Step 4: Term $\rho^{111;22}$}. The term Term $\rho^{111;22}$ has to be considered as a remainder term in both directions $1$ and $2$. It will thus be decomposed by computing $\der\rho^{111;22}$. To this aim, we first recall from \eqref{E:rho^(11;22)} that
%By the definition of $\rho^{111;22}$ in \eqref{E:rho^(11;22)}, we have 
\begin{equation*}
\rho^{111;22}= \int_{1}\int_{2} \der_2 \br \,d_{12} x,
\end{equation*}
where we recall that $\br$ is an element of $\br\in\mathcal{P}_{2,1}^{2\gamma_1,*}$ such that $\der_2 \br\in \mathcal{P}_{2,2}^{2\gamma_1,\ga_2}$.
If one considers the variable $2$ only, the above expression is of the form $\int d_2fd_2g$ for two functions $f,g$ of the $t$-variables. Hence applying Proposition \ref{prop:dif-intg-1} we get
%Then $\delta_2\rho^{111;22}$ can be expressed as 
\begin{align}\label{E:dt-rho^(111;22)}
\dt\rho^{111;22}=\int_{1}\int_{2} d_2 \br \,\int_2 d_{12} x.
\end{align}
Furthermore, since $\br$ is a function of 1 variable only in the second direction, we have $\int_2 d_2\br=\delta_{2}\br$. Thus using the controlled structure of $\zeta$ in \eqref{E:controlled-(2,1)reminder}, we can write 
\begin{equation*}
\int_2 d_2\br=\delta \zeta -\zeta^{1}\,\bx^{1;2}-\delta_2 \zeta^{1} \bx^{1;0}=\int_1\int_2d_{12}\zeta-\zeta^1\int_1\int_2d_{12}x-\dt\zeta^1\int_1d_{1}x,
\end{equation*}
where the indices for e.g the last term above are written as $\dt\zeta^1_{s_{1};t_{1}t_{2}}(\int_1d_{1}x)_{s_{1}\si;t_{2}}$. Hence plugging this into \eqref{E:dt-rho^(111;22)}, we obtain
\begin{align}\label{E:dt-rho^(111;22)2}
\dt\rho^{111;22}&=\int_{1}\int_{2} d_{12} \zeta \,\int_2 d_{12} x- \zeta^1\int_{1}\int_{2} d_{12}x \,\int_2 d_{12} x-\dt\zeta^1\int_1d_1x\int_2d_{12}x\notag\\
&=\int_{1}\int_{2} d_{12} \zeta \,\int_2 d_{12} x-\zeta^1\bx^{11;2\cdot2}-\dt \zeta^1\bx^{11;0\cdot2}.
\end{align}
We now apply $\doo$ to both sides of \eqref{E:dt-rho^(111;22)2}. With the same considerations as for \eqref{E:dt-rho^(111;22)}, namely considering the $1$-variable only and applying Proposition \ref{prop:dif-intg-1}, and relying on decompositions like \eqref{E:d1x^(11;2.2)} in Remark \ref{rmk:geometric}, we get
\begin{align}
\der \rho^{111;22}&=\int_{1}\int_{2} d_{12} \zeta \,\int_1\int_2 d_{12} x
+ \doo\zeta^1\bx^{11;2\cdot2}
-\zeta^1\bx^{1;2}\bx^{1;2}+\der\zeta^1\bx^{11;0\cdot2}-\dt \zeta^1\bx^{1;0}\bx^{1;2}\nonumber\\
&=\,\der \zeta\,\bx^{1;2}+\doo\zeta^1\bx^{11;2\cdot2}-\zeta^1\bx^{1;2}\bx^{1;2}+\der\zeta^1\bx^{11;0\cdot2}-\dt\zeta^1\bx^{1;0}\bx^{1;2}.
%&=\dt\br\,\bx^{1;2}+\doo\zeta^1\bx^{11;2\cdot2}+\der\zeta^1\bx^{11;0\cdot2}.
\end{align}
Invoking decomposition \eqref{E:controlled-(2,1)reminder} for $\der \zeta$ and looking for cancellations, we end up with 
\begin{equation*}
\der \rho^{111;22}=\dt\br\,\bx^{1;2}+\doo\zeta^1\bx^{11;2\cdot2}+\der\zeta^1\bx^{11;0\cdot2}.
\end{equation*}
Furthermore, under our assumptions on the rough sheet $\XX$ and on the controlled path $\zeta$ in Definition~\ref{def:controlled-process}, we have that the right hand side above is in $\cp_{3,3}^{3\ga_1;2\ga_2}$. Hence applying the operator $\Lambda$ to both sides gives
\begin{align}\label{E:rho^(111;22)-Lambda}
\rho^{111;22}=\Lambda\Bigl(\dt\br\,\bx^{1;2}+\doo\zeta^1\bx^{11;2\cdot2}+\der\zeta^1\bx^{11;0\cdot2}\Bigl).
\end{align}

Our next objective is to write the right hand side of \eqref{E:rho^(111;22)-Lambda} in terms of $\Lambda\delta$ applied to a regular increment.
In view of~\eqref{E:a^(11;02)-Lambda1-delta1} and~\eqref{E:d2-zet-x}, it seems natural to
apply $\doo$ to both sides of~\eqref{E:d2-zet-x}. To this aim, we resort again to the product rules in Proposition \ref{prop:difrul}. Moreover, as mentioned in Remark \ref{rmk:geometric}, the geometric property of $\XX$ implies that  $\doo \bx^{11;0\cdot2}=\bx^{1;0}\bx^{1;2}$, and $\doo \bx^{11;2\cdot2}=\bx^{1;2}\bx^{1;2}$. Applying those elementary rules, we obtain 
%On the other hand, applying $\doo$ to both sides of \eqref{E:d2-zet-x} and using the facts that $\doo\bx^{11;0\cdot2}=\bx^{1;0}\bx^{1;2}$ and $\doo\bx^{11;2\cdot2}=\bx^{1;2}\bx^{1;2}$, we obtain 
\begin{align}\label{E:d12-zet-x}
\delta(\zeta\, \mathbf{x}^{1;2}+ \zeta^{1}\, \bx^{11;0\cdot2})&=\doo\Bigl(-\delta_2 \zeta\,\bx^{1;2}-\delta_2 \zeta^{1}\bx^{11;0\cdot2}-\zeta^{1}\bx^{11;2\cdot2}\Bigl)\notag\\
&=\,\der \zeta\,\bx^{1;2}+\der\zeta^1\bx^{11;0\cdot2}-\dt\zeta^1\doo\bx^{11;0\cdot2}+\doo\zeta^1\bx^{11;2\cdot2}-\zeta^1\doo\bx^{11;2\cdot2} \notag\\
&=\,\der \zeta\,\bx^{1;2}+\der\zeta^1\bx^{11;0\cdot2}-\dt\zeta^1\bx^{1;0}\bx^{1;2}+\doo\zeta^1\bx^{11;2\cdot2}-\zeta^1\bx^{1;2}\bx^{1;2} 
\end{align}
Hence owing to \eqref{E:controlled-(2,1)reminder} and looking for cancellations, we get 
\begin{equation}\label{E:d12-zet-x2}
\delta(\zeta\, \mathbf{x}^{1;2}+ \zeta^{1}\, \bx^{11;0\cdot2}) =\, \dt\br\,\bx^{1;2}+\doo\zeta^1\bx^{11;2\cdot2}+\der\zeta^1\bx^{11;0\cdot2}.
\end{equation}
Combining \eqref{E:rho^(111;22)-Lambda} and \eqref{E:d12-zet-x2} we have 
\begin{equation}\label{E:rho^(111;22)-L-der}
\rho^{111;22}=\Lambda\delta(\zeta\, \mathbf{x}^{1;2}+ \zeta^{1}\, \bx^{11;0\cdot2}).
\end{equation}
\noindent
\emph{Step 5: Conclusion}. Let us summarize our considerations so far. We first wrote a decomposition \eqref{E:rep:2D-integral} for $\bz$. Then we combined some of the terms in this decomposition to get 
\begin{equation}\label{E:z-expansion}
\bz=\zeta\bx^{1;2}+ a^{11;02} + b^{01;22}+b_1^{11;22} + b_2^{11;22}+ \rho^{111;22},
\end{equation}
where this expression stems from \eqref{E:rho^(11;22)}. Then applying \eqref{E:a^(11;02)-Lambda1-delta1}, \eqref{E:sum-bs} and \eqref{E:rho^(111;22)-L-der} to relation \eqref{E:z-expansion}, we have shown
\begin{multline}
\bz =\zeta\bx^{1;2}+\zeta^1\bx^{11;0\cdot2}
-\Lambda_1\delta_1\bigl(\zeta\bx^{1;2}+\zeta^1\bx^{11;0\cdot2}\bigl) \\
-\Lambda_2\delta_2\bigl(\zeta\bx^{1;2}+\zeta^1\bx^{11;0\cdot2}\bigl) 
+\Lambda\delta\bigl(\zeta\bx^{1;2}+\zeta^1\bx^{11;0\cdot2}\bigl),
\end{multline}
from which ou claim \eqref{E:z1-rough-young} is easily derived. We let the reader check that \eqref{E:z2-rough-young} can be proved along the same lines, using the a priori increments $\bx^{1\hone;0\htwo}$, $\bx^{1\hone;0\cdot\htwo}$, $\bx^{1\hone;2\htwo}$ and $\bx^{1\hone;2\cdot\htwo}$ from Table~\ref{table:rs2}.

The other items in Theorem \ref{th:main-yr-integral} are now easily seen. Item (ii) is a direct consequence of our method of proof, since we started from a smooth regularization $x^n$ of $x$ in order to get \eqref{E:z1-rough-young}-\eqref{E:z2-rough-young}. The Riemann sum convergences in \eqref{Eq:riemann z^1}-\eqref{Eq:riemann z^2} are a direct consequence of Proposition \ref{proposition:integ-2d}.
\end{proof}

\subsection{Change of variable formula}

The change-of-variable proof will invoke two further iterated integrals, which arise from the decomposition of $\bx^{1\hone;0\cdot\htwo}$ established in identity~\eqref{Eq:x^{1 hone;0.htwo}-dcp} below. Unlike the increments collected in Hypotheses~\ref{hyp:rs1} and~\ref{hyp:rs2}, these two elements are not assumed to be part of the rough sheet $\XX$. Instead, as the following lemma shows, they can be constructed from the data already provided by Hypotheses~\ref{hyp:rs1} and~\ref{hyp:rs2}.

\begin{lemma}\label{hyp:rs3}
Under Hypotheses~\ref{hyp:rs1} and~\ref{hyp:rs2}, the following increments are well defined:
\begin{table}[htp]
\caption{Iterated integrals constructed from $\XX$ for the change-of-variable formula}
\begin{center}
\begin{tabular}{|c|c|c||c|c|c|}
\hline
\emph{Increment} & \emph{Interpretation} & \emph{Regularity} & \emph{Increment} & \emph{Interpretation} & \emph{Regularity} \\
\hline
$\bx^{111;002}$ & $\int_{1} d_{1}x\,d_{1}x\int_{2} d_{12}x$ & $(3\ga_{1},\ga_{2})$ &
$\bx^{110;022}$ & $\int_{1} d_{1}x\int_{2} d_{12}x\,d_{2}x$ & $(2\ga_{1},2\ga_{2})$ \\
\hline
\end{tabular}
\end{center}
\label{table:rs3}
\end{table}
\noindent
\end{lemma}

\begin{proof}
We will prove our claim in two separate steps, for each element of Table~\ref{table:rs3}.

\smallskip
\noindent 
\emph{Step 1: Definition of $\bx^{111;002}$.} A slight elaboration of Proposition \ref{prop:dif-intg-1} states that, for $f_1,f_2,f_3 \in \cac_1^{\infty}$, we have 
\begin{align}\label{E:ch-rule}
\der\lp
\int df_1\,df_2 df_3\rp = \der f_1 \, \int df_2df_3\, + \, \int df_1df_2\,\der f_3.
\end{align}
For the proof, one can see Proposition 2.7 in \cite{GT}. Using this identity, we obtain 
\begin{align}\label{eq:x111-002}
    \delta_1 \bx^{111;002}&=\doo x \int_{1} d_{1}x\int_{2} d_{12}x \,+\, \int_{1} d_{1}x d_{1}x \int_1\int_{2} d_{12}x= \bx^{1;0}\bx^{11;0\cdot 2}+\bx^{11;00}\bx^{1;2}.
\end{align}
Using Hypotheses \ref{hyp:rs1} and \ref{hyp:rs2} we obtain that the right hand side of \eqref{eq:x111-002} is in $\cp_{3,2}^{3\ga_1;*}$. Since $3\ga_1>1$, we may apply the sewing map $\Lambda_1$ in direction $1$, given in Proposition~\ref{proposition:planar Lambda}, to both sides of \eqref{eq:x111-002}. This yields
\begin{align}\label{eq:x111-002-sewing}
\bx^{111;002}= \Lambda_1\Bigl(\bx^{1;0}\bx^{11;0\cdot 2}+\bx^{11;00}\bx^{1;2}\Bigl).
\end{align}
Since the argument of $\Lambda_1$ in \eqref{eq:x111-002-sewing} only involves increments furnished by Hypotheses~\ref{hyp:rs1} and~\ref{hyp:rs2}, this identity defines $\bx^{111;002}$ entirely in terms of the rough sheet $\XX$ given a priori. The increment $\bx^{111;002}$ is therefore well defined, with the regularity $(3\ga_1,\ga_2)$ recorded in Table~\ref{table:rs3}.

\smallskip
\noindent \emph{Step 2: Definition of $\bx^{110;022}$.}
We follow the same scheme as in Step~1, except that the iterated integral $\bx^{110;022}$ now carries its innermost integration in direction~$2$. Accordingly we start by applying the differentiation rule of Proposition~\ref{prop:dif-intg-1} to $\dt\bx^{110;022}$, which gives
\begin{equation}\label{e3}
\dt \bx^{110;022}= \int_{1} d_{1}x\int_{2} d_{12}x\,\int_2d_{2}x .
\end{equation}
Next we split the pure direction-$2$ integral above as follows, for $s<\sigma<s'$ and $t<t'$:
\begin{equation*}
\left(\int_2 d_2x\right)_{\sigma;t\,t'}:=\dt x_{\sigma;t\,t'}=\dt x_{s';t\,t'}- \der x_{\sigma s';t\,t'},
\quad\text{that is}\quad
\int_2 d_2x=\dt x-\int_2 d_{12}x .
\end{equation*}
Plugging this identity into \eqref{e3}, we get
\begin{align}\label{E:d2-x^(110;022)}
\dt \bx^{110;022}= \int_{1} d_{1}x\int_{2} d_{12}x\,\int_2d_{2}x = \int_{1} d_{1}x\int_{2} d_{12}x\,\cdot \dt x - \int_{1} d_{1}x\int_{2} d_{12}x\,\int_2 d_{12}x.
\end{align}
The first term in the righ hand side above is nothing but $\bx^{11;0\cdot2}\,\bx^{0;2}$. For the second term, applying $\delta_1$ we obtain
\begin{align*}
\doo\Bigl(\int_{1} d_{1}x\int_{2} d_{12}x\,\int_2 d_{12}x\Bigl) &= \doo x \int_1\int_2 d_{12}x\int_2d_{12}x + \int_1 d_1 x \int_2 d_{12}x \cdot \delta x \\
&= \bx^{1;0}\, \bx^{11;2\cdot2} + \bx^{11;0\cdot2}\bx^{1;2}.
\end{align*}
Since the last term is in $\cp_{3,3}^{3\ga_1,*}$, applying $\Lambda_1$ to both sides above gives
\begin{equation*}
\int_{1} d_{1}x\int_{2} d_{12}x\,\int_2 d_{12}x= \Lambda_1\Bigl(\bx^{1;0}\, \bx^{11;2\cdot2} + \bx^{11;0\cdot2}\bx^{1;2}\Bigl).
\end{equation*}
Plugging this back into \eqref{E:d2-x^(110;022)} we have
\begin{align*}
\dt \bx^{110;022}= \int_{1} d_{1}x\int_{2} d_{12}x\,\int_2d_{2}x = \bx^{11;0\cdot2}\,\bx^{0;2} - \Lambda_1\Bigl(\bx^{1;0}\, \bx^{11;2\cdot2} + \bx^{11;0\cdot2}\bx^{1;2}\Bigl).
\end{align*}
The right hand side above is in $\cp_{2,3}^{*,2\ga_2}$. Since $2\ga_2>1$, we may apply the sewing map $\Lambda_2$ in direction $2$, given in Proposition~\ref{proposition:planar Lambda}, to both sides above. This yields
\begin{equation}\label{eq:x110-022-sewing}
\bx^{110;022}
 = \Lambda_2\Bigl( \bx^{11;0\cdot2}\,\bx^{0;2} - \Lambda_{1} \bigl(\bx^{1;0}\, \bx^{11;2\cdot2} + \bx^{11;0\cdot2}\bx^{1;2}\bigl) \Bigl) .
\end{equation}
As in Step~1, the right hand side of \eqref{eq:x110-022-sewing} is built solely from increments furnished by Hypotheses~\ref{hyp:rs1} and~\ref{hyp:rs2}. The increment $\bx^{110;022}$ is therefore well defined from the rough sheet $\XX$ given a priori, with the regularity $(2\ga_1,2\ga_2)$ recorded in Table~\ref{table:rs3}.

Together with \eqref{eq:x111-002-sewing}, this proves that both increments of Table~\ref{table:rs3} can be constructed from the data of Hypotheses~\ref{hyp:rs1} and~\ref{hyp:rs2}, which completes the proof.
\end{proof}

With Lemma~\ref{hyp:rs3} in place, we can now state the main result of this section.

\begin{proposition}\label{prop:ch-var}
Let $x$ be an element of $\mathcal{P}_{1,1}^{\ga_1,\ga_2}$ with $\ga_1>1/3$ and  $\ga_2>1/2$. Like in Theorem \ref{th:main-yr-integral}, we assume that $x$ can be lifted as a rough sheet $\XX$ satisfying Hypotheses \ref{hyp:rs1}, \ref{hyp:rs2}, and~\ref{hyp:rs3}. Consider a function $\vp \in \mathcal{C}_b^5(\RR)$ and for $i\ge 1$ we set $y^{(i)}:=\vp^{(i)}(x)$. Then the following holds true for $y=\vp(x)$:
\begin{enumerate}
\item\label{prop:ch-var-i} The rectangular increment of $y$ can be decomposed as
\begin{equation}\label{eq:ch-var}
\der y = \int_{1}\int_{2} y^{(1)} \, d_{12}x + \int_{1}\int_{2} y^{(2)} \, d_{\hone\htwo}x,
\end{equation}
where the integrals are understood as in \eqref{E:z1-rough-young}-\eqref{E:z2-rough-young}-\eqref{z-z-hate} and Remark \ref{Rmk:control2}.

\item\label{prop:ch-var-ii} Let $\pi_n^1$ and $\pi_n^2$ be two sequences of partitions, respectively of $[s_1,s_2]$ and $[t_1,t_2]$, whose mesh go to $0$ as $n\rightarrow \infty$. Then we have convergrence of the following Riemann sums:
\begin{multline}\label{Eq:riemann-delta(y)}
   \delta y_{s_1s_2;t_1t_2} 
   =\lim_{n \to \infty}\Bigl( 
\sum_{\pi_n^1, \pi_n^2}
\big[y^{(1)}_{\sigma_i, \tau_j} \,
\bx^{1,2}_{\sigma_i \sigma_{i+1}; \tau_j \tau_{j+1}}+y^{(2)}_{\sigma_i, \tau_j} \, \bx^{11;0\cdot2}_{\sigma_i \sigma_{i+1}; \tau_j \tau_{j+1}} \big]\\
+ 
\sum_{\pi_n^1, \pi_n^2}
\big[y^{(2)}_{\sigma_i, \tau_j} \,
\bx^{\hone,\htwo}_{\sigma_i \sigma_{i+1}; \tau_j \tau_{j+1}}+y^{(3)}_{\sigma_i,\tau_i}\bx^{1\hone;0\cdot\htwo}_{\sigma_i \sigma_{i+1}; \tau_j \tau_{j+1}} \big]\Bigl).
\end{multline}
\end{enumerate}
\end{proposition}
\begin{proof}
In Lemmas \ref{L:2D-Taylor-expansion1} and \ref{L:2D-Taylor-expansion2}, we have decomposed the increments of $y$ by looking at $\doo y$ first and then $\dt \doo y$, by means of Taylor expansions. We now adopt the same strategy, using 1$d$ change of variable formulae and starting with $\delta_2$.

\smallskip
\noindent
\emph{Step 1: Decomposition of $\delta y$.} In direction $2$, the sheet $x$ is $\ga_2$-H\"older with $\ga_2>1/2$. Therefore one can use a Young type $1d$ change of variable formula for the increment $\dt y$. In our differential calculus formalism, this can be written as 
\begin{equation}
\dt y= \int_2 y^{(1)}d_2x,
\end{equation}
where the right hand side above is understood thanks to Proposition \ref{prop:young}. More specifically, according to \eqref{E:young-integral} we have 
\begin{equation}\label{E:d2-y}
\delta_2 y= \lc  \id - \laa_2 \dt \rc\bigl(y^{(1)}\delta_2 x \bigl)=\lc  \id - \laa_2 \dt \rc\bigl(y^{(1)}\bx^{0;2} \bigl),
\end{equation}
where $\bx^{0;2}$ is defined in Table \ref{table:rp}. Hence applying $\delta_1$ to both sides of \eqref{E:d2-y} and applying \eqref{eq:difrulu-C1-C1} to $\doo$ inside $\id- \Lambda_2\delta_2$, we get 
\begin{equation}\label{E:d-y}
\delta y= \lc  \id - \laa_2 \dt \rc\bigl(\doo y^{(1)}\bx^{0;2}+y^{(1)}\bx^{1;2} \bigl).
\end{equation}
Next we use the fact that $y^{(1)}$ is a second order controlled process in direction 1. Namely as an elaboration of \eqref{E:controlled-(2,0)reminder} one can write 
\begin{equation}\label{E:2nd order controlled}
\delta_1 y^{(1)}=y^{(2)}\bx^{1;0}+ y^{(3)}\bx^{11;00} +\br^{(1)}_1,
\end{equation}
where $\br^{(1)}_1\in \cp_{2,1}^{3\ga_1;*}$. Plugging  this identity into \eqref{E:d-y} we obtain 
\begin{equation}\label{eq:dt-y-expan1}
\delta y =\lc  \id - \laa_2 \dt \rc\Bigl(y^{(1)}\bx^{1;2}+y^{(2)}\bx^{1;0}\bx^{0;2}+ y^{(3)}\bx^{11;00}\bx^{0;2}+ \br_1^{(1)}\bx^{0;2}\Bigl).
\end{equation}
We now wish to recast the expression in the right hand side of \eqref{eq:dt-y-expan1} as $\id -\Lambda_1\delta_1$ applied to an increment $h$. To this aim, let us observe the following: 
\begin{enumerate}
    \item If one considers the term $\br^{(1)}\bx^{0;2}$ in the right hand side of \eqref{eq:dt-y-expan1}, we have $\br^{(1)}\bx^{0;2}\in\cp_{2,2}$ and
\begin{align}\label{E:factor1}
\delta_1\bigl(\br_1^{(1)}\bx^{0;2}\bigl)=
%\delta_1\bigl(\delta_1 y^{(1)}\bx^{0;2}-y^{(2)}\bx^{1;0}\bx^{0;2}- y^{(3)}\bx^{11;00}\bx^{0;2}\bigl)\nonumber\\
\delta_1 \br_1^{(1)}\bx^{0;2}+\br_1^{(1)}\bx^{1;2}.
\end{align}
In addition, it is readily checked from \eqref{E:2nd order controlled} and \eqref{E:factor1} that both $\br^{(1)}\bx^{0;2}$ and $\doo\bigl(\br^{(1)}\bx^{0;2}\bigl)$ have regularity $(3\ga_1,*)$. Hence one can apply the sewing map $\Lambda_1$ to write 
\begin{equation*}
\br^{(1)}\bx^{0;2}=\Lambda_{1}\bigl(\delta_1 \br_1^{(1)}\bx^{0;2}+\br_1^{(1)}\bx^{1;2}\bigl).
\end{equation*}
\item If we apply $\delta_1$ to the remaining terms in the right hand side of \eqref{eq:dt-y-expan1}, we obtain 
\begin{multline*}
\delta_1\bigl(y^{(1)}\bx^{1;2}+y^{(2)}\bx^{1;0}\bx^{0;2}+ y^{(3)}\bx^{11;00}\bx^{0;2}\bigl) \
=-\delta_1y^{(1)}\bx^{1;2}-\delta_1y^{(2)}\bx^{1;0}\bx^{0;2} \\
+ y^{(2)}\bx^{1;0}\bx^{1;2}
-\delta_1y^{(3)}\bx^{11;00}\bx^{0;2}+y^{(3)}\bx^{1;0}\bx^{1;0}\bx^{0;2}+y^{(3)}\bx^{11;00}\bx^{1;2} .
\end{multline*}
This can be simplified into
\begin{equation}\label{E:d1-sum}
\bigl(-\doo y^{(1)}+y^{(2)}\bx^{1;0}+y^{(3)}\bx^{11;00}\bigl)\bx^{1;2}+\bigl(-\doo y^{(2)}\bx^{1;0}-\doo y^{(3)}\bx^{11;00}+y^{(3)}\bx^{1;0}\bx^{1;0}\bigl)\bx^{0;2}.
\end{equation}
Furthermore, applying $\doo$ on both sides of \eqref{E:2nd order controlled} we get 
\begin{equation}\label{E:factor2} 
\delta_1\bigl(y^{(1)}\bx^{1;2}+y^{(2)}\bx^{1;0}\bx^{0;2}+ y^{(3)}\bx^{11;00}\bx^{0;2}\bigl)=-\delta_1 \br_1^{(1)}\bx^{0;2}-\br_1^{(1)}\bx^{1;2}.
\end{equation}
\end{enumerate}
Hence combining \eqref{E:factor1} and \eqref{E:factor2}, we obtain the following:
\begin{align}\label{E:factor-1-2}   
\br_1^{(1)}\bx^{0;2}=-\Lambda_1\delta_1\Big(y^{(1)}\bx^{1;2}+y^{(2)}\bx^{1;0}\bx^{0;2}+ y^{(3)}\bx^{11;00}\bx^{0;2}\Big).
\end{align}
Plugging this into \eqref{eq:dt-y-expan1}, we obtain
\begin{align}\label{E:simple-ch-var}
\der y=\lc  \id - \laa_1 \doo \rc \lc  \id - \laa_2 \dt \rc\Bigl(y^{(1)}\bx^{1;2}+y^{(2)}\bx^{1;0}\bx^{0;2}+ y^{(3)}\bx^{11;00}\bx^{0;2}\Bigl).
\end{align}
    
\smallskip
\noindent
\emph{Step 2: Identification of rough-Young integrals.} Observe that \eqref{E:simple-ch-var} already allows to write $\delta y$ as the limit of bidimensional Riemann sums. Our aim is now to identify this limit as the sum of rough-Young integrals in the right hand side of \eqref{eq:ch-var}. In view of our decomposition~\eqref{E:z1-rough-young}, it is natural to insert  $\pm y^{(2)} \bx^{11;0\cdot2}$ in the right hand side of \eqref{E:simple-ch-var}. A careful analysis of regularities (left to the reader of sake of conciseness), plus  linearity of the maps $\ \doo,\ \dt,\ \laa_1$, and $\laa_2$ enable to write
\begin{align}\label{E:dy-sew}
\der y&=\lc  \id - \laa_1 \doo \rc \lc  \id - \laa_2 \dt \rc\Bigl(y^{(1)}\bx^{1;2}+y^{(2)}\bx^{11;0\cdot2}\Bigl)\notag\\
&\quad+\lc  \id - \laa_1 \doo \rc \lc  \id - \laa_2 \dt \rc\Bigl(y^{(2)}\bigl(\bx^{1;0}\bx^{0;2}-\bx^{11;0\cdot2}\bigl)+ y^{(3)}\bx^{11;00}\bx^{0;2}\Bigl).
\end{align}
Our aim is now to prove that the right hand side of \eqref{E:dy-sew} is also equal to $\int_1\int_2 y^{(1)}d_{12}x+ \int_1\int_2y^{(2)}d_{\hone\htwo}x$. 
Using \eqref{E:z1-rough-young} and \eqref{Eq:riemann z^1} from Theorem \ref{th:main-yr-integral} with $\zeta=y^{(1)}$ and $\zeta^1=y^{(2)}$, the first term in \eqref{E:dy-sew} is exactly the indefinite integral $\int_1\int_2 y^{(1)}d_{12}x$, which is well defined. The second term is also well defined. In fact, it is not hard to check that the increment $y^{(2)}\bigl(\bx^{1;0}\bx^{0;2}-\bx^{11;0\cdot2}\bigl)+ y^{(3)}\bx^{11;00}\bx^{0;2}$ satisfies the three conditions of Proposition~\ref{proposition:integ-2d}. On the other hand, using Theorem \ref{th:main-yr-integral} we have 
\begin{align}\label{E:int-y1-y2-sew}
&\int_{1}\int_{2} y^{(1)} \, d_{12}x + \int_{1}\int_{2} y^{(2)} \, d_{\hone\htwo}x\\
&=\lc  \id - \laa_1 \doo \rc \lc  \id - \laa_2 \dt \rc\Bigl(y^{(1)}\bx^{1;2}+y^{(2)} \bx^{11;0\cdot2}\Bigl)
+\lc  \id - \laa_1 \doo \rc \lc  \id - \laa_2 \dt \rc\Bigl( y^{(2)}\bx^{\hone;\htwo}+ y^{(3)}\bx^{1\hone;0\cdot\htwo}\Bigl).\notag
%&=\lc  \id - \laa_1 \doo \rc \lc  \id - \laa_2 \dt \rc\Bigl(y^{(1)}\bx^{1;2}+y^{(2)}\bigl( \bx^{11;0\cdot2}+\bx^{\hone;\htwo}\bigl)+y^{(3)}\bx^{1\hone;0\cdot\htwo}\Bigl).
\end{align}
In order to prove \eqref{eq:ch-var}, we compare the right hand side of \eqref{E:dy-sew} and \eqref{E:int-y1-y2-sew}. We discover that it is sufficient to show that 
\begin{align}\label{Eq:claim-ch-var-corr}
\lc  \id - \laa_1 \doo \rc &\lc  \id - \laa_2 \dt \rc\Bigl( y^{(2)}\bx^{\hone;\htwo}+ y^{(3)}\bx^{1\hone;0\cdot\htwo}\Bigl)\notag\\
&=\lc  \id - \laa_1 \doo \rc \lc  \id - \laa_2 \dt \rc\Bigl( y^{(2)}\bigl(\bx^{1;0}\bx^{0;2}-\bx^{11;0\cdot2}\bigl)+ y^{(3)}\bx^{11;00}\bx^{0;2}\Bigl).
\end{align}
To prove this claim, we first decompose $\bx^{\hone;\htwo}$ and $\bx^{1\hone;0\cdot\htwo}$ further. Indeed, suppose first that the path $x$ is smooth. The computation below follows the same corner-expansion as~\eqref{CT}: the last identity in the first line uses $\partial_1 x_{u;v}=\partial_1 x_{u;t_1}+\delta_2\partial_1 x_{u;t_1v}$ (expansion in direction~$2$), and the second line uses $\bx^{0;2}_{u;t_1t_2}=\bx^{0;2}_{s_2;t_1t_2}-\delta_1\bx^{0;2}_{us_2;t_1t_2}$ (expansion in direction~$1$).
\begin{align}\label{Eq:x^{hone;htwo}-dcp}
\bx^{\hone;\htwo}&=\int_1\int_2 d_{\hone\htwo}x= \int_1\int_2d_1xd_2x
=\int_1 d_1x\int_2 d_2x+ \int_1\int_2 d_{12}xd_{2}x\notag\\
&=\Bigl(\int_1 d_1x\Bigl)\Bigl(\int_2 d_2x\Bigl)-\int_1 d_1x\int_2 d_{12}x+ \int_1\int_2 d_{12}xd_{2}x\notag\\
&=\bx^{1;0}\bx^{0;2}- \bx^{11;0\cdot2}+\bx^{10;22}.
\end{align}
The same two expansions apply to $\bx^{1\hone;0\cdot\htwo}$: the inner integrand $\partial_1 x_{u;v}$ is expanded in direction~$2$ (third identity in the first line), and $\bx^{0;2}_{u;t_1t_2}$ is expanded in direction~$1$ (second line), with the outer $\int_1 d_1x$ factor carried along throughout. We obtain
\begin{align}\label{Eq:x^{1 hone;0.htwo}-dcp}
\bx^{1\hone;0\cdot\htwo}&=\int_1d_1x\int_2  d_{\hone\htwo}x= \int_1 d_1x \int_2 d_1x d_2x
= \int_1 d_1x\, d_1x \int_2 d_2x+\int_1 d_1x \int_2d_{12}x d_2x \notag\\
&= \Bigl(\int_1 d_1x\, d_1x\Bigl) \Bigl(\int_2 d_2x\Bigl)- \int_1 d_1x\, d_1x \int_2 d_{12}x + \int_1 d_1x \int_2d_{12}x d_2x\notag\\
&= \bx^{11;00}\bx^{0;2} -  \bx^{111;002}+\bx^{110;022}.
\end{align}
For general $x\in \mathcal{P}_{1,1}^{\ga_1,\ga_2}$, identities \eqref{Eq:x^{hone;htwo}-dcp} and \eqref{Eq:x^{1 hone;0.htwo}-dcp} follow by a limiting argument rooted in the geometric assumptions of Hypotheses~\ref{hyp:rs1} and~\ref{hyp:rs3}: these identities hold exactly for each smooth approximation $x^{n}$. Furthermore the convergence of all rough sheet elements of $\XX^{n}$ to those of $\XX$, including $\bx^{111;002}$ and $\bx^{110;022}$ by Lemma~\ref{hyp:rs3}, allows passing to the limit $n\to\infty$. Substituting these two decompositions into the left-hand side of \eqref{Eq:claim-ch-var-corr} and separating remainder terms by linearity we get
\begin{align}\label{Eq:id-ga-del-pre}
&\lc\id-\laa_1\doo\rc\lc\id-\laa_2\dt\rc\Bigl(y^{(2)}\bx^{\hone;\htwo}+y^{(3)}\bx^{1\hone;0\cdot\htwo}\Bigl)\notag\\
&\quad=\lc\id-\laa_1\doo\rc\lc\id-\laa_2\dt\rc\Bigl(y^{(2)}\bigl(\bx^{1;0}\bx^{0;2}-\bx^{11;0\cdot2}\bigl)+y^{(3)}\bx^{11;00}\bx^{0;2}\Bigl)\notag\\
&\quad+\lc\id-\laa_1\doo\rc\lc\id-\laa_2\dt\rc\Bigl(y^{(2)}\bx^{1 0;2 2}-y^{(3)}\bx^{111;002}+y^{(3)}\bx^{110;022}\Bigl).
\end{align}
The remainder terms in the second line are governed by the regularities $\bx^{1 0;2 2}, \bx^{110;022} \in \mathcal{P}_{2,2}^{*;2\ga_2}$ and $\bx^{111;002} \in \mathcal{P}_{2,2}^{3\ga_1;*}$. Since $2\ga_2>1$ and $3\ga_1>1$, the sewing Proposition~\ref{proposition:planar Lambda} gives
\begin{equation}\label{Eq:id-ga-del}
\lc  \id - \laa_2 \dt \rc\bigl(y^{(2)}\bx^{1 0;2 2} \bigl)
\,=\,
 \lc  \id - \laa_1 \doo \rc\bigl(y^{(3)}\bx^{111;002}\bigl)
\,=\,
\lc  \id - \laa_2 \dt \rc\bigl(y^{(3)}\bx^{110;022} \bigl)=0 .
\end{equation}
Hence the third line of \eqref{Eq:id-ga-del-pre} vanishes. We end up with
\begin{multline*}
\lc\id-\laa_1\doo\rc\lc\id-\laa_2\dt\rc\Bigl(y^{(2)}\bx^{\hone;\htwo}+y^{(3)}\bx^{1\hone;0\cdot\htwo}\Bigl)\\
=\lc\id-\laa_1\doo\rc\lc\id-\laa_2\dt\rc\Bigl(y^{(2)}\bigl(\bx^{1;0}\bx^{0;2}-\bx^{11;0\cdot2}\bigl)+y^{(3)}\bx^{11;00}\bx^{0;2}\Bigl),
\end{multline*}
which is exactly the identity \eqref{Eq:claim-ch-var-corr} we set out to prove in Step~2.

Having proved \eqref{Eq:claim-ch-var-corr}, we can now conclude. Since \eqref{E:dy-sew} and \eqref{E:int-y1-y2-sew} share the same first term $\lc\id-\laa_1\doo\rc\lc\id-\laa_2\dt\rc\bigl(y^{(1)}\bx^{1;2}+y^{(2)}\bx^{11;0\cdot2}\bigl)$, their equality reduces to \eqref{Eq:claim-ch-var-corr}, which yields~\eqref{eq:ch-var} and establishes item~\ref{prop:ch-var-i}. Item~\ref{prop:ch-var-ii} then follows from the Riemann sum convergence for each of the two rough-Young integrals, as provided by Theorem~\ref{th:main-yr-integral}, completing the proof of Proposition~\ref{prop:ch-var}.
\end{proof}

\section{Application to a fractional Brownian sheet}\label{sec:fbs}

This section is devoted to showing that a fractional Brownian sheet (fBs) $x$ with Hurst parameters $H_1>1/3$ and $H_2>1/2$ can be lifted to a geometric rough sheet $\XX$ satisfying Hypotheses~\ref{hyp:rs1}, \ref{hyp:rs2}, and~\ref{hyp:rs3}. Once this is established, the main results of the paper (namely the rough-Young integral of Theorem~\ref{th:main-yr-integral} and the change of variable formula of Proposition~\ref{prop:ch-var}) apply directly to the fBs. The construction proceeds in two steps: we first develop a generalized multiparameter version of the Garsia-Rodemich-Rumsey inequality providing the key Hölder regularity criterion for 2D increments, and then apply it to the fBs to construct the rough sheet elements as limits of regularized increments.

\subsection{A generalized multiparameter Garsia-Rodemich-Rumsey inequality}

The classical Garsia-Rodemich-Rumsey (GRR) lemma bounds the Hölder norm of a function by a weighted $L^p$ functional of its increments. In the rough sheet setting, the objects of interest are not simple path increments but higher-order increments in $\cp_{2,2}$, for which one needs a version of the GRR lemma that additionally handles a coboundary condition. We recall in Lemma~\ref{L:1d-general-GRR} a one-dimensional generalization that covers the case of a general increment in $\cac_{2}$, and then establish in Lemma~\ref{L:2D-general-GRR} a biparameter extension by iterating the 1D result in each direction. These estimates are the main analytic tool for verifying the regularity of the iterated integrals forming $\XX$ in the next subsection.

First for sake of comparison, let us recall a version of Garsia-Rodemich-Rumsey (GRR for short) lemma generalized to $1d$ increments, which can be found e.g in~\cite[Lemma 3.4]{HTW} in a Volterra equation context.
\begin{lemma}\label{L:1d-general-GRR}
Recall from Section~\ref{incr} that $\cac_{1}$ is the space of functions on $[0,1]$ and $\cac_{2}$ is the space of $1$-increments. Let $f\in \cac_1$, $g \in \cac_2$, $\kappa > 0$, and $p \geq 1$. Define
\begin{equation}\label{eq:U}
U^1_{\kappa,p}(g) :=
\left(
\int \int_{[0,1]^2}
\frac{|g_{vw}|^{2p}}{|w-v|^{2\kappa p + 2}}
\, dv \, dw
\right)^{\frac{1}{2p}}.
\end{equation}
Then the following holds true:
   \begin{enumerate}[label=\textbf{\emph{(\roman*)}}]
       \item If $U^1_{\kappa,p}(\delta f) < \infty$, then $f \in \cac_1^{\kappa}$ (in the sense of \eqref{def:hnorm-c1}); more precisely,
       there exists a universal constant $c_0 > 0$ such that
\begin{align}\label{eq:1d-GRR}
\|\delta f\|_{\kappa} \leq c_0  \, U^1_{\kappa,p}(\delta f).
\end{align}

       \item If $\delta g \in \cac_{3}^{\kappa}$ and $U^1_{\kappa,p}(g) < \infty$, then $g \in \cac_2^{\kappa}$ (in the sense of \eqref{eq:def-norm-C2}); specifically, 
              there exists a universal constant $c_1 > 0$ such that
\begin{equation}\label{eq:1d-general-GRR}
\|g\|_{\kappa} \leq c_1 \big( U^1_{\kappa,p}(g) + \|\delta g\|_{\kappa} \big).
\end{equation}
   \end{enumerate}
\end{lemma}

With Lemma~\ref{L:1d-general-GRR} in hand, we turn to the biparameter setting. Recall from Section~\ref{sec:planar-increments} that $\cp_{2,2}$ is the space of rectangular $2$-increments on $[0,1]^{2}$. The natural planar analogue of $U^{1}_{\kappa,p}(g)$ from~\eqref{eq:U} is obtained by replacing the single weighted integral over $[0,1]^{2}$ by a double weighted integral over $[0,1]^{4}$, penalizing separately the increments in the $s$- and $t$-directions.

\begin{definition}\label{def:U2}
Let $G\in\cp_{2,2}$, where $\cp_{2,2}$ is given by~\eqref{c3}. We consider $\kappa_{1},\kappa_{2}>0$, and $p\geq 1$. We define the following integral quantity related to $G$:
\begin{equation}\label{eq:U^n}
U_{\kappa_{1},\kappa_{2},p}^{2}(G) :=
\left(
\int \int_{[0,1]^{4}}
\frac{|G_{s_{1}s_{2};t_{1}t_{2}}|^{2p}}{|s_{2}-s_{1}|^{2\kappa_{1} p + 2}|t_{2}-t_{1}|^{2\kappa_{2} p + 2}}
\, ds_{1}\,ds_{2}\,dt_{1}\,dt_{2}
\right)^{\frac{1}{2p}}.
\end{equation}
\end{definition}
The quantity $U^{2}_{\kappa_{1},\kappa_{2},p}$ from Definition~\ref{def:U2} is the right object to state the two-dimensional GRR lemma, obtained by iterating Lemma~\ref{L:1d-general-GRR} in each direction.

\begin{lemma}\label{L:2D-general-GRR}
Let $f\in\cp_{1,1}$ and $g\in\cp_{2,2}$ (see~\eqref{c3}), let $\kappa_{1},\kappa_{2}>0$ and $p\geq 1$, and let $\eps>0$ be a small parameter. In all items below, $U^{2}_{\kappa_{1},\kappa_{2},p}$ is the functional introduced in Definition~\ref{def:U2}.
\begin{enumerate}[label=\textnormal{\textbf{(\roman*)}}]
    \item\label{it:grr-gen-i} If $U^2_{\kappa_1,\kappa_2,p}(\delta f) < \infty$, then $f \in \cp^{\kappa_1,\kappa_2}_{1,1}$; more precisely
\begin{align}\label{eq:2d-GRR}
\|\delta f\|_{\kappa_1,\kappa_2} \leq c_2  U^2_{\kappa_1,\kappa_2,p}(\delta f),
\end{align}
for a universal constant $c_2 > 0$.
   
   \item\label{it:grr-gen-ii} 
   Assume that \,$\doo g\in \cp_{3,2}^{\ka_1,\ka_2}$, $\dt g\in \cp_{2,3}^{\ka_1+\eps,\ka_2}$,  and $U^2_{\ka_1,\ka_2,p}(g)<\infty$.
   Then for all $p>\frac{1}{2\eps}$ we have
   \begin{align}\label{eq:2D-GRR1}
   \lVert g\rVert_{\kappa_1,\kappa_2}\le c_3 \Bigl(U^2_{\ka_1,\ka_2,p}(g)+\lVert \doo g \rVert_{\ka_1,\ka_2}+\lVert \dt g \rVert_{\ka_1+\eps,\ka_2}\Bigl),
   \end{align}
   for some constant $c_3$ depending on $\varepsilon$ and $p$.
    
    \item\label{it:grr-gen-iii} 
    Suppose that $\doo g\in \cp_{3,2}^{\ka_1,\ka_2+\eps}$, $\dt g\in \cp_{2,3}^{\ka_1,\ka_2}$, and $U^2_{\ka_1,\ka_2,p}(g)<\infty$. Then for all $p>\frac{1}{2\eps}$ we have
   \begin{align}\label{eq:2D-GRR2}
   \lVert g\rVert_{\kappa_1,\kappa_2}\le c_4 \Bigl(U^2_{\ka_1,\ka_2,p}(g)+\lVert \doo g \rVert_{\ka_1,\ka_2+\eps}+\lVert \dt g \rVert_{\ka_1,\ka_2}\Bigl),
   \end{align}
    for some constant $c_4$ depending on $\varepsilon$ and $p$.
\end{enumerate}
\end{lemma}
\begin{proof}

We will prove item~\ref{it:grr-gen-i} and~\ref{it:grr-gen-ii} separately. Then item~\ref{it:grr-gen-iii} is proved very similarly to  item~\ref{it:grr-gen-ii}.

\noindent 
\textit{Proof of item}~\ref{it:grr-gen-i}. Since $\der f= \doo \dt f=\dt \doo f$, using \eqref{eq:1d-GRR} from Lemma \ref{L:1d-general-GRR}, first in the $s$-direction and then in the $t$-direction, we obtain 
\begin{align}
\lvert \delta f_{s_1s_2;t_1t_2}\rvert^{2p}& \le c_5 |s_2-s_1|^{2\ka_1 p}\int_{[0,1]^2}\frac{|\der f_{u_1u_2;t_1t_2}|^{2p}}{|u_2-u_1|^{2\ka_1 p+2}}du_1du_2\notag\\
&\le c_6 |s_2-s_1|^{2\ka_1 p}|t_2-t_1|^{2\ka_2 p}\int_{[0,1]^4}\frac{|\der f_{u_1u_2;v_1v_2}|^{2p}}{|u_2-u_1|^{2\ka_1 p+2}|v_2-v_1|^{2\ka_2 p+2}}du_1du_2dv_1dv_2,\notag
\end{align}
which proves \eqref{eq:2d-GRR}.

\noindent 
\textit{Proof of item}~\ref{it:grr-gen-ii}. 
Applying \eqref{eq:1d-general-GRR} in the $s$-direction, we get
\begin{align}\label{eq:grr-gen-ii-s}
|g_{s_1s_2;t_1t_2}|^{2p}& \le c_7 |s_2-s_1|^{2\ka_1 p}\Big(\int_{[0,1]^2}\frac{|g_{u_1u_2;t_1t_2}|^{2p}}{|u_2-u_1|^{2\ka_1 p+2}}du_1du_2 +\left\lVert \doo g_{.;t_1t_2}\right\rVert^{2p}_{\ka_1}\Big).
\end{align}
Applying \eqref{eq:1d-general-GRR} once more, now in the $t$-direction, to each term in the parenthesis in \eqref{eq:grr-gen-ii-s}, we obtain
\begin{align}\label{eq:grr-gen-ii-t}
|g_{s_1s_2;t_1t_2}|^{2p}& \le c_8 |s_2-s_1|^{2\ka_1 p}|t_2-t_1|^{2\ka_2 p}\Big(\bigl(U^2_{\ka_1,\ka_2,p}(g)\bigr)^{2p} + V_{1}(g)^{2p} + V_{2}(g)^{2p}\Big),
\end{align}
where the identification of the first term with $(U^2_{\ka_1,\ka_2,p}(g))^{2p}$ follows directly from \eqref{eq:U^n} in Definition~\ref{def:U2}, and where we have set
\begin{equation}\label{eq:V12-def}
V_{1}(g)^{2p} = \int_{[0,1]^2}\frac{\lVert \dt g_{u_1u_2;\cdot} \rVert_{\ka_2}^{2p}}{|u_2-u_1|^{2\ka_1 p+2}}\,du_{1}\,du_{2}, \quad\text{and}\quad
V_{2}(g)^{2p} = \sup_{t_1 \ne t_2}\frac{\left\lVert \doo g_{.;\,t_1t_2}\right\rVert^{2p}_{\ka_1}}{|t_2-t_1|^{2\ka_2 p}} .
\end{equation}
Let us now examine the terms $V_{1}(g)$ and $V_{2}(g)$: first, recalling the $\kappa_{1}$-norm from \eqref{eq:normOCC2}, it is readily checked from \eqref{eq:V12-def} that
\begin{equation}\label{eq:V2-bound}
V_{2}(g)^{2p} = \lVert \doo g \rVert_{\ka_1,\ka_2}^{2p} .
\end{equation}
Furthermore, the term $V_1(g)$ reads  
\begin{align*}
V_1(g)^{2p}=\int_{[0,1]^2}\frac{du_1du_2}{|u_2-u_1|^{2\ka_1 p+2}}\sup_{v_1<v_3}\frac{\left\lvert \delta_2 g_{u_1u_2;v_1v_2v_3}\right\lvert^{2p}}{|v_3-v_1|^{2\ka_2 p}}.
\end{align*}
Therefore, considering the extra parameter $\eps$ such tat $2p \eps >1$, 
%since $\dt g\in \cp_{2,3}^{\kappa_1+\eps,\kappa_2}$, the term $V_{1}(g)$ defined in \eqref{eq:V12-def} satisfies
\begin{align*}
V_{1}(g)^{2p} &\le \lVert \dt g \rVert_{\ka_1+\eps,\ka_2}^{2p}\int_{[0,1]^2}\frac{\,du_1\,du_2}{|u_2-u_1|^{2\ka_1 p +2 }}\,|u_2-u_1|^{2(\ka_1+\eps)p}%.
\end{align*}
%where in the last inequality we used \eqref{eq:2d-GRR}. 
It is now trivial to check that the integral on the right-hand side above is finite, so that
\begin{equation}\label{eq:V1-bound}
V_{1}(g)\le c\, \lVert \dt g \rVert_{\ka_1+\eps,\ka_2}.
\end{equation}
Inserting \eqref{eq:V2-bound} and \eqref{eq:V1-bound} into \eqref{eq:grr-gen-ii-t} and taking the supremum over $s_{1}\neq s_{2}$ and $t_{1}\neq t_{2}$ yields~\eqref{eq:2D-GRR1}, concluding the proof of item~\ref{it:grr-gen-ii}.
The proof of~\ref{it:grr-gen-iii} is similar to~\ref{it:grr-gen-ii}.
\end{proof}

We now turn to establishing the existence and regularity of the iterated integrals $\bx^{i;j}$ and $\bx^{1i;0\cdot j}$, for $(i,j)\in\{(1,2),(\hat{1},\hat{2})\}$. Those integrals appear in Table~\ref{table:rp} and in the Riemann sum~\eqref{Eq:riemann-delta(y)} of the change of variable formula~\eqref{eq:ch-var} for $\delta\vp(x)$. An important ingredient in this analysis is the following consequence of the two-dimensional generalized GRR lemma (Lemma~\ref{L:2D-general-GRR}).
\begin{lemma}\label{L:general-2d-GRR-sign}
Let $x$ and $z$ be two regular paths indexed by the plane, let $(\ga_{1},\ga_{2})\in (0,\infty)^{2}$, $p\geq 1$, and $\eps>0$ be a small parameter. In all items below, $U^{2}_{\ga_{1},\ga_{2},p}$ is the functional from Definition~\ref{def:U2}, and $(i,j)$ stands for either $(1,2)$ or $(\hat{1},\hat{2})$. Then the following holds true:

\begin{enumerate}[label=\textnormal{\textbf{(\roman*)}}]
    \item\label{it:grr-sign-i} If $U^{2}_{\ga_{1},\ga_{2},p}(\bx^{i;j})<\infty$ and $U^{2}_{\ga_{1},\ga_{2},p}(\bz^{i;j})<\infty$, then $\bx^{i;j}, \bz^{i;j}\in \cp_{2,2}^{\ga_{1},\ga_{2}}$. Furthermore, there exists a universal constant $c_{2} > 0$ such that
\begin{align}\label{eq:GRR-x^{1;2}}
\bigl\lVert \bx^{i;j}-\bz^{i;j}\bigr\rVert_{\ga_{1};\ga_{2}}\le c_{2}\,U^{2}_{\ga_{1},\ga_{2},p}(\bx^{i;j}-\bz^{i;j}).
\end{align}

    \item\label{it:grr-sign-ii} Assume that $U^{2}_{2\ga_{1},\ga_{2},p}(\bx^{1i;0\cdot j})<\infty$ and $U^{2}_{2\ga_{1},\ga_{2},p}(\bz^{1i;0\cdot j})<\infty$. Assume in addition that
\begin{equation}\label{eq:reg-x^{1i;2j}}
\bx^{i;j},\bz^{i;j}\in \cp_{2,2}^{\ga_{1},\ga_{2}},
\quad\text{and}\quad
\bx^{1i;2\cdot j},\bz^{1i;2\cdot j}\in \cp_{2,3}^{2\ga_{1}+\eps,\ga_{2}}.
\end{equation}
Then for all $p>\frac{1}{2\eps}$ we have $\bx^{1i;0\cdot j}, \bz^{1i;0\cdot j}\in \cp_{2,2}^{2\ga_{1},\ga_{2}}$. In addition, there exists a  constant $c_{3} > 0$ depending on $\varepsilon$ and $p$ such that
\begin{align}\label{eq:GRR-x^{1i;0.j}}
\bigl\lVert \bx^{1i;0\cdot j} - \bz^{1i;0\cdot j} \bigr\rVert^{2p}_{2\gamma_{1};\gamma_{2}} &\le
c_{3}\,\Bigl(\bigl(U^{2}_{2\ga_{1},\ga_{2},p}(\bz^{1i;0\cdot j}-\bx^{1i;0\cdot j})\bigr)^{2p} \notag\\
&\qquad +\bigl\lVert \bz^{1;0}\,\bz^{i;j} -\bx^{1;0}\,\bx^{i;j}\bigr\rVert_{2\ga_{1};\ga_{2}}^{2p}+\bigl\lVert \bz^{1i;2\cdot j} -\bx^{1i;2\cdot j}\bigr\rVert_{2\ga_{1}+\eps;\ga_{2}}^{2p}\Bigr).
\end{align}
\end{enumerate}
\end{lemma}

\begin{proof}

\noindent
\textit{Proof of item}~\ref{it:grr-sign-i}. Equation~\eqref{eq:GRR-x^{1;2}} follows directly from \eqref{eq:2d-GRR} in Lemma~\ref{L:2D-general-GRR} applied with $\delta f = \bx^{i;j}-\bz^{i;j}$.

\noindent
\textit{Proof of item}~\ref{it:grr-sign-ii}. We apply item~\ref{it:grr-gen-ii} of Lemma~\ref{L:2D-general-GRR} with $g=\bx^{1i;0\cdot j} - \bz^{1i;0\cdot j}$, $\ka_{1}=2\ga_{1}$, and $\ka_{2}=\ga_{2}$. Using \eqref{E:d1x^(11;02)}, we have $\doo g=\bx^{1;0}\,\bx^{i;j}- \bz^{1;0}\,\bz^{i;j}$. Since $\bx^{1;0}, \bz^{1;0}\in \cp_{2,1}^{\ga_{1},*}$ and $\bx^{i;j}, \bz^{i;j} \in \cp_{2,2}^{\ga_{1},\ga_{2}}$ owing to assumption~\eqref{eq:reg-x^{1i;2j}}, some elementary computations yield $\doo g \in \cp_{3,2}^{2\ga_{1},\ga_{2}}$. On the other hand, using \eqref{E:d2-x^(11;0.2)} (which remains true for $(i,j)=(\hone,\htwo)$ also) and \eqref{eq:reg-x^{1i;2j}}, we obtain $\dt g= \bz^{1\hone;2\cdot\htwo}-\bx^{1\hone;2\cdot\htwo}\in \cp_{2,3}^{2\ga_{1}+\eps,\ga_{2}}$. Hence \eqref{eq:2D-GRR1} yields \eqref{eq:GRR-x^{1i;0.j}}, completing the proof.
\end{proof}

\begin{remark}
It is worth emphasizing that Lemma \ref{L:general-2d-GRR-sign} requires only $\ga_2$-Hölder regularity in the $t$-direction for $\bx^{1i;2\cdot j}$ and $\bz^{1i;2\cdot j}$. However, the structural properties of the signature dictate that such increments actually possess $2\ga_2$-Hölder regularity. This provides more than enough regularity to safely satisfy the hypothesis.
\end{remark}

Since we are mainly interested in constructing a rough sheet above the fBs, we need to check that all the assumptions of \textbf{\textit{(i)}} and \textbf{\textit{(ii)}} in Lemma \ref{L:general-2d-GRR-sign} are satisfied. 
The $U_{\ka_1,\ka_2,p}$-integrability conditions can be verified by taking their $L^{2p}$-expectation and using Fubini's theorem. However, the H\"older regularity condition
$\bx^{1i;2\cdot j},\bz^{1i;2\cdot j}\in \cp_{2,3}^{2\ga_1+\eps,\ga_2}$ is not trivial.
Specifically,
the analytical difficulty in bounding the $\cp_{2,3}$-H\"older norm of the term $\bx^{1i;2\cdot j}$ stems from its asymmetric structure: it acts as a $2$-variables increment in the $s$-direction and a $3$-variables increment in the $t$-direction.
Standard GRR bounds \eqref{eq:1d-general-GRR} and \eqref{eq:2D-GRR1}-\eqref{eq:2D-GRR2} are tailored for increments in $\cac_2$ and $\cp_{2,2} \equiv \cac_2\otimes \cac_2$, respectively. To bypass this structural mismatch, for fixed $s_1,s_2\in [0,1]$, we ``lift'' the $3$-variable increment in the $t$-direction 
$\bx^{1i;2\cdot j}_{s_1s_2;t_1t_2t_3}$ into a $(2+2)$-variable increment in $\cac_2\otimes\cac_2$, denoted by $\bx^{1i;2\otimes j}_{s_1s_2;v_1v_2;w_1w_2}$, evaluated over $t_1=v_1$, $t_2=v_2=w_1$ and $t_3=w_2$. This splitting procedure allows us to treat the $3$-variables increment $\bx^{1i;2\cdot j}$, for fixed $s$-parameters, as an ``exact'' rectangular increment in the $t$-direction, in $\cp_{2,2}$, for which we can apply the simplest 2D-GRR bound \eqref{eq:2d-GRR} easily. We formalize this construction in the following definition.

\begin{definition}\label{D:x^(1i;2otimesj)}
For a path $x\in \cp_{1,1}^{\ga_1,\ga_2}$, the tensor-like increment $\bx^{1i;2\otimes j}$, seen as an element of $\cac_2\bigl(\cac_2\otimes \cac_2\bigl)$, is defined by
\begin{equation}\label{E:x^(1i;2xj)}
\bx^{1i;2\otimes j}_{s_1s_2;t_1t_2;t_3t_4} := \int_{s_1<\sigma_1<\sigma_2<s_2}\Big(\int_{t_1}^{t_2}d_{12}x_{\sigma_1;\tau_1}\Big)\Big(\int_{t_3}^{t_4}d_{ij}x_{\sigma_2;\tau_2}\Big),
\end{equation}
where similarly to Hypothesis~\ref{hyp:rs2}, we interpret the right hand side of~\eqref{E:x^(1i;2xj)} in the geometric sense.
\end{definition}

The following lemma shows that the $\cp_{2,3}^{\ka_1,\ka_2}$-norm of $\bx^{1i;2\cdot j}-\bz^{1i;2\cdot j}$ can be controlled by the integrability of the lifted increment $\bx^{1i;2\otimes j}-\bz^{1i;2\otimes j}$ from Definition~\ref{D:x^(1i;2otimesj)}.

\begin{lemma}\label{L:x-z^(1i;2.j)}
Let $x, z\in \cp_{1,1}^{\ga_1+\eps,\ga_2}$, and let $\bx^{1i;2\otimes j}$, $\bz^{1i;2\otimes j}$ be corresponding the tensor-like increments from Definition~\ref{D:x^(1i;2otimesj)}. For $\ka_1 := 2\ga_1 +\eps$, $\ka_2 := \ga_2$, %and $x,z \in \cp_{1,1}^{\ga_1,\ga_2} $
define the term
\begin{equation}\label{E:U^3-def}
U^3_{\ka_1,\ka_2,\ka_2,p}\bigl(\bx^{1i;2\otimes j}\bigl) 
:= \Bigl(\int_{[0,1]^6}
\frac{\lvert\bx^{1i;2\otimes j}_{u_1u_2;v_1v_2;w_1w_2} \rvert^{2p}}
{|u_2-u_1|^{2\ka_1 p+2}|v_2-v_1|^{2\ka_2 p+2}|w_2-w_1|^{2\ka_2 p+2}}
\, du_1\cdots dw_2 \Bigl)^{\frac{1}{{2p}}}.
\end{equation}
If\ $x,z\in \cp_{1,1}^{\ga_1+\eps,\ga_2}$ such that\  $U^3_{\ka_1,\ka_2,\ka_2,p}\bigl(\bx^{1i;2\otimes j}-\bz^{1i;2\otimes j}\bigl)<\infty$, then we have:
\begin{align}\label{E:x^(1i;2.j)-z^(1i;2.j)}
\left\lVert \bz^{1i;2\cdot j} -\,\bx^{1i;2\cdot j}\right\rVert_{\ka_1;\ka_2}\le c\Bigl(U^3_{\ka_1,\ka_2,\ka_2,p}\bigl(\bx^{1i;2\otimes j}-\bz^{1i;2\otimes j}\bigl) \ +\ \bigl\lVert \bx^{1;2}\bx^{i;j}-\bz^{1;2}\bz^{i;j} \bigl\rVert_{\ka_1,\ka_2}\Bigl).
\end{align}
\end{lemma}

\begin{proof}
To establish this regularity, we combine and iterate the generalized 1D and 2D GRR lemmas \ref{L:1d-general-GRR} and \ref{L:2D-general-GRR}. First, applying Lemma~\ref{L:1d-general-GRR} to $(s_1,s_2)\mapsto \bx^{1i;2\cdot j}_{s_1s_2;t_1t_2t_3}-\bz^{1i;2\cdot j}_{s_1s_2;t_1t_2t_3}$ for fixed $t_1<t_2<t_3$, we get
\begin{multline}\label{E:holder-(2,3)-i}
\sup_{0\le s_1<s_2\le 1}
\frac{\bigl\lvert \bx^{1i;2\cdot j}_{s_1s_2;t_1t_2t_3}-\bz^{1i;2\cdot j}_{s_1s_2;t_1t_2t_3}\bigr\rvert^{2p}}{|s_2-s_1|^{2\ka_1 p}} 
\\
\le c\Bigl(U^1_{\ka_1,p}\bigl(\bx^{1i;2\cdot j}_{\cdot;t_1t_2t_3}-\bz^{1i;2\cdot j}_{\cdot;t_1t_2t_3}\bigl)^{2p}+\bigl\lVert \bx^{1;2}_{\cdot;t_1t_2}\bx^{i;j}_{\cdot;t_2t_3}-\bz^{1;2}_{\cdot;t_1t_2}\bz^{i;j}_{\cdot;t_2t_3}\bigl\rVert_{\ka_1}^{2p}\Bigr),
\end{multline}
where we have used the fact that $\doo \bx^{1i;2\cdot j}=\bx^{1\cdot \,i;2\cdot j}=\bx^{1;2}\bx^{i;j}$. Dividing \eqref{E:holder-(2,3)-i} by $|t_3-t_1|^{2\ka_2 p}$, taking the supremum over $0\le t_1<t_2<t_3\le 1$, and using the definitions of $\lVert\cdot\rVert_{\ka_1;\ka_2}$ and $U^1_{\ka_1,p}(\lVert\cdot\rVert_{\ka_2})$, we obtain
\begin{align}\label{E:holder-(2,3)}
\left\lVert \bz^{1i;2\cdot j} -\,\bx^{1i;2\cdot j}\right\rVert_{\ka_1;\ka_2}^{2p} \le c \Bigl(U^1_{\ka_1,p}\bigl(\bigl\lVert \bx^{1i;2\cdot j}-\bz^{1i;2\cdot j}\bigl \rVert_{\ka_2}\bigl)^{2p}+\bigl\lVert \bx^{1;2}\bx^{i;j}-\bz^{1;2}\bz^{i;j} \bigl\rVert^{2p}_{\ka_1,\ka_2}\Bigl).
\end{align}
We will now focus on bounding the right hand side above.

Let us first bound the second term in the right hand side of \eqref{E:holder-(2,3)}. To this aim, we write
\begin{equation*}
\bx^{1;2}\bx^{i;j}-\bz^{1;2}\bz^{i;j} = \bigl(\bx^{1;2}-\bz^{1;2}\bigr)\bx^{i;j}+\bz^{1;2}\bigl(\bx^{i;j}-\bz^{i;j}\bigr).
\end{equation*}
Since $x, z\in \cp_{1,1}^{\ga_1+\eps,\ga_2}$, both pairs $\bx^{1;2},\bz^{1;2}$ and $\bx^{i;j},\bz^{i;j}$ belong to $\cp_{2,2}^{\ga_{1}+\eps,\ga_{2}}$. Standard product estimates for Hölder functions then yield $\lVert\bx^{1;2}\bx^{i;j}-\bz^{1;2}\bz^{i;j}\rVert_{\ka_1,\ka_2}<\infty$, by an argument analogous to the one carried out for the term $\lVert\bz^{1;0}\bz^{i;j}-\bx^{1;0}\bx^{i;j}\rVert_{2\ga_{1};\ga_{2}}$ in the proof of item~\ref{it:grr-sign-ii} of Lemma~\ref{L:general-2d-GRR-sign}. 

Let now turn to the first term in the right hand side of \eqref{E:holder-(2,3)}, which needs to be estimated further thanks to our splitting mechanism. Indeed, owing to the geometric assumption observe that for $t_1<t_2<t_3$ we have $\bx^{1i;2\cdot j}_{s_1s_2;t_1t_2t_3}=\bx^{1i;2\otimes j}_{s_1s_2;t_1t_2;t_2t_3}$, where $\bx^{1i;2\otimes j}$ is defined in~\eqref{E:x^(1i;2xj)}. Then, for $s_1,s_2\in [0,1]$ fixed, the H\"older norm $\lVert \bx^{1i;2\cdot j}_{s_1s_2;\cdot}-\bz^{1i;2\cdot j}_{s_1s_2;\cdot} \rVert_{\ka_2}$ can be estimated as follows:
\begin{align}
\bigl\lVert \bx^{1i;2\cdot j}_{s_1s_2;\cdot}-\bz^{1i;2\cdot j}_{s_1s_2;\cdot}\bigl \rVert_{\ka_2}\le \bigl\lVert \bx^{1i;2\otimes j}_{s_1s_2;\cdot;\cdot}-\bz^{1i;2\otimes j}_{s_1s_2;\cdot;\cdot}\bigr\rVert_{\ka_2,\ka_2} 
\end{align} 
The idea behind embedding the increment $\bx^{1i;2\cdot j}_{s_1s_2;t_1t_2t_3}$ into the one $\bx^{1i;2\otimes j}_{s_1s_2;t_1t_2;t_2t_3}$, defined in~\eqref{E:x^(1i;2xj)}, is that the latter can be viewed, for $s_1,s_2$ fixed, as an increment in $\cp_{2,2}$ which we can write as an ``exact'' rectangular increment. Then for this rectangular increment, we can apply~\eqref{eq:2d-GRR} directly. Indeed, setting $f_{r,t}:=\bx^{1i;2\otimes j}_{s_1s_2;0\,r;0\,t}$ and recalling from \eqref{eq:rect-incr} that $\delta f_{r_1r_2;t_1t_2} = f_{r_2,t_2}-f_{r_1,t_2}-f_{r_2,t_1}+f_{r_1,t_1}$, the definition~\eqref{E:x^(1i;2xj)} and bilinearity of the product give
\begin{align*}
\delta f_{r_1r_2;t_1t_2}
&= \int_{s_1<\sigma_1<\sigma_2<s_2}
\Bigl[\Bigl(\int_{0}^{r_2}-\int_{0}^{r_1}\Bigr)d_{12}x_{\sigma_1;\tau_1}\Bigr]
\Bigl[\Bigl(\int_{0}^{t_2}-\int_{0}^{t_1}\Bigr)d_{ij}x_{\sigma_2;\tau_2}\Bigr]\\
&= \int_{s_1<\sigma_1<\sigma_2<s_2}
\Bigl(\int_{r_1}^{r_2}d_{12}x_{\sigma_1;\tau_1}\Bigr)
\Bigl(\int_{t_1}^{t_2}d_{ij}x_{\sigma_2;\tau_2}\Bigr)
= \bx^{1i;2\otimes j}_{s_1s_2;r_1r_2;t_1t_2}.
\end{align*} 
Therefore using \eqref{eq:2d-GRR} we obtain
\begin{equation}\label{f1}
\bigl\lVert \bx^{1i;2\otimes j}_{s_1s_2;\cdot;\cdot}-\bz^{1i;2\otimes j}_{s_1s_2;\cdot;\cdot}\bigr\rVert_{\ka_2,\ka_2} \le c_2\, U^2_{\ka_2,\ka_2,p}\bigl(\bx^{1i;2\otimes j}_{s_1s_2;\cdot;\cdot}-\bz^{1i;2\otimes j}_{s_1s_2;\cdot;\cdot}\bigl).
\end{equation}
Hence recalling that we are bounding the term $U$ in the right hand side of~\eqref{E:holder-(2,3)} and plugging~\eqref{f1}, we discover that
\begin{align*}
U^{1}_{\ka_1,p}\bigl(\bigl\lVert \bx^{1i;2\cdot j}-\bz^{1i;2\cdot j}\bigl \rVert_{\ka_2}\bigl)^{2p}
&= \int_{[0,1]^2}\frac{\bigl\lVert \bx^{1i;2\cdot j}_{u_1u_2;\cdot}-\bz^{1i;2\cdot j}_{u_1u_2;\cdot}\bigl \rVert_{\ka_2}^{2p}}{|u_2-u_1|^{2\ka_1 p+2}}du_1du_2\\
&\le c \int_{[0,1]^2}\frac{U^2_{\ka_2,\ka_2,p}\bigl(\bx^{1i;2\otimes j}_{u_1u_2;\cdot;\cdot}-\bz^{1i;2\otimes j}_{u_1u_2;\cdot;\cdot}\bigl)^{2p}}{|u_2-u_1|^{2\ka_1 p+2}}du_1du_2.
\end{align*}
Expanding $U^2_{\ka_2,\ka_2,p}$ according to its expression \eqref{eq:U^n} in Definition~\ref{def:U2}, we then get
\begin{align}
&U^{1}_{\ka_1,p}\bigl(\bigl\lVert \bx^{1i;2\cdot j}-\bz^{1i;2\cdot j}\bigl \rVert_{\ka_2}\bigl)^{2p} \notag\\
&\le c \int_{[0,1]^6}\frac{\lvert\bx^{1i;2\otimes j}_{u_1u_2;v_1v_2;w_1w_2} -\bz^{1i;2\otimes j}_{u_1u_2;v_1v_2;w_1w_2} \rvert^{2p}}{|u_2-u_1|^{2\ka_1 p+2}|v_2-v_1|^{2\ka_2 p+2}|w_2-w_1|^{2\ka_2 p+2}}du_1du_2dv_1dv_2dw_1dw_2\notag\\
&= c\,U^3_{\ka_1,\ka_2,\ka_2,p}\bigl(\bx^{1i;2\otimes j}-\bz^{1i;2\otimes j}\bigl)^{2p} ,\label{E:U^1<U^3}
\end{align}
where $U^3_{\ka_1,\ka_2,\ka_2,p}\bigl(\bx^{1i;2\otimes j}-\bz^{1i;2\otimes j}\bigl)$ is defined in \eqref{E:U^3-def}. Inserting \eqref{E:U^1<U^3} for the first term on the right-hand side of \eqref{E:holder-(2,3)}, and taking the $2p$-th root of the resulting inequality, we conclude~\eqref{E:x^(1i;2.j)-z^(1i;2.j)}. This finishes our proof.
\end{proof}
The bound \eqref{eq:GRR-x^{1i;0.j}} in Lemma~\ref{L:general-2d-GRR-sign}-(ii) still involves the norm $\lVert \bz^{1i;2\cdot j}-\bx^{1i;2\cdot j}\rVert_{2\ga_{1}+\eps;\ga_{2}}$ on its right-hand side, which is itself an intermediate quantity rather than a directly checkable condition on the input paths. Lemma~\ref{L:x-z^(1i;2.j)} was designed precisely to control this term via~\eqref{E:x^(1i;2.j)-z^(1i;2.j)}. Substituting \eqref{E:x^(1i;2.j)-z^(1i;2.j)} into \eqref{eq:GRR-x^{1i;0.j}} and taking the $2p$-th root, we obtain the following self-contained estimate.
\begin{corollary}\label{cor:grr-x-z-1i0.j}
Under the assumptions of Lemma~\ref{L:general-2d-GRR-sign}-(ii) and Lemma~\ref{L:x-z^(1i;2.j)}, the Hölder norm of $\bx^{1i;0\cdot j}-\bz^{1i;0\cdot j}$ is controlled entirely by the integrability functionals $U^{2}$, $U^{3}$ and product Hölder norms of the input paths $x$ and $z$:
\begin{align}\label{eq:GRR-x-z^(1i;0.j)}
\bigl\lVert \bx^{1i;0\cdot j} - \bz^{1i;0\cdot j} \bigr\rVert_{2\gamma_{1};\gamma_{2}} &\le
c\,\Bigl(U^{2}_{2\ga_{1},\ga_{2},p}(\bz^{1i;0\cdot j}-\bx^{1i;0\cdot j}) + U^3_{2\ga_1+\eps,\ga_2,\ga_2,p}\bigl(\bx^{1i;2\otimes j}-\bz^{1i;2\otimes j}\bigl) \notag\\
&\qquad +\bigl\lVert \bz^{1;0}\,\bz^{i;j} -\bx^{1;0}\,\bx^{i;j}\bigr\rVert_{2\ga_{1};\ga_{2}}+\ \bigl\lVert \bx^{1;2}\bx^{i;j}-\bz^{1;2}\bz^{i;j} \bigl\rVert_{2\ga_1+\eps,\ga_2}\Bigr).
\end{align}
\end{corollary}

\subsection{Fractional Brownian sheet}

With the multiparameter GRR estimates of the preceding subsection in hand, we now carry out the main construction. That is, as mentioned in the introduction to this section, we consider a fractional Brownian sheet $x$. We will lift this field to a geometric rough sheet $\XX$ satisfying Hypotheses~\ref{hyp:rs1}, \ref{hyp:rs2}, and~\ref{hyp:rs3}, so that the rough-Young integral of Theorem~\ref{th:main-yr-integral} and the change of variable formula of Proposition~\ref{prop:ch-var} apply to $x$. In this section we will use a complete probability space $(\Omega,\mathcal{F},P)$ on which $x$ is defined. The relevant Hurst parameter regime is $H_{1}>1/3$ and $H_{2}>1/2$.

The construction follows the regularization strategy of Remark~\ref{Rmk:rs-regularization}. Using the harmonizable representation of $x$ (see \eqref{eq:harmonizable-x} below), we approximate $x$ by a family of smooth random fields $x^{N}$, obtained by restricting the spectral integral to frequencies $|\xi|,|\eta|\le N$. For each $N$, the associated rough sheet $\XX^{N}$ is constructed explicitly by computing the iterated integrals via the partial derivatives of $x^{N}$. The GRR bounds of Lemma~\ref{L:general-2d-GRR-sign} and Corollary~\ref{cor:grr-x-z-1i0.j} then imply that $(\XX^{N})_{N\ge 1}$ is a Cauchy sequence in the appropriate Hölder topology, and the limit defines the desired rough sheet $\XX$ above $x$.
We begin with the precise definition of the fBs.
\begin{definition}\label{def:fbs}
The process $\{x_{s;t} : (s, t) \in [0, 1]^2\}$ defined on the probability space
$(\Omega, \mathcal{F}, P)$ is called a fractional Brownian sheet (fBs) with
Hurst parameters $H_1, H_2 \in (0, 1)$ if $x$ is a centered Gaussian process with
covariance function
$R_{s_1s_2;t_1t_2}=\EE[x_{s_1;t_1}x_{s_2;t_2}]$ defined by
\begin{align}
R_{s_1 s_2; t_1 t_2}
= \frac{1}{4}
\left(
|s_1|^{2 H_1} + |s_2|^{2 H_1} - |s_1 - s_2|^{2 H_1}
\right)
\left(
|t_1|^{2 H_2} + |t_2|^{2 H_2} - |t_1 - t_2|^{2 H_2}
\right).
\end{align}
With this definition, the fBs $x$ can be represented (see e.g.~\cite{ST}) as
follows:
\begin{equation}\label{eq:harmonizable-x}
x_{s;t} \stackrel{d}{=}
\mathbf{c}
\int_{\mathbb{R}^2}
\frac{e^{is\xi} - 1}{|\xi|^{H_1 + 1/2}}
\frac{e^{it\eta} - 1}{|\eta|^{H_2 + 1/2}}
\hat{W}(d\xi, d\eta),
\end{equation}
where $\hat{W}$ is the Fourier transform of the white noise $W$ and the constant $\mathbf{c}$ depends on $H_1$ and $H_2$ only. The
representation \eqref{eq:harmonizable-x} is called harmonizable
representation for the fBs.
\end{definition}

\begin{remark}\label{rmk:Malliavin}
Since $x$ is defined via the Fourier transform $\hat{W}$ of a white noise $W$ on $\mathbb{R}^{2}$, we work throughout Section~\ref{sec:fbs} with the Wiener space built on the Hilbert space $\mathfrak{H} := L^{2}(\mathbb{R}^{2})$. For $n \ge 1$ and a symmetric kernel $f \in \mathfrak{H}^{\otimes n}$, we write $I_{n}(f)$ for the $n$-th Wiener iterated integral with respect to $\hat{W}$. These integrals generate the Wiener chaos decomposition $L^{2}(\Omega) = \bigoplus_{n=0}^{\infty} \mathcal{H}_{n}$, where $\mathcal{H}_{n}$ is the $n$-th Wiener chaos. A key identity used below (see~\eqref{E:partial-diff}) is the product formula: for $f, g \in \mathfrak{H}$,
\begin{equation}\label{E:product-formula}
  I_{1}(f)\,I_{1}(g) = \mathbb{E}\bigl[I_{1}(f)\,I_{1}(g)\bigr] + I_{2}(f \otimes g),
\end{equation}
where $f \otimes g$ denotes the standard tensor product. We refer to~\cite{Nu-bk} for a systematic treatment of Malliavin calculus, Wiener chaos, and iterated integrals.
\end{remark}

Since our rough sheet $\mathbb{X}$ in Table~\ref{table:rp} is defined by
regularization, see e.g.~Remark~\ref{Rmk:rs-regularization}, we introduce a regularizing parameter $N\ge 1$ and the
following sequence of random fields approximating the fBs $x$:
\begin{align}\label{eq:harmonizable-x^N}
x_{s;t}^N
:=
\mathbf{c}
\int_{D_{N}}
\frac{e^{is\xi} - 1}{|\xi|^{H_1 + 1/2}}
\frac{e^{it\eta} - 1}{|\eta|^{H_2 + 1/2}}
\hat{W}(d\xi, d\eta),
\end{align}
where $\mathbf{c}$ and $\hat{W}$ appear in~\eqref{eq:harmonizable-x}, and $D_{N}:=[-N,N]^{2}$. It is easily seen that $x^{N}$ is a smooth field, and we have
\begin{equation}\label{E:partial1}
\partial_1 x^N_{u;v}=\mathbf{c}\int_{D_{N}} \mathbf{K}_{1}(u,v,\xi,\eta)\,\hat{W}(d\xi, d\eta),
\quad
\partial_2 x^N_{u;v}=\mathbf{c}\int_{D_{N}} \mathbf{K}_{2}(u,v,\xi,\eta)\,\hat{W}(d\xi, d\eta),
\end{equation}
where, for $(\xi,\eta)\in\mathbb{R}^{2}$, we define the following kernels:
\begin{equation}\label{E:K-defs}
\mathbf{K}_{1}(u,v,\xi,\eta) := \frac{i\xi\,e^{iu\xi}(e^{iv\eta}-1)}{|\xi|^{H_{1}+1/2}|\eta|^{H_{2}+1/2}},
\quad\text{and}\quad
\mathbf{K}_{2}(u,v,\xi,\eta) := \frac{(e^{iu\xi}-1)\,i\eta\,e^{iv\eta}}{|\xi|^{H_{1}+1/2}|\eta|^{H_{2}+1/2}} .
\end{equation}
In the same way, one can write
\begin{equation}\label{E:partial12}
\partial_{12}x_{s;t}^N=\mathbf{c}\int_{D_{N}}\mathbf{K}_{12}(s,t,\xi,\eta)\,\hat{W}(d\xi, d\eta),
\quad\text{with}\quad
\mathbf{K}_{12}(s,t,\xi,\eta) := \frac{-\xi\eta\,e^{is\xi+it\eta}}{|\xi|^{H_{1}+1/2}|\eta|^{H_{2}+1/2}}.
\end{equation}
Denote by $\mathbb{X}^N$ the rough sheet containing iterated integrals constructed from $x^N$. Our main goal in this section is to show that the fBs can be enhanced in a rough sheet $\mathbb{X}$, where all the associated iterated integrals are constructed as limit of the integrals of $\mathbb{X}^N$.

Now, let %$\ga_1=H_1-\varepsilon$ and $\ga_2=H_2-\varepsilon$ for some $\varepsilon>0$ small enough so that 
$\ga_1\in (1/3,H_1)$ and $\ga_2\in (1/2,H_2)$. Using Lemma \ref{L:general-2d-GRR-sign} with $z=x^N$, we have the following approximation of integrals of the rough sheet $\mathbb{X}$ by integrals of $\mathbb{X}^N$:  
\begin{proposition}\label{prop:iterated-xN}
Let $x$ be a fBs with Hurst parameters $(H_1,H_2)$ such that $1/3<H_1<1/2$, $H_2>1/2$, as introduced in Definition~\ref{def:fbs}. Let  $\ga_1\in (1/3,H_1)$, $\ga_2\in (1/2,H_2)$, and $p>1$. Then the following convergences hold true:
\begin{align}
\lim_{N\rightarrow \infty}&\EE\Bigl[\bigl\lVert \bx^{N;1;2}-\bx^{1;2}\rVert_{\ga_1;\ga_2}^{2p}\Bigl]=0,\label{E:approxim-x^(1,2)}\\
\lim_{N,M\rightarrow \infty}&\EE\Bigl[\bigl\lVert \bx^{M;\hone;\htwo}-\bx^{N;\hone;\htwo}\rVert_{\ga_1;\ga_2}^{2p}\Bigl]=0,\label{E:approxim-x^(hone,htwo)}\\
\lim_{M,N\rightarrow \infty}&\EE\Bigl[\bigl\lVert \bx^{M;11;0\cdot2} - \bx^{N;11;0\cdot2} \bigl\rVert^{2p}_{2\gamma_1;\gamma_2}\Bigl]=0,\label{E:approxim-x^(11,02)}\\
\lim_{M,N\rightarrow \infty}&\EE\Bigl[\bigl\lVert \bx^{M;1\hone;0\cdot\htwo} - \bx^{N;1\hone;0\cdot\htwo} \bigl\rVert^{2p}_{2\gamma_1;\gamma_2}\Bigl]=0.\label{E:approxim-x^(1hone,0htow)}
\end{align}
\end{proposition}
\begin{proof}
We prove the four convergences \eqref{E:approxim-x^(1,2)}-\eqref{E:approxim-x^(1hone,0htow)} in turn. The arguments for \eqref{E:approxim-x^(1,2)} and~\eqref{E:approxim-x^(hone,htwo)} rest on direct second-moment estimates using the harmonizable representation, combined with dominated convergence. The proof of \eqref{E:approxim-x^(11,02)} additionally requires the GRR bound of Corollary~\ref{cor:grr-x-z-1i0.j}. The argument for \eqref{E:approxim-x^(1hone,0htow)} is entirely analogous and is left to the reader.

\smallskip
\noindent
\emph{Step 1: Convergence of $\bx^{N;1;2}$}. We follow the same lines as \cite[Section 6.1, Step 1, p.\,43]{CG}. Namely by \eqref{eq:GRR-x^{1;2}} and Gaussian hypercontractivity we have
\begin{eqnarray}\label{eq:E-norm-x^(1;2)}
\EE\Bigl[\bigl\lVert \bx^{N;1;2}-\bx^{1;2}\rVert_{\ga_1;\ga_2}^{2p}\Bigl]
&\le& 
c_2 \EE\Bigl[\big(U^2_{\ga_1,\ga_2,p}(\bx^{N;1;2}-\bx^{1;2})\big)^{2p}\Bigl] \\
&\le& \int_{[0,1]^4}\frac{ \EE\Bigl[|\bx^{N;1;2}_{u_1u_2;v_1v_2}-\bx^{1;2}_{u_1u_2;v_1v_2}|^{2p}\Bigl]}{|u_2-u_1|^{2\ga_1 p+2}|v_2-v_1|^{2\ga_2 p+2}}du_1du_2dv_1dv_2\notag\\
&\le& \int_{[0,1]^4}\frac{\EE\Bigl[|\bx^{N;1;2}_{u_1u_2;v_1v_2}-\bx^{1;2}_{u_1u_2;v_1v_2}|^{2}\Bigl]^{p}}{|u_2-u_1|^{2\ga_1 p+2}|v_2-v_1|^{2\ga_2 p+2}}du_1du_2dv_1dv_2.\nonumber
\end{eqnarray}
The harmonizable representations of $x$ and $x^N$ (see \eqref{eq:harmonizable-x} and \eqref{eq:harmonizable-x^N}), plus an elementary change of variable, ensure that 
\begin{align}\label{eq:2nd-moment-x^{i,j}}
&\EE\Bigl[|\bx^{N;1;2}_{u_1u_2;v_1v_2}-\bx^{1;2}_{u_1u_2;v_1v_2}|^{2}\Bigl]
=\mathbf{c}\int_{|\xi|\vee |\eta|\ge N}\frac{\rvert e^{iu_2\xi}-e^{iu_1\xi}\lvert^2\rvert e^{iv_2\eta}-e^{iv_1\eta}\lvert^2}{|\xi|^{2H_1+1}|\eta|^{2H_2+1}}d\xi d\eta\notag\\[2ex]
&\qquad\le \mathbf{c} |u_2-u_1|^{2H_1}|v_2-v_1|^{2H_2}\, I_N\bigl(u_1,u_2;v_1,v_2\bigl),
\end{align}
where we set
\begin{equation*}
I_N\bigl(u_1,u_2;v_1,v_2\bigl) := \int_{\bigl\{|\xi|\vee |\eta|\ge N|u_2-u_1||v_2-v_1|\bigl\}}\frac{|e^{i\xi}-1|^2|e^{i\eta}-1|^2}{|\xi|^{2H_1+1}|\eta|^{2H_2+1}}d\xi\, d\eta.
\end{equation*}
Plugging \eqref{eq:2nd-moment-x^{i,j}} into \eqref{eq:E-norm-x^(1;2)}, we thus obtain
\begin{align}\label{eq:E-norm-x^(1;2)2}
    \EE&\Bigl[\bigl\lVert \bx^{N;1;2}-\bx^{1;2}\rVert_{\ga_1;\ga_2}^{2p}\Bigl]\notag\\
    &\le \int_{[0,1]^4}{|u_2-u_1|^{2(H_1-\ga_1) p-2}|v_2-v_1|^{2(H_2-\ga_2) p-2}}I_N\bigl(u_1,u_2;v_1,v_2\bigl)^p du_1du_2dv_1dv_2. 
\end{align}
Notice that for all $H_1, H_2>1/3$, we have the following uniform bound over $(u_1,u_2,v_1,v_2)\in [0,1]^4$ and $N\in \NN$:
\begin{equation*}
I_N\bigl(u_1,u_2;v_1,v_2\bigl)
\le
\int_{\RR^{2}}\frac{|e^{i\xi}-1|^2|e^{i\eta}-1|^2}{|\xi|^{2H_1+1}|\eta|^{2H_2+1}}d\xi\, d\eta
<\infty.
\end{equation*}
Hence by dominated convergence we get
\begin{align*}
\lim_{N\rightarrow \infty}I_N\bigl(u_1,u_2;v_1,v_2\bigl)=0.
\end{align*}
%for all $u_1, u_2, v_1,v_2\in [0,1]$ with $u_1\neq u_2$ and $v_1\neq v_2$. 
Since $\ga_1<H_1$ and $\ga_2<H_2$, we can choose $p$ large enough and use the dominated convergence theorem again, so that the integral in \eqref{eq:E-norm-x^(1;2)2} converge to zero as $N\rightarrow 0$. This yields
\begin{equation}
\lim_{N\to\infty}\EE\Bigl[\bigl\lVert \bx^{N;1;2}-\bx^{1;2}\rVert_{\ga_1;\ga_2}^{2p}\Bigl]
=
0 .
\end{equation}

\smallskip
\noindent
\emph{Step 2: Convergence of $\bx^{N;\hone;\htwo}$}. Since $\bx^{\hone;\htwo}$ is not automatically defined (unlike $\bx^{1;2}$) as a rectangular increment of $x$, we show that $\{\bx^{N;\hone;\htwo}\,:\,N\in \NN\}$ is a Cauchy sequence in $L^{2p}(\Omega,\mathcal{P}_{2,2}^{\ga_1,\ga_2})$. To this aim, let $M,N\in \NN$ such that $M<N$. Using \eqref{eq:GRR-x^{1;2}} and the Gaussian hypercontractivity again, we get 
\begin{align}\label{eq:E-norm-x^(hone;htwo)}
\EE\Bigl[\bigl\lVert \bx^{N;\hone;\htwo}-\bx^{M;\hone;\htwo}\rVert_{\ga_1;\ga_2}^{2p}\Bigl]
\le \int_{[0,1]^4}\frac{\EE\Bigl[|\bx^{N;\hone;\htwo}_{u_1u_2;v_1v_2}-\bx^{M;\hone;\htwo}_{u_1u_2;v_1v_2}|^{2}\Bigl]^{p}}{|u_2-u_1|^{2\ga_1 p+2}|v_2-v_1|^{2\ga_2 p+2}}du_1du_2dv_1dv_2.
\end{align}

Recalling~\eqref{E:partial1}, it is readily checked that the product formula~\eqref{E:product-formula} implies
\begin{align} \label{E:partial-diff}
&\partial_1 x^N_{u;v}\partial_2 x^N_{u;v}-\partial_1 x^M_{u;v}\partial_2 x^M_{u;v}
= \mathcal{J}_{MN}(u,v) + I_2\bigl(\mathbf{1}_{D_N^2\setminus D_M^2}\mathbf{K}(u,v;\cdot)\bigr),
\end{align}
where $D_{N}$, $\mathbf{K}_{1}$, $\mathbf{K}_{2}$ are as in~\eqref{E:K-defs}, $D_{N}^{2}:=D_{N}\times D_{N}\subset\mathbb{R}^{4}$, and we set
\begin{align}\label{f12}
\mathbf{K}(u,v;\xi_{1},\eta_{1},\xi_{2},\eta_{2}) &:= \mathbf{K}_{1}(u,v,\xi_{1},\eta_{1})\otimes\mathbf{K}_{2}(u,v,\xi_{2},\eta_{2})
\notag\\
&\phantom{:}=\frac{-\xi_{1}\eta_{2}\,e^{iu\xi_{1}}(e^{iu\xi_{2}}-1)\,e^{iv\eta_{2}}(e^{iv\eta_{1}}-1)}{\prod_{j=1}^{2}|\xi_{j}|^{H_{1}+1/2}|\eta_{j}|^{H_{2}+1/2}}, 
\end{align}
as well as
\begin{equation}\label{f11}
\mathcal{J}_{MN}(u,v) := \int_{D_{N}\setminus D_{M}}\mathbf{K}_{1}(u,v,\xi,\eta)\,\overline{\mathbf{K}_{2}(u,v,\xi,\eta)}\,d\xi\,d\eta.
\end{equation}
Therefore we have obtained
\begin{align}
\bx^{N;\hone;\htwo}_{u_1u_2;v_1v_2}-\bx^{M;\hone;\htwo}_{u_1u_2;v_1v_2}
&\overset{\mathrm{def}}{=} \int_{u_1}^{u_2}\int_{v_1}^{v_2}\bigl(\partial_{1}x^{N}_{u;v}\,\partial_{2}x^{N}_{u;v}-\partial_{1}x^{M}_{u;v}\,\partial_{2}x^{M}_{u;v}\bigr)du\,dv \label{f2}\\
&\overset{\eqref{E:partial-diff}}{=} \int_{u_1}^{u_2}\int_{v_1}^{v_2}\Bigl[\mathcal{J}_{MN}(u,v) + I_2\bigl(\mathbf{1}_{D_N^2\setminus D_M^2}\mathbf{K}(u,v;\cdot)\bigr)\Bigr]du\,dv \notag
\end{align}
Next starting from~\eqref{f2}, we set
\begin{align}
\mathcal{K}^{\hone;\htwo}_{u_1u_2;v_1v_2}(\cdot) &:= \int_{u_1}^{u_2}\int_{v_1}^{v_2}\mathbf{K}(u,v;\xi_1,\cdot)\,du\,dv,
\label{f4}\\
\mathcal{J}_{MN}^{\hone;\htwo}(u_{1}u_{2};v_{1}v_{2}) &:= \int_{u_1}^{u_2}\int_{v_1}^{v_2}\mathcal{J}_{MN}(u,v)\,du\,dv,
\label{f5}
\end{align}
and we apply the stochastic Fubini theorem. This yields
\begin{equation}\label{E:Wick-xMN}
\bx^{N;\hone;\htwo}_{u_1u_2;v_1v_2}-\bx^{M;\hone;\htwo}_{u_1u_2;v_1v_2}
= I_2\Bigl(\mathbf{1}_{D_N^2\setminus D_M^2}\mathcal{K}^{\hone;\htwo}_{u_1u_2;v_1v_2}\Bigr) 
+ \mathcal{J}_{MN}^{\hone;\htwo}(u_{1}u_{2};v_{1}v_{2}).
\end{equation}
Hence using the fact that $\mathcal{J}_{MN}^{\hone;\htwo}$ is deterministic, the second moment of~\eqref{E:Wick-xMN} reads
\begin{equation}\label{E:2nd-chaos-expectation}
\mathbb{E}\Bigl[\Bigl|\bx^{N;\hone;\htwo}_{u_1u_2;v_1v_2}-\bx^{M;\hone;\htwo}_{u_1u_2;v_1v_2}\Bigr|^{2}\Bigr] = \mathcal{A}^{1}_{MN}(u_1u_2;v_1v_2) + \mathcal{A}^{2}_{MN}(u_1u_2;v_1v_2),
\end{equation}
where we set
\begin{align}
\mathcal{A}^{1}_{MN}(u_1u_2;v_1v_2) &:= \bigl(\mathcal{J}_{MN}^{\hone;\htwo}(u_1u_2;v_1v_2)\bigr)^{2},\label{E:A1-def}\\
\mathcal{A}^{2}_{MN}(u_1u_2;v_1v_2) &:= \EE\Bigl[I_{2}\Bigl(\mathbf{1}_{D_{N}^{2}\setminus D_{M}^{2}}\mathcal{K}^{\hone;\htwo}_{u_1u_2;v_1v_2}\Bigr)^{2}\Bigr].\label{E:A2-def}
\end{align}
Below we provide separate bounds on $\mathcal{A}^{1}_{MN}$ and $\mathcal{A}^{2}_{MN}$, respectively.

\smallskip
\noindent
\emph{Step 3: Bound on $\mathcal{A}^{1}_{MN}$.} 
We first bound the deterministic term~\eqref{E:A1-def}, where $\mathcal{J}_{MN}^{\hone;\htwo}$ is given by~\eqref{f5}. To this aim, observe that according to~\eqref{E:partial-diff}  and~\eqref{f11}, we have $\mathcal{J}_{MN}(u,v)=\mathcal{J}_{N}(u,v)-\mathcal{J}_{M}(u,v)$, where $\mathcal{J}_{N}(u,v)$  can be expressed as
\begin{equation*}
\mathcal{J}_N  
=\,\int_{D_N} \frac{i\xi e^{iu\xi}(e^{iv\eta}-1)}{|\xi|^{H_1+1/2}|\eta|^{H_2+1/2}}
\frac{(-i\eta) e^{-iv\eta}(e^{-iu\xi}-1)}{|\xi|^{H_1+1/2}|\eta|^{H_2+1/2}}d\xi d\eta .
\end{equation*}
Hence some elementary manipulations show that
\begin{align}\label{E:Jn-def}
\mathcal{J}_N 
&=\,{\int_{-N}^N \frac{\xi(1 - e^{iu\xi})}{|\xi|^{2H_1+1}} d\xi}{ \int_{-N}^N \frac{\eta(1 - e^{-iv\eta})}{|\eta|^{2H_2+1}} d\eta }=-(2i)^2\int_0^{N}\frac{\sin(u\xi)}{|\xi|^{2H_1}}d\xi\int_0^{N}\frac{\sin(u\eta)}{|\eta|^{2H_2}}d\eta\nonumber\\
&= 4 u^{2H_1-1} v^{2H_2-1}\,\cs_{H_{1}}(uN)\,\cs_{H_{2}}(vN),
\end{align}
where for $i=1,2$ we define
\begin{equation*}
\cs_{H_{i}}(A):=\int_{0}^{A}\frac{\sin(x)}{x^{2H_{i}}}dx .
\end{equation*}
Now, using the elementary identity $AB - ab = A(B-b) + b(A-a)$ we obtain
\begin{align}\label{E:Jmn-formula}
&\mathcal{J}_{MN}(u,v) = \mathcal{J}_N(u,v) - \mathcal{J}_M(u,v) \\
&= 4 u^{2H_1-1} v^{2H_2-1} \Bigl[ \mathcal{S}_{H_1}(uN) \bigl( \mathcal{S}_{H_2}(vN) - \mathcal{S}_{H_2}(vM) \bigr) 
 + \mathcal{S}_{H_2}(vM) \bigl( \mathcal{S}_{H_1}(uN) - \mathcal{S}_{H_1}(uM) \bigr) \Bigr].\nonumber
\end{align}
Since $\mathcal{S}_{H_{i}}(x)$ converges to a finite limit as $x \to \infty$, the differences $\mathcal{S}_{H_{i}}(B) - \mathcal{S}_{H_{i}}(A)$ are uniformly bounded. Specifically, for $i=1,2$ we have
\begin{equation}\label{E:integ-bnd1}
\sup_{0<A<B<\infty} \bigl|\mathcal{S}_{H_i}(B)-\mathcal{S}_{H_i}(A)\bigl| \le C,
\end{equation}
where $C$ is a positive constant depends on $H_i$ only. Next we use integration by parts in order to control the term $\mathcal{S}_{H_1}(uN) - \mathcal{S}_{H_1}(uM)$ in terms of $M$. Namely we have
\begin{align}\label{E:integ-bnd2}
\Bigl\lvert \mathcal{S}_{H_i}(B) - \mathcal{S}_{H_i}(A)\Bigl\rvert &=\Bigl\rvert\int_A^B x^{-2H_i}\sin(x)dx\Bigl\rvert\nonumber \\
&= \Bigl\rvert\Bigl[ -x^{-2H_i}\cos(x) \Bigr]_A^B - \int_A^B 2H_i x^{-2H_i-1}\cos(x)dx\Bigl\lvert \le C A^{-2H_i}.
\end{align}
Combining \eqref{E:integ-bnd1} and \eqref{E:integ-bnd2}, then for any $\theta\in (0,1)$ as small as possible but fixed, we have (recall that $C$ is a generic constant which can change from line to line):
\begin{equation*}
\Bigl\lvert \mathcal{S}_{H_i}(B) - \mathcal{S}_{H_i}(A)\Bigl\rvert\le C^{1-\theta}A^{-2\theta H_i}.
\end{equation*}
Plugging this inequality in \eqref{E:Jmn-formula}, we obtain
\begin{align}\label{E:Jmn-bnd}
\bigl|\cj_{MN}(u,v)\bigl|\le C\,M^{-2\theta H_1} u^{2(H_1-\theta)-1} v^{2(H_2-\theta)-1}. 
\end{align}
We now choose $\theta \equiv \eps < (H_{1}-\gamma_{1})\wedge(H_{2}-\gamma_{2})$ small enough. By definition~\eqref{f5} of $\mathcal{J}^{\hone;\htwo}_{MN}$ and the triangle inequality,
\begin{equation*}
\bigl|\mathcal{J}^{\hone;\htwo}_{MN}(u_{1},u_{2};v_{1},v_{2})\bigr| \le \int_{u_{1}}^{u_{2}}\int_{v_{1}}^{v_{2}} \bigl|\mathcal{J}_{MN}(u,v)\bigr|\,du\,dv.
\end{equation*}
Inserting the pointwise bound~\eqref{E:Jmn-bnd} and using the fact that the integrand factors, we compute each integral explicitly:
\begin{equation*}
\int_{u_{1}}^{u_{2}} u^{2(H_{1}-\eps)-1}\,du = \frac{u_{2}^{2(H_{1}-\eps)}-u_{1}^{2(H_{1}-\eps)}}{2(H_{1}-\eps)}, \qquad \int_{v_{1}}^{v_{2}} v^{2(H_{2}-\eps)-1}\,dv = \frac{v_{2}^{2(H_{2}-\eps)}-v_{1}^{2(H_{2}-\eps)}}{2(H_{2}-\eps)},
\end{equation*}
which gives 
\begin{align}
\bigl|\mathcal{J}^{\hone;\htwo}_{MN}(u_{1},u_{2};v_{1},v_{2})\bigr| &\le \frac{C M^{-2\eps H_1}}{4(H_1-\eps)(H_2-\eps)} \left( u_2^{2(H_1-\eps)} - u_1^{2(H_1-\eps)} \right) \left( v_2^{2(H_2-\eps)} - v_1^{2(H_2-\eps)} \right) \nonumber \\
&\le c\, M^{-2\eps H_1} \bigl|u_2-u_1\bigl|^{H_1-\eps} \bigl|v_2-v_1\bigl|^{H_2-\eps}.\label{E:C-tilde-MN}
\end{align}
Recalling from~\eqref{E:A1-def} that $\mathcal{A}^{1}_{MN} = \bigl(\mathcal{J}^{\hone;\htwo}_{MN}\bigr)^{2}$, we square~\eqref{E:C-tilde-MN} to conclude:
\begin{equation}\label{E:A1-bnd}
\mathcal{A}^{1}_{MN}(u_{1},u_{2};v_{1},v_{2}) \le C\,M^{-4\eps H_{1}}\,|u_{2}-u_{1}|^{2(H_{1}-\eps)}\,|v_{2}-v_{1}|^{2(H_{2}-\eps)}.
\end{equation}
This completes the bound on $\mathcal{A}^{1}_{MN}$.

\smallskip
\noindent
\emph{Step 4: Bound on $\mathcal{A}^{2}_{MN}$.} We now bound the Wiener chaos term~\eqref{E:A2-def}. By the It\^{o} isometry applied to $\mathcal{A}^{2}_{MN}$,
\begin{align}\label{E:expected-I_2(K)}
\mathcal{A}^{2}_{MN}(u_1u_2;v_1v_2)
=
\EE\Bigl[I_2\Bigl(\mathbf{1}_{D_N^2\setminus D_M^2}\mathcal{K}^{\hone;\htwo}_{u_1u_2;v_1v_2}\Bigl)^2\Bigl] %&=2!\bigl\lVert \widetilde{\mathbf{1}_{D_N^2\setminus D_M^2}\mathcal{K}^{\hone;\htwo}_{u_1u_2;v_1v_2}}\bigl\rVert_{L^2}^2
\le 2 \bigl\lVert \mathbf{1}_{D_N^2\setminus D_M^2}\mathcal{K}^{\hone;\htwo}_{u_1u_2;v_1v_2}\bigl\rVert_{L^2}^2,
\end{align}
%where $\widetilde{\mathbf{1}_{D_N^2\setminus D_M^2}\mathcal{K}}$ is the symmetrization of $\mathbf{1}_{D_N^2\setminus D_M^2}\ck$.      
% $$
% \widetilde{\mathcal{K}}_{u_1u_2;v_1v_2}(z_1; z_2) := \frac{1}{2} \Bigl( \mathcal{K}_{u_1u_2;v_1v_2}(z_1; z_2) + \mathcal{K}_{u_1u_2;v_1v_2}(z_2; z_1) \Bigl)\quad \text{with\  $z_i=(\xi_i,\eta_i)$ for $i=1,2$.}
% $$
%
% \begin{align*}
% | \widetilde{\mathcal{K}} |_{L^2}^2 &= \int \int \left| \frac{\mathcal{K}(z_1, z_2) + \mathcal{K}(z_2, z_1)}{2} \right|^2 dz_1 dz_2 \\
% &\le \int \int \frac{1}{2} \left( |\mathcal{K}(z_1, z_2)|^2 + |\mathcal{K}(z_2, z_1)|^2 \right) dz_1 dz_2 \\
% &= \frac{1}{2} \int \int |\mathcal{K}(z_1, z_2)|^2 dz_1 dz_2 + \frac{1}{2} \int \int |\mathcal{K}(z_2, z_1)|^2 dz_1 dz_2 \\
% &= \int \int |\mathcal{K}(z_1, z_2)|^2 dz_1 dz_2 \\
% &= | \mathcal{K} |_{L^2}^2
% \end{align*}
%
% \begin{align}\label{E:ito-iso-ubnd}
% \mathbb{E}\Bigl[I_2(\mathcal{K}_{u_1u_2;v_1v_2})^2\Bigl] &\le 2 \int_{D_N^2 \setminus D_M^2} |\mathcal{K}_{u_1u_2;v_1v_2}(\xi_1, \eta_1, \xi_2, \eta_2)|^2 d\xi_1 d\eta_1 d\xi_2 d\eta_2.
% \end{align}
Recall from~\eqref{f4} that $\mathcal{K}^{\hone,\htwo}_{u_1u_2;v_1v_2}(\xi_1,\xi_2;\eta_1,\eta_2) = \int_{u_1}^{u_2}\int_{v_1}^{v_2}\mathbf{K}(u,v;\xi_1,\xi_2,\eta_1,\eta_2)\,du\,dv$, where $\mathbf{K}$ is given by~\eqref{f12}. Substituting~\eqref{f12} and observing that the integrand factors into a function of $(u,\xi_1,\xi_2)$ times a function of $(v,\eta_1,\eta_2)$, we get
\begin{multline*}
\mathcal{K}^{\hone,\htwo}_{u_1u_2;v_1v_2}(\xi_1,\xi_2;\eta_1,\eta_2) \\
= \frac{-\xi_1\,\eta_2}{\prod_{j=1}^{2}|\xi_j|^{H_1+1/2}|\eta_j|^{H_2+1/2}} \left(\int_{u_1}^{u_2}e^{iu\xi_1}(e^{iu\xi_2}-1)\,du\right)\left(\int_{v_1}^{v_2}e^{iv\eta_2}(e^{iv\eta_1}-1)\,dv\right).
\end{multline*}
Setting
\begin{equation}\label{E:I-defs}
\mathcal{I}^{\hone}_{u_1 u_2}(\xi_1,\xi_2) := \int_{u_1}^{u_2}e^{iu\xi_1}(e^{iu\xi_2}-1)\,du, \qquad \mathcal{I}^{\htwo}_{v_1 v_2}(\eta_1,\eta_2) := \int_{v_1}^{v_2}e^{iv\eta_2}(e^{iv\eta_1}-1)\,dv,
\end{equation}
we obtain
\begin{align}\label{E:kernel-K}
\mathcal{K}^{\hone,\htwo}_{u_1u_2;v_1v_2}(\xi_1,\xi_2;\eta_1,\eta_2) = \frac{-\xi_1  \ci^{\hone}_{u_1 u_2}(\xi_1, \xi_2) \, \eta_2 \ci^{\htwo}_{v_1 v_2}(\eta_1, \eta_2)}{\prod_{j=1}^2 |\xi_j|^{H_1+1/2}|\eta_j|^{H_2+1/2}}.
\end{align}
The expressions for $\mathcal{I}^{\hone}_{u_1 u_2}$ and $\mathcal{I}^{\htwo}_{v_1 v_2}$ will be further decomposed as follows. We set
\begin{align*}
\mathcal{I}_{u_1 u_2}^{\hone'}(\xi_1, \xi_2) &:= i\xi_2\int_{u_1}^{u_2}e^{iu\xi_1}\int_{u_1}^{u}e^{iu'\xi_2}du'du, \quad &
\mathcal{I}_{u_1 u_2}^{\hone''}(\xi_1, \xi_2) &:= \bigl(e^{iu_1\xi_2}-1\bigl)\frac{e^{iu_2\xi_1}-e^{iu_1\xi_1}}{i\xi_1},\\
\mathcal{I}_{v_1 v_2}^{\htwo'}(\eta_1, \eta_2) &:= i\eta_1\int_{v_1}^{v_2}e^{iv\eta_2}\int_{v_1}^{v}e^{iv'\eta_1}dv'dv, \quad &
\mathcal{I}_{v_1 v_2}^{\htwo''}(\eta_1, \eta_2) &:= \bigl(e^{iv_1\eta_1}-1\bigl)\frac{e^{iv_2\eta_2}-e^{iv_1\eta_2}}{i\eta_2}.
\end{align*}
Then according to \eqref{E:I-defs} we have
\begin{eqnarray}\label{E:I_u=I1+I2}
\mathcal{I}^{\hone}_{u_1 u_2}(\xi_1, \xi_2)&=&\int_{u_1}^{u_2}e^{iu\xi_1}\bigl(e^{iu\xi_2}-1\bigl)du \notag\\
&=&\,\int_{u_1}^{u_2}e^{iu\xi_1}\bigl(e^{iu\xi_2}-e^{iu_1\xi_2}\bigl)du\,+\,\bigl(e^{iu_1\xi_2}-1\bigl)\int_{u_1}^{u_2}e^{iu\xi_1}du
\nonumber\\
&=&\, \mathcal{I}_{u_1 u_2}^{\hone'}(\xi_1, \xi_2) + \mathcal{I}_{u_1 u_2}^{\hone''}(\xi_1, \xi_2).
\end{eqnarray}
In the same way, we let the reader check that
\begin{equation}\label{E:I_v=I1+I2}
\mathcal{I}^{\htwo}_{v_1 v_2}(\eta_1, \eta_2)
=\int_{v_1}^{v_2}e^{iv\eta_2}\bigl(e^{iv\eta_1}-1\bigl)dv
= \mathcal{I}_{v_1 v_2}^{\htwo'}(\eta_1, \eta_2) + \mathcal{I}_{v_1 v_2}^{\htwo''}(\eta_1, \eta_2).
\end{equation}
We now simplify $\mathcal{I}_{u_1 u_2}^{\hone'}$  by a linear change of variables. To this aim, observe from~\eqref{E:I_u=I1+I2} that
\begin{equation}\label{f51}
\mathcal{I}_{u_1 u_2}^{\hone'}(\xi_1,\xi_2) = i\xi_2\int_{u_1}^{u_2}e^{iu\xi_1}\int_{u_1}^{u}e^{iu'\xi_2}du'du
\end{equation}
 Then we set
\begin{equation*}
s = \frac{u-u_1}{u_2-u_1}, \qquad t = \frac{u'-u_1}{u_2-u_1},
\end{equation*}
so that $u = u_1+s(u_2-u_1)$ and $u' = u_1+t(u_2-u_1)$. After some elementary computations we end up with
\begin{equation}\label{E:I1-D}
\mathcal{I}_{u_1 u_2}^{\hone'}(\xi_1, \xi_2)=i\xi_2\,(u_2-u_1)^2\,e^{iu_1(\xi_1+\xi_2)}\,\mathcal{D}\bigl((u_2-u_1)\xi_1,(u_2-u_1)\xi_2\bigl) ,
\end{equation}
where we have set
\begin{equation}\label{f52}
\cd(x_1,x_2):=\int_0^1e^{i u x_1}\int_0^{u}e^{i u' x_2}du'du .
\end{equation}
Similarly, one can also check that
\begin{equation}
\mathcal{I}_{v_1 v_2}^{\htwo'}(\eta_1, \eta_2)=i\eta_1\,(v_2-v_1)^2\,e^{iv_1(\eta_1+\eta_2)}\,\mathcal{D}\bigl((v_2-v_1)\eta_2,(v_2-v_1)\eta_1\bigl).\label{E:I2-D}
\end{equation}

Now recall from \eqref{E:expected-I_2(K)} and \eqref{E:kernel-K} that we wish to integrate kernels like $\mathcal{I}_{u_1 u_2}^{\hone'}, \, \mathcal{I}_{v_1 v_2}^{\htwo'}$.  Having this in mind, we define a new function
\begin{equation*}
\phi_{\alpha}(z_1,z_2):=\frac{|\cd(z_1,z_2)|^2}{|z_1|^{2\alpha-1}|z_2|^{2\alpha-1}}.
\end{equation*}
Then another elementary change of variables reveals that 
\begin{align}
\int_{\RR^2}\frac{\bigl\lvert\xi_1\ci_{u_1 u_2}^{\hone'}(\xi_1,\xi_2)\bigl\rvert^2}{|\xi_1|^{2H_1+1}|\xi_2|^{2H_1+1}} d\xi_1 d\xi_2 &=(u_2-u_1)^{4H_1}\int_{\RR^2}\phi_{H_1}(x_1,x_2)\,dx_1\,dx_2 = c\,(u_2-u_1)^{4H_1},\label{E:int-I^1_u}\\
\int_{\RR^2}\frac{\bigl\lvert\eta_2\ci_{v_1 v_2}^{\htwo'}(\eta_1,\eta_2)\bigl\rvert^2}{|\eta_1|^{2H_2+1}|\eta_2|^{2H_2+1}} d\eta_1 d\eta_2 &=(v_2-v_1)^{4H_2}\int_{\RR^2}\phi_{H_2}(y_1,y_2)\,dy_1\,dy_2 = c\,(v_2-v_1)^{4H_2}, \label{E:int-I^1_v}
\end{align}
where for the second identities in~\eqref{E:int-I^1_u} and~\eqref{E:int-I^1_v} we have used the following fact, which can be checked by elementary integration tools  for any $1/3<\alpha<1$ (for our purposes $\alpha \in \{H_1, H_2\}$):
\begin{equation}\label{L:6.4-CG}
\int_{\RR^2}\phi_{\alpha}(z_1,z_2)\,dz_1\,dz_2 = \int_{\RR^2}\frac{\bigl\lvert\cd(z_1,z_2)\bigl\lvert^2}{|z_1|^{2\alpha-1} |z_2|^{2\alpha-1}}\,dz_1\,dz_2<\infty.
\end{equation}

Let us now focus on the terms $\ci_{u_1 u_2}^{\hone''}, \, \ci_{u_1 u_2}^{\htwo''}$ in~\eqref{E:I_u=I1+I2}-\eqref{E:I_v=I1+I2}. 
For those terms we follow the same kind of elementary change of variables and integrations as for~\eqref{E:int-I^1_u}-\eqref{E:int-I^1_v}. We end up with integrals of the form
\begin{equation*}
\int_{\RR^2}\frac{\bigl\lvert e^{ia z_k}-1\bigl\lvert^2 \bigl\lvert e^{i z_{\ell}}-1\bigl\lvert^2}{ |z_k|^{2\alpha+1}|z_{\ell}|^{2\alpha+1}}dz_{k}dz_{\ell}  .
\end{equation*}
The analysis of those integrals involve the interpolation inequality $\lvert e^{ix}-1\rvert \le 2^{1-\theta}|x|^{\theta}$ for  $|x|<1$ and  $\theta \in [0,1]$, which implies that for any $1/3<\alpha<1$, $\eps>0$ small and $(k,\ell)=(1,2),\ (2,1)$, we have
\begin{equation*}
\int_{\RR^2}\frac{\bigl\lvert e^{ia z_k}-1\bigl\lvert^2 \bigl\lvert e^{i z_{\ell}}-1\bigl\lvert^2}{ |z_k|^{2\alpha+1}|z_{\ell}|^{2\alpha+1}}dz_{k}dz_{\ell}  
\le c\, |a|^{2(\alpha+\eps)}\int_{\RR^2}\psi_{\alpha,\eps}(z_k,z_{\ell})dz_kdz_{\ell} ,
\end{equation*}
where 
\begin{equation*}
\psi_{\alpha,\eps}(z_1,z_2):=\frac{\bigl\lvert e^{iz_1}-1\bigl\lvert^2}{|z_1|^{2\alpha+1}}\Bigl(|z_2|^{-2\alpha-1}\textbf{1}_{\{|z_2|>1\}}+|z_2|^{-1+2\eps}\textbf{1}_{\{|z_2|\le 1\}}\Bigl).
\end{equation*} 
This type of arguments easily yields
\begin{align}
\int_{\RR^2}\frac{\bigl\lvert\xi_1\ci_{u_1 u_2}^{\hone''}(\xi_1,\xi_2)\bigl\rvert^2}{|\xi_1|^{2H_1+1}|\xi_2|^{2H_1+1}} d\xi_1 d\xi_2 & =(u_2-u_1)^{4H_1}\int_{\RR^2}\frac{\bigl\lvert e^{ix_1}-1\bigl\lvert^2 \bigl\lvert e^{i\frac{u_1}{(u_2-u_1)}x_2}-1\bigl\lvert^2}{|x_1|^{2H_1+1} |x_2|^{2H_1+1}}dx_1dx_2\notag\\
&\le \,(u_2-u_1)^{2(H_1-\eps)}\,\int_{\RR^2}\psi_{H_1,\eps}(x_1,x_2)dx_1dx_2,\label{E:int-I^2_u} 
\end{align}
as well as
\begin{align}
\int_{\RR^2}\frac{\bigl\lvert\eta_2\ci_{v_1 v_2}^{\htwo''}(\eta_1,\eta_2)\bigl\rvert^2}{|\eta_1|^{2H_2+1}|\eta_2|^{2H_2+1}} d\eta_1 d\eta_2 &=(v_2-v_1)^{4H_2}\int_{\RR^2}\frac{\bigl\lvert e^{i\frac{v_1}{(v_2-v_1)}y_1}-1\bigl\lvert^2 \bigl\lvert e^{iy_2}-1\bigl\lvert^2}{|y_1|^{2H_2+1} |y_2|^{2H_2+1}}dy_1dy_2\notag\\
&\le  \,(v_2-v_1)^{2(H_2-\eps)}\,\int_{\RR^2}\psi_{H_2,\eps}(y_2,y_1)dy_1dy_2 .\label{E:int-I^2_v}
\end{align}

We now assemble the $L^{2}$ bound on $\mathcal{K}^{\hone;\htwo}_{u_1u_2;v_1v_2}$. Squaring~\eqref{E:kernel-K} and using $|AB|^{2}=|A|^{2}|B|^{2}$, the integrand already factors:
\begin{equation*}
\bigl\lvert\mathcal{K}^{\hone;\htwo}_{u_1u_2;v_1v_2}(\xi_1,\xi_2;\eta_1,\eta_2)\bigl\rvert^{2}
= \frac{\bigl\lvert\xi_1\,\ci^{\hone}_{u_1 u_2}(\xi_1,\xi_2)\bigl\rvert^{2}}{|\xi_1|^{2H_1+1}|\xi_2|^{2H_1+1}}
\cdot
\frac{\bigl\lvert\eta_2\,\ci^{\htwo}_{v_1 v_2}(\eta_1,\eta_2)\bigl\rvert^{2}}{|\eta_1|^{2H_2+1}|\eta_2|^{2H_2+1}}.
\end{equation*}
Plugging the decompositions~\eqref{E:I_u=I1+I2}--\eqref{E:I_v=I1+I2} into each factor and applying $(a+b)^{2}\le 2(a^{2}+b^{2})$, we bound the $(\xi_1,\xi_2)$-factor by a sum involving $\ci^{\hone'}_{u_1 u_2}$ and $\ci^{\hone''}_{u_1 u_2}$, and similarly for the $(\eta_1,\eta_2)$-factor. Reporting the bounds~\eqref{E:int-I^1_u}--\eqref{E:int-I^2_u} for the $(\xi_1,\xi_2)$-factor and~\eqref{E:int-I^1_v}--\eqref{E:int-I^2_v} for the $(\eta_1,\eta_2)$-factor into the full integral, we obtain
\begin{align}\label{E:L^2-K-integral}
\int_{\RR^4}&\Bigl\lvert\mathcal{K}^{\hone;\htwo}_{u_1u_2;v_1v_2}(\xi_1,\xi_2;\eta_1,\eta_2)\Bigl\lvert^2 d\xi_1d\xi_2d\eta_1d\eta_2\notag\\
&\le (u_2-u_1)^{2(H_1-\eps)} (v_2-v_1)^{2(H_2-\eps)}{\int_{\RR^4}\Psi^{\hone;\htwo}(x_1,x_2;y_1,y_2)dx_1dx_2dy_1dy_2<\infty},
\end{align}
where $\Psi^{\hone;\htwo}$ is spelled out as
\begin{equation}\label{f6}
\Psi^{\hone;\htwo}(x_1,x_2;y_1,y_2):=\bigl(\phi_{H_1}+\psi_{H_1,\eps}\bigl)(x_1,x_2)\bigl(\phi_{H_2}+\psi_{H_2,\eps}\bigl)(y_2,y_1).
\end{equation}
Moreover, the same steps as for the integral over $\RR^{4}$ yield
\begin{align}\label{E:L^2-tails}
\bigl\lVert& \mathbf{1}_{D_N^2\setminus D_M^2}{\mathcal{K}}^{\hone;\htwo}_{u_1u_2;v_1v_2}\bigl\rVert_{L^2(\RR^4)}^2 =\int_{D_N^2\setminus D_M^2}\Bigl\lvert\mathcal{K}^{\hone;\htwo}_{u_1u_2;v_1v_2}(\xi_1,\xi_2;\eta_1,\eta_2)\Bigl\lvert^2 d\xi_1d\xi_2d\eta_1d\eta_2\\
&\le (u_2-u_1)^{2(H_1-\eps)} (v_2-v_1)^{2(H_2-\eps)}\,{I}'_{M}(u_1,u_2,v_1,v_2),\notag
\end{align}
where we set
\begin{equation*}
{I}'_{M}(u_1,u_2,v_1,v_2) := \int_{\substack{\lVert (x_1,x_2)\rVert_{\infty}\ge M/|u_2-u_1|\\ \lVert (y_1,y_2)\rVert_{\infty}\ge M/|v_2-v_1|}}\Psi^{\hone;\htwo}(x_1,x_2;y_1,y_2)\,dx_1\,dx_2\,dy_1\,dy_2.
\end{equation*}
Recalling~\eqref{E:expected-I_2(K)} and inserting~\eqref{E:L^2-tails}, we conclude:
\begin{equation}\label{E:A2-bnd}
\mathcal{A}^{2}_{MN}(u_1u_2;v_1v_2) 
\le C\,(u_2-u_1)^{2(H_1-\eps)}\,(v_2-v_1)^{2(H_2-\eps)}\,{I}'_{M}(u_1,u_2,v_1,v_2).
\end{equation}
This completes the bound on $\mathcal{A}^{2}_{MN}$.

\smallskip
\noindent
\emph{Step 5: Conclusion of Step~2.} Inserting the bounds~\eqref{E:A1-bnd} on $\mathcal{A}^{1}_{MN}$ from Step~3 and~\eqref{E:A2-bnd} on $\mathcal{A}^{2}_{MN}$ from Step~4 into~\eqref{E:2nd-chaos-expectation}, we obtain
\begin{align}\label{E:expected-x^(MN1^2^)}
\mathbb{E} \left[ \left| \bx^{N;\hone;\htwo}_{u_1u_2;v_1v_2}-\bx^{M;\hone;\htwo}_{u_1u_2;v_1v_2} \right|^2 \right]\le c(u_2-u_1)^{2(H_1-\eps)} (v_2-v_1)^{2(H_2-\eps)}\,{I}^{\hone;\htwo}_{M}(u_1,u_2,v_1,v_2),
\end{align}
where we set
\begin{equation*}
{I}^{\hone;\htwo}_{M}(u_1,u_2,v_1,v_2) := \frac{1}{M^{4\eps H_1}} + {I}'_{M}(u_1,u_2,v_1,v_2).
\end{equation*}
Plugging this into \eqref{eq:E-norm-x^(hone;htwo)}, we obtain
\begin{equation}\label{f7}
  \EE\Bigl[\bigl\lVert \bx^{N;\hone;\htwo}-\bx^{M;\hone;\htwo}\rVert_{\ga_1;\ga_2}^{2p}\Bigl]\le \int_{[0,1]^4}\frac{{I}^{\hone;\htwo}_M\bigl(u_1,u_2;v_1,v_2\bigl)^p du_1du_2dv_1dv_2}{|u_2-u_1|^{2(\ga_1-H_1+\eps) p+2}|v_2-v_1|^{2(\ga_2-H_2+\eps) p+2}} .
\end{equation}
In addition, it is readily seen that the function $\Psi^{\hone;\htwo}$ introduced in \eqref{f6} is an element of $L^1(\RR^4)$. Therefore it is clear that
\begin{align}
\sup_{\substack{(u_1,u_2,v_1,v_2)\in [0,1]^4,\ M\in \NN}}{I}^{\hone;\htwo}_M\bigl(u_1,u_2;v_1,v_2\bigl)<\infty \quad\text{and}\quad
\lim_{M\rightarrow \infty}{I}^{\hone;\htwo}_M\bigl(u_1,u_2;v_1,v_2\bigl)=0.\label{E:dom-cv-cond}
\end{align}
Then, choosing $p>1$ in \eqref{f7} large enough and using the dominated convergence theorem, we obtain that 
\begin{equation}
\lim_{N\rightarrow \infty}\bx^{N;\hone;\htwo}
\quad\text{exists in}\quad
L^{p}\lp  \Omega;  \cp_{2,2}^{\ga_1,\ga_2}\rp .
\end{equation}

\smallskip
\noindent
\emph{Step 6: Convergence of $\bx^{N;11;0\cdot2}$ and $\bx^{N;1\hone;0\cdot\htwo}$}. 
The analysis of $\bx^{N;11;0\cdot2}$ relies on the bound~\eqref{eq:GRR-x-z^(1i;0.j)} from Corollary~\ref{cor:grr-x-z-1i0.j}, applied to $z=x^{N}$ and $x$ replaced by $x^{M}$. We get
\begin{align}\label{E:expec-x^(MN1102)}
&\EE\Bigl[\bigl\lVert \bx^{N;11;0\cdot 2} - \bx^{M;11;0\cdot 2} \bigr\rVert_{2\gamma_{1};\gamma_{2}}^{2p}\Bigl]
\le c\,\Bigl(\mathcal{B}^{1}_{MN} + \mathcal{B}^{2}_{MN} + \mathcal{B}^{3}_{MN} + \mathcal{B}^{4}_{MN}\Bigr),
\end{align}
where we set
\begin{align}
\mathcal{B}^{1}_{MN} &:= \EE\bigl[U^{2}_{2\ga_{1},\ga_{2},p}(\bx^{N;11;0\cdot 2}-\bx^{M;11;0\cdot 2})^{2p}\bigl],\label{E:B1-def}\\
\mathcal{B}^{2}_{MN} &:= \EE\bigl[U^{3}_{2\ga_1+\eps,\ga_2,\ga_2,p}\bigl(\bx^{N;11;2\otimes 2}-\bx^{M;11;2\otimes 2}\bigl)^{2p}\bigl],\label{E:B2-def}\\
\mathcal{B}^{3}_{MN} &:= \EE\bigl[\bigl\lVert \bx^{N;1;0}\,\bx^{N;1;2} -\bx^{M;1;0}\,\bx^{M;1;2}\bigr\rVert_{2\ga_{1};\ga_{2}}^{2p}\bigl],\label{E:B3-def}\\
\mathcal{B}^{4}_{MN} &:= \EE\bigl[\bigl\lVert \bx^{N;1;2}\bx^{N;1;2}-\bx^{M;1;2}\bx^{M;1;2} \bigl\rVert_{2\ga_1+\eps,\ga_2}^{2p}\bigl].\label{E:B4-def}
\end{align}
Below we provide separate bounds on each $\mathcal{B}^{i}_{MN}$. In fact it is not hard to show, using the canonical definition of $\bx^{1;0}$, \eqref{E:approxim-x^(1,2)} and \eqref{E:approxim-x^(hone,htwo)}, that $\mathcal{B}^{3}_{MN}$ and $\mathcal{B}^{4}_{MN}$ go to zero as $M,N\rightarrow \infty$. Therefore for sake of conciseness, we will focus on $\mathcal{B}^{1}_{MN}$ and $\mathcal{B}^{2}_{MN}$. Now to bound $\mathcal{B}^{1}_{MN}$ and $\mathcal{B}^{2}_{MN}$, we use a similar argument to the proof of \eqref{E:approxim-x^(hone,htwo)}. Beginning with $\mathcal{B}^{1}_{MN}$, similarly to~\eqref{eq:E-norm-x^(hone;htwo)} we have
\begin{align}\label{E:exp-U2-x^MN1102-bnd}
\mathcal{B}^{1}_{MN} \le \int_{[0,1]^4}\frac{\EE\Bigl[|\bx^{N;11;0\cdot2}_{u_1u_2;v_1v_2}-\bx^{M;11;0\cdot2}_{u_1u_2;v_1v_2}|^{2}\Bigl]^{p}}{|u_2-u_1|^{4\ga_1 p+2}|v_2-v_1|^{2\ga_2 p+2}}du_1du_2dv_1dv_2.
\end{align}
In order to upper bound the integrand in the right hand side of \eqref{E:exp-U2-x^MN1102-bnd} we recall definition~\eqref{E:def-bx-1102}, which gives
\begin{align}\label{f8}
\bx^{N;11;0\cdot2}_{u_1u_2;v_1v_2}-\bx^{M;11;0\cdot2}_{u_1u_2;v_1v_2}=\int_{u_1<\sigma_1<\sigma_2<u_2}d\sigma_1d\sigma_2\int_{v_1}^{v_2}d\tau_2\bigl(\partial_1x^N_{\sigma_1;v_1}\partial_{12}x^N_{\sigma_2;\tau_2}-\partial_1x^M_{\sigma_1;v_1}\partial_{12}x^M_{\sigma_2;\tau_2}\bigl).
\end{align}
Using the product formula \eqref{E:product-formula} for the second Wiener chaos, we decompose the product of the derivatives into a second chaos term and a zeroth chaos term. Similarly to what we did in order to obtain~\eqref{f2}, for $M < N$ we get:
\begin{align}\label{E:Wick-xMN-1102}
\bx^{N;11;0\cdot2}_{u_1u_2;v_1v_2}-\bx^{M;11;0\cdot2}_{u_1u_2;v_1v_2} 
&= I_2\Bigl(\mathbf{1}_{D_N^2\setminus D_M^2} \mathcal{K}^{11;0\cdot 2}_{u_1u_2;v_1v_2} \Bigl) \ +\ \bigl\lVert \mathbf{1}_{D_N\setminus D_M}{\mathcal{J}}^{11;0\cdot 2}_{u_1u_2;v_1v_2}\bigl\rVert_{L^1(\RR^2)} ,
\end{align}
where the kernel $\mathcal{K}^{11;0\cdot 2}_{u_1u_2;v_1v_2}\in L^2(\RR^4)$ is defined by
\begin{equation}\label{E:kernel-K-1102}
\mathcal{K}^{11;0\cdot 2}_{u_1u_2;v_1v_2}(\xi_1, \eta_1, \xi_2, \eta_2) 
:= 
-\frac{\xi_1 \xi_2  \bigl(\int_{u_1}^{u_2}\int_{u_1}^{\sigma_2} e^{i\sigma_1 \xi_1} e^{i\sigma_2 \xi_2} d\sigma_1 d\sigma_2\bigl)  (e^{iv_1\eta_1}-1)\bigl(\int_{v_1}^{v_2} i\eta_2 e^{i\tau_2 \eta_2} d\tau_2\bigl)}{\prod_{j=1}^2 |\xi_j|^{H_1+1/2}|\eta_j|^{H_2+1/2}}.
\end{equation}
and where the function ${\mathcal{J}}^{11;0\cdot 2}_{u_1u_2;v_1v_2}\in L^1(\RR^2)$ are given by
\begin{equation}\label{E:def-J-1102}
\cj^{11;0\cdot2}_{u_1u_2;v_1v_2}(\xi,\eta):=\frac{  \bigl(\int_{u_1}^{u_2}\int_{u_1}^{\sigma_1}i\xi e^{i\sigma_1 \xi}\,\overline{ i\xi e^{i\sigma_2 \xi}} d\sigma_1 d\sigma_2 \bigl)  (e^{iv_1\eta}-1)\bigl(\int_{v_1}^{v_2} \overline{i\eta e^{i\tau_2 \eta} }d\tau_2\bigl)}{|\xi|^{2H_1+1}|\eta|^{2H_2+1}}.
\end{equation}
As we will see below, we have that $\mathcal{K}^{11;0\cdot 2}_{u_1u_2;v_1v_2}\in L^2(\RR^4)$ and ${\mathcal{J}}^{11;0\cdot 2}_{u_1u_2;v_1v_2}\in L^1(\RR^2)$.

Our next task is to get a more enlightening expression for $\mathcal{K}^{11;0\cdot 2}$. To this aim, we simplify each factor in \eqref{E:kernel-K-1102} in turn. The computation is similar to \eqref{E:I1-D}--\eqref{E:I2-D}, and we mostly highlight the differences below.

\noindent
\emph{$u$-factor.} The double integral $\int_{u_1}^{u_2}\int_{u_1}^{\sigma_2} e^{i\sigma_1\xi_{1}}e^{i\sigma_2\xi_{2}}\,d\sigma_1\,d\sigma_2$ has the same structure as the term $\mathcal{I}_{u_1 u_2}^{\hone'}(\xi_1,\xi_2)/(i\xi_2)$ from \eqref{f51}, with one difference: the outer integral carries the frequency $\xi_{2}$ (via $e^{i\sigma_2\xi_{2}}$) rather than $\xi_{1}$. Applying the same change of variables $s = (\sigma_1-u_1)/(u_2-u_1)$, $t = (\sigma_2-u_1)/(u_2-u_1)$ as in the derivation of \eqref{E:I1-D} therefore gives, 
\begin{equation*}
\int_{u_1}^{u_2}\int_{u_1}^{\sigma_2} e^{i\sigma_1\xi_{1}}e^{i\sigma_2\xi_{2}}\,d\sigma_1\,d\sigma_2 = (u_2-u_1)^2\,e^{iu_{1}(\xi_{1}+\xi_{2})}\,\mathcal{D}\bigl((u_2-u_1)\xi_{2},(u_2-u_1)\xi_{1}\bigr),
\end{equation*}
where we recall that the function $\cd$ is defined by~\eqref{f52}.

\noindent
\emph{$v$-factor.} Unlike the double integral in $\mathcal{I}_{v_1 v_2}^{\htwo'}$ appearing in \eqref{E:I2-D}, the $v$-contribution to \eqref{E:kernel-K-1102} is already a product of two separate single-variable factors. The term $(e^{iv_{1}\eta_{1}}-1)$ comes from evaluating $\partial_{1}x^{N}_{\sigma_1;v}$ at $v = v_1$ (the zeroth-order factor in the second direction), while the single integral evaluates directly as
\begin{equation*}
\int_{v_1}^{v_2} i\eta_{2}\,e^{i\tau_{2}\eta_{2}}\,d\tau_2 = e^{iv_{2}\eta_{2}}-e^{iv_{1}\eta_{2}}.
\end{equation*}

\noindent
Inserting both simplifications into \eqref{E:kernel-K-1102} yields
\begin{multline}\label{E:kernel-K-1102-D}
\mathcal{K}^{11;0\cdot 2}_{u_1u_2;v_1v_2}(\xi_1, \eta_1, \xi_2, \eta_2) \\
= \frac{-\xi_1\xi_2\,(u_2-u_1)^2\,e^{iu_{1}(\xi_{1}+\xi_{2})}\,\mathcal{D}\bigl((u_2-u_1)\xi_{2},(u_2-u_1)\xi_{1}\bigr)\,(e^{iv_{1}\eta_{1}}-1)\,(e^{iv_{2}\eta_{2}}-e^{iv_{1}\eta_{2}})}{\prod_{j=1}^2 |\xi_j|^{H_1+1/2}|\eta_j|^{H_2+1/2}}.
\end{multline}
With the same type of consideration, one can also go from~\eqref{E:def-J-1102} to the following expression:
\begin{multline}\label{E:kernel-J-1102}
\cj^{11;0\cdot2}_{u_1u_2;v_1v_2}(\xi,\eta)\\
=\frac{\xi^2 e^{i(u_1-u_2)\xi}(u_2-u_1)^2\mathcal{D}\bigl((u_2-u_1)\xi,-(u_2-u_1)\xi\bigl)(e^{iv_1\eta}-1)(e^{-iv_2\eta}-e^{-iv_1\eta})}{|\xi|^{2H_1+1}|\eta|^{2H_2+1}}.
\end{multline}
The integrability of these kernels follows from the same steps as in \eqref{E:L^2-K-integral}.
Taking the second moment of \eqref{E:Wick-xMN-1102}, and using similar argument as in \eqref{E:expected-x^(MN1^2^)}, we let the reader go through the tedious computations leading to
\begin{align}
\EE\Bigl[\bigl| \bx^{N;11;0\cdot2}_{u_1u_2;v_1v_2}-\bx^{M;11;0\cdot2}_{u_1u_2;v_1v_2} \bigr|^2\Bigl] &\le  c(u_2-u_1)^{4(H_1-\eps)} (v_2-v_1)^{2(H_2-\eps)} {I}^{\,11;0\cdot2}_{M}(u_1,u_2,v_1,v_2),\label{E:exp-x^MN1102-bnd}
\end{align}
where ${I}^{\,11;0\cdot2}_{M}(u_1,u_2,v_1,v_2)$ satisfies similar conditions to \eqref{E:dom-cv-cond}. Plugging \eqref{E:exp-x^MN1102-bnd} back into \eqref{E:exp-U2-x^MN1102-bnd}, as $\ga_1 < H_1$ and $\ga_2 < H_2$, choosing $\eps > 0$ sufficiently small and using the dominated convergence theorem then yields:
\begin{equation}\label{E:U^2-limit}
\lim_{M,N\rightarrow \infty} \mathcal{B}^{1}_{MN} = 0.
\end{equation}

Finally, we estimate the term $\mathcal{B}^{2}_{MN}$ defined by~\eqref{E:B2-def}. The argument is entirely analogous to the one just carried out for $\mathcal{B}^{1}_{MN}$: we expand the difference using Definition~\ref{D:x^(1i;2otimesj)}, decompose into Wiener chaos via the product formula~\eqref{E:product-formula}, bound the resulting kernels as in~\eqref{E:L^2-K-integral}, and conclude by dominated convergence. The only structural difference from $\mathcal{B}^{1}_{MN}$ is that $\bx^{N;11;2\otimes 2}$ carries two independent direction-$2$ integrations (in $\tau_{1}$ and $\tau_{2}$) rather than one, which introduces the extra set of variables $(w_{1},w_{2})$ in the kernels.  More specifically, by definition~\eqref{E:x^(1i;2xj)} we can write
\begin{align*}
&\bx^{N;11;2\otimes 2}_{u_1u_2;v_1v_2;w_1w_2}-\bx^{M;11;2\otimes 2}_{u_1u_2;v_1v_2;w_1w_2}\\
&\quad=\int_{u_1<\sigma_1<\sigma_2<u_2}d\sigma_1d\sigma_2\int_{v_1}^{v_2}d\tau_1\int_{w_1}^{w_2}d\tau_2\bigl(\partial_{12}x^N_{\sigma_1;\tau_1}\partial_{12}x^N_{\sigma_2;\tau_2}-\partial_{12}x^M_{\sigma_1;\tau_1}\partial_{12}x^M_{\sigma_2;\tau_2}\bigl).
\end{align*}
Applying the product formula~\eqref{E:product-formula} exactly as in~\eqref{E:Wick-xMN-1102}, for $M < N$ we get
\begin{align}%\label{E:Wick-xMN-112x2}
\bx^{N;11;2\otimes 2}_{u_1u_2;v_1v_2;w_1w_2}-\bx^{M;11;2\otimes 2}_{u_1u_2;v_1v_2;w_1w_2}
&= I_2\Bigl(\mathbf{1}_{D_N^2\setminus D_M^2} \mathcal{K}^{11;2\otimes 2}_{u_1u_2;v_1v_2;w_1w_2} \Bigl) + \bigl\lVert \mathbf{1}_{D_N\setminus D_M}{\mathcal{J}}^{11;2\otimes 2}_{u_1u_2;v_1v_2;w_1w_2}\bigl\rVert_{L^1(\RR^2)} ,
\end{align}
where the kernels $\mathcal{K}^{11;2\otimes 2}$ and $\mathcal{J}^{11;2\otimes 2}$ play the same role as $\mathcal{K}^{11;0\cdot 2}$ and $\mathcal{J}^{11;0\cdot 2}$ in~\eqref{E:Wick-xMN-1102}. They are expressed directly in their simplified form, analogously to~\eqref{E:kernel-K-1102-D} and~\eqref{E:kernel-J-1102}: the $u$-factor involves the same $\mathcal{D}$-function from~\eqref{f52} obtained by the change of variables used in the derivation of~\eqref{E:kernel-K-1102-D}, while each direction-$2$ factor evaluates to a single exponential difference. That is the kernel $\mathcal{K}^{11;2\otimes 2}$ is given by 
\begin{align}\label{E:kernel-K-112x2}
&\mathcal{K}^{11;2\otimes 2}_{u_1u_2;v_1v_2;w_1w_2}(\xi_1, \eta_1, \xi_2, \eta_2) \\
&:= \frac{i^4\xi_1\eta_1 \xi_2\eta_2 \bigl(\int_{u_1}^{u_2}\int_{u_1}^{\sigma_2} e^{i\sigma_1 \xi_1} e^{i\sigma_2 \xi_2} d\sigma_1 d\sigma_2\bigl) \bigl(\int_{v_1}^{v_2} e^{i\tau_1 \eta_1} d\tau_1\bigl) \bigl(\int_{w_1}^{w_2} e^{i\tau_2 \eta_2} d\tau_2\bigl)}{\prod_{j=1}^2 |\xi_j|^{H_1+1/2}|\eta_j|^{H_2+1/2}} \notag\\
&=\frac{-\xi_1\xi_2e^{i(u_1\xi_1+u_2\xi_2)}(u_2-u_1)^2\mathcal{D}\bigl((u_2-u_1)\xi_1,(u_2-u_1)\xi_2\bigl)(e^{iv_2\eta_1}-e^{iv_1\eta_1})(e^{iw_2\eta_2}-e^{iw_1\eta_2})}{\prod_{j=1}^2 |\xi_j|^{H_1+1/2}|\eta_j|^{H_2+1/2}},\notag
\end{align}
while the function $\cj^{11;2\otimes 2}$ is spelled out as
\begin{align}\label{E:kernel-J-112x2}
&\cj^{11;2\otimes 2}_{u_1u_2;v_1v_2;w_1w_2}(\xi,\eta)\\
&\quad=\frac{\xi^2 e^{i(u_1-u_2)\xi}(u_2-u_1)^2\mathcal{D}\bigl((u_2-u_1)\xi,-(u_2-u_1)\xi\bigl)(e^{iv_2\eta}-e^{iv_1\eta})(e^{-iw_2\eta}-e^{-iw_1\eta})}{|\xi|^{2H_1+1}|\eta|^{2H_2+1}}.\notag
\end{align}
The integrability of these kernels follows from the same steps as in~\eqref{E:L^2-K-integral}. Taking the second moment and arguing as in~\eqref{E:expected-x^(MN1^2^)} and~\eqref{E:exp-x^MN1102-bnd}, now with three pairs of integration variables corresponding to the $u$, $v$, and $w$-directions, we obtain
\begin{align}
&\EE\Bigl[\bigl| \bx^{N;11;2\otimes 2}_{u_1u_2;v_1v_2;w_1w_2}-\bx^{M;11;2\otimes 2}_{u_1u_2;v_1v_2;w_1w_2} \bigr|^2\Bigl] \notag\\
&\quad\le c(u_2-u_1)^{4(H_1-\eps)} (v_2-v_1)^{2(H_2-\eps)} (w_2-w_1)^{2(H_2-\eps)} {I}^{\,11;2\otimes 2}_{M}(u_1,u_2,v_1,v_2,w_1,w_2),\label{E:exp-x^MN112x2-bnd}
\end{align}
where ${I}^{\,11;2\otimes 2}_{M}$ is a tail integral satisfying the same conditions as in~\eqref{E:dom-cv-cond}. Plugging~\eqref{E:exp-x^MN112x2-bnd} into the definition~\eqref{E:B2-def}, one obtains a six-variable integral analogous to~\eqref{E:exp-U2-x^MN1102-bnd}, with the powers of $(u_{2}-u_{1})$, $(v_{2}-v_{1})$, and $(w_{2}-w_{1})$ in the numerator competing against the corresponding singularities in the denominator. The same dominated convergence argument as below~\eqref{E:exp-x^MN1102-bnd} then applies: since $\ga_1 < H_1$ and $\ga_2 < H_2$, one can choose $\eps > 0$ sufficiently small so that $4H_1 - 4\eps > 4\ga_1 + 2\eps$ and $2H_2 - 2\eps > 2\ga_2$, which guarantees that the integrated fractional moments are finite. Thus the dominated convergence theorem yields:
\begin{equation}\label{E:U^3-limit}
\lim_{M,N\rightarrow \infty} \mathcal{B}^{2}_{MN} = 0.
\end{equation}

To conclude the proof of~\eqref{E:approxim-x^(11,02)}, we gather the four limits established above. The convergences $\mathcal{B}^{1}_{MN}\to 0$ and $\mathcal{B}^{2}_{MN}\to 0$ follow from~\eqref{E:U^2-limit} and~\eqref{E:U^3-limit}, respectively, while $\mathcal{B}^{3}_{MN}\to 0$ and $\mathcal{B}^{4}_{MN}\to 0$ were recorded below~\eqref{E:B4-def} using the canonical definition of $\bx^{1;0}$ together with~\eqref{E:approxim-x^(1,2)} and~\eqref{E:approxim-x^(hone,htwo)}. Inserting all four limits into~\eqref{E:expec-x^(MN1102)}, we conclude that~\eqref{E:approxim-x^(11,02)} holds true. 

The proof of~\eqref{E:approxim-x^(1hone,0htow)} parallels Step~6, with $(i,j)=(1,2)$ replaced throughout by $(i,j)=(\hone,\htwo)$: applying Corollary~\ref{cor:grr-x-z-1i0.j} with these indices reduces the problem to four terms, of which the two product terms (involving $\bx^{N;1;0}\bx^{N;\hone;\htwo}$ and $\bx^{N;1;2}\bx^{N;\hone;\htwo}$) tend to zero by~\eqref{E:approxim-x^(1,2)} and~\eqref{E:approxim-x^(hone,htwo)}, while the $U^{2}$ and $U^{3}$ terms are controlled by second-moment estimates analogous to~\eqref{E:U^2-limit}--\eqref{E:U^3-limit}. The main difference from Step~6 is that the integrand of $\bx^{N;1\hone;0\cdot\htwo}$ is the triple product $\partial_{1}x^N \cdot \partial_{1}x^N \cdot \partial_{2}x^N$ (compare~\eqref{f8}, where the integrand was $\partial_{1}x^N \cdot \partial_{12}x^N$), so the Wiener chaos decomposition yields first- and third-order chaoses rather than the zeroth and second-order ones in~\eqref{E:Wick-xMN-1102}; the Fourier-space estimates are nonetheless of the same type as in Steps~3-4, and we leave the details to the reader.

\smallskip

\noindent
\emph{Step 7: Conclusion}. We have now established all four convergences of Proposition~\ref{prop:iterated-xN}. The convergence~\eqref{E:approxim-x^(1,2)} was proved in Step~1; \eqref{E:approxim-x^(hone,htwo)} in Steps~2--5 via the Wiener chaos decomposition~\eqref{E:Wick-xMN} and the second-moment bounds of Steps~3--4; \eqref{E:approxim-x^(11,02)} in Step~6 by inserting the vanishing of $\mathcal{B}^{1}_{MN},\ldots,\mathcal{B}^{4}_{MN}$ into~\eqref{E:expec-x^(MN1102)}; and \eqref{E:approxim-x^(1hone,0htow)} by the parallel argument outlined at the end of Step~6. This completes the proof.
\end{proof}

\begin{remark}
    It is worth noting that the proof of the existence of the limits $\bx^{\hone;\htwo}$ in $\cp_{2,2}^{\ga_1,\ga_2}$ and $\bx^{11;0\cdot2}$ in $\cp_{2,2}^{2\ga_1,\ga_2}$ does not fundamentally rely on the assumption that $H_2 > 1/2$. In fact, the argument remains entirely valid for any $H_2 > 1/3$. The core integrability estimates, particularly the convergence of the integrals \eqref{E:int-I^1_u}-\eqref{E:int-I^2_v}, hold symmetrically for both $H_1, H_2 \in (1/3, 1)$. 
    %The restriction $H_2 > 1/2$ is only required globally in our framework to ensure that the sum of the regularities $\ga_1 + \ga_2 > 1$, which is the crucial condition needed to apply the 2D Sewing Lemma in the subsequent steps.
\end{remark}
\bibliographystyle{abbrvnat}
\bibliography{ito-rough-bib.bib}

\begin{thebibliography}{21}
\providecommand{\natexlab}[1]{#1}
\providecommand{\url}[1]{\texttt{#1}}
\expandafter\ifx\csname urlstyle\endcsname\relax
  \providecommand{\doi}[1]{doi: #1}\else
  \providecommand{\doi}{doi: \begingroup \urlstyle{rm}\Url}\fi

\bibitem[Alberti et~al.(2024)Alberti, Stepanov, and Trevisan]{AST24}
G.~Alberti, E.~Stepanov, and D.~Trevisan.
\newblock Integration of nonsmooth 2-forms: from {Y}oung to {I}t\^o{} and
  {S}tratonovich.
\newblock \emph{J. Funct. Anal.}, 286\penalty0 (2):\penalty0 Paper No. 110212,
  37, 2024.
\newblock \doi{10.1016/j.jfa.2023.110212}.

\bibitem[Cass et~al.(2026)Cass, Crisan, Iannucci, and Turner]{cass2026}
T.~Cass, D.~Crisan, A.~Iannucci, and W.~F. Turner.
\newblock Signature kernel and schwinger-dyson kernel equations as
  two-parameter rough differential equations.
\newblock \emph{arxiv.org/abs/2605.08844}, 2026.

\bibitem[Chouk and Gubinelli(2014)]{CG}
K.~Chouk and M.~Gubinelli.
\newblock Rough sheets.
\newblock \emph{arxiv.org/pdf/1406.7748}, 2014.

\bibitem[Chouk and Tindel(2015)]{CT}
K.~Chouk and S.~Tindel.
\newblock Skorohod and {S}tratonovich integration in the plane.
\newblock \emph{Electron. J. Probab.}, 20:\penalty0 no. 39, 39, 2015.
\newblock \doi{10.1214/ejp.v20-3041}.

\bibitem[Diehl and Schmitz(2026)]{DS26}
J.~Diehl and L.~Schmitz.
\newblock Two-parameter sums signatures and corresponding quasisymmetric
  functions.
\newblock \emph{J. Algebraic Combin.}, 63\penalty0 (1):\penalty0 Paper No. 4,
  61, 2026.
\newblock \doi{10.1007/s10801-025-01478-4}.

\bibitem[Diehl et~al.(2025)Diehl, Ebrahimi-Fard, Harang, and Tindel]{DEHT25}
J.~Diehl, K.~Ebrahimi-Fard, F.~N. Harang, and S.~Tindel.
\newblock On the signature of an image.
\newblock \emph{Stochastic Process. Appl.}, 187:\penalty0 Paper No. 104661, 26,
  2025.
\newblock \doi{10.1016/j.spa.2025.104661}.

\bibitem[Friz and Hairer(2014)]{Friz2014}
P.~K. Friz and M.~Hairer.
\newblock \emph{A course on rough paths}.
\newblock Universitext. Springer, Cham, 2014.
\newblock ISBN 978-3-319-08331-5; 978-3-319-08332-2.
\newblock \doi{10.1007/978-3-319-08332-2}.
\newblock With an introduction to regularity structures.

\bibitem[Friz and Victoir(2010)]{Friz2010}
P.~K. Friz and N.~B. Victoir.
\newblock \emph{Multidimensional stochastic processes as rough paths}, volume
  120 of \emph{Cambridge Studies in Advanced Mathematics}.
\newblock Cambridge University Press, Cambridge, 2010.
\newblock ISBN 978-0-521-87607-0.
\newblock \doi{10.1017/CBO9780511845079}.
\newblock Theory and applications.

\bibitem[Gerasimovičs and Hairer(2019)]{HG2018}
A.~Gerasimovičs and M.~Hairer.
\newblock H\"ormander's theorem for semilinear {SPDE}s.
\newblock \emph{Electron. J. Probab.}, 24:\penalty0 Paper No. 132, 56, 2019.
\newblock \doi{10.1214/19-ejp387}.

\bibitem[Giusti et~al.(2025)Giusti, Lee, Nanda, and Oberhauser]{GLNO22}
C.~Giusti, D.~Lee, V.~Nanda, and H.~Oberhauser.
\newblock A topological approach to mapping space signatures.
\newblock \emph{Adv. in Appl. Math.}, 163:\penalty0 Paper No. 102787, 68, 2025.
\newblock \doi{10.1016/j.aam.2024.102787}.

\bibitem[Gubinelli(2004)]{Gu}
M.~Gubinelli.
\newblock Controlling rough paths.
\newblock \emph{J. Funct. Anal.}, 216\penalty0 (1):\penalty0 86--140, 2004.

\bibitem[Gubinelli and Tindel(2010)]{GT}
M.~Gubinelli and S.~Tindel.
\newblock Rough evolution equations.
\newblock \emph{Ann. Probab.}, 38\penalty0 (1):\penalty0 1--75, 2010.

\bibitem[Harang et~al.(2024)Harang, Tindel, and Wang]{HTW}
F.~Harang, S.~Tindel, and X.~Wang.
\newblock Volterra equations driven by rough signals 3: probabilistic
  construction of the {V}olterra rough path for fractional {B}rownian motions.
\newblock \emph{J. Theoret. Probab.}, 37\penalty0 (1):\penalty0 307--351, 2024.

\bibitem[Lee and Oberhauser(2026)]{LO24}
D.~Lee and H.~Oberhauser.
\newblock Random surfaces and higher algebra.
\newblock \emph{arxiv.org/abs/2311.08366}, 2026.

\bibitem[Nualart(2006)]{Nu-bk}
D.~Nualart.
\newblock \emph{The {M}alliavin calculus and related topics}.
\newblock Probability and its Applications (New York). Springer-Verlag, Berlin,
  second edition, 2006.
\newblock ISBN 978-3-540-28328-7; 3-540-28`328-5.

\bibitem[R{\'e}veillac et~al.(2012)R{\'e}veillac, Stauch, and
  Tudor]{reveillac2012hermite}
A.~R{\'e}veillac, M.~Stauch, and C.~A. Tudor.
\newblock Hermite variations of the fractional brownian sheet.
\newblock \emph{Stochastics and Dynamics}, 12\penalty0 (03):\penalty0 1150021,
  2012.

\bibitem[Samorodnitsky and Taqqu(1994)]{ST}
G.~Samorodnitsky and M.~S. Taqqu.
\newblock \emph{Stable Non-Gaussian Random Processes}.
\newblock Chapman and Hall, 1994.

\bibitem[Stepanov and Trevisan(2021)]{ST21}
E.~Stepanov and D.~Trevisan.
\newblock Towards geometric integration of rough differential forms.
\newblock \emph{J. Geom. Anal.}, 31\penalty0 (3):\penalty0 2766--2828, 2021.
\newblock \doi{10.1007/s12220-020-00375-5}.

\bibitem[Tudor and Viens(2003)]{tudor2003ito}
C.~Tudor and F.~Viens.
\newblock It{\^o} formula and local time for the fractional brownian sheet.
\newblock \emph{Electron. J. Probab.}, 8:\penalty0 no. 14, 31, 2003.

\bibitem[Tudor and Viens(2006)]{tudor2006ito}
C.~A. Tudor and F.~G. Viens.
\newblock It\^o{} formula for the two-parameter fractional {B}rownian motion
  using the extended divergence operator.
\newblock \emph{Stochastics}, 78\penalty0 (6):\penalty0 443--462, 2006.
\newblock ISSN 1744-2508,1744-2516.
\newblock \doi{10.1080/17442500601014912}.
\newblock URL
  \url{https://doi-org.ezproxy.lib.purdue.edu/10.1080/17442500601014912}.

\bibitem[Zhang et~al.(2022)Zhang, Lin, and Tindel]{ZLT22}
S.~Zhang, G.~Lin, and S.~Tindel.
\newblock Two-dimensional signature of images and texture classification.
\newblock \emph{Proc. A}, 478\penalty0 (2266):\penalty0 Paper No. 20220346, 13,
  2022.
\newblock \doi{10.1098/rspa.2022.0346}.

\end{thebibliography}
\end{document}